\documentclass[12pt]{report}
\usepackage{thesis}
\usepackage{comment}
\usepackage{amsmath,amssymb,amsthm}

\renewcommand{\(}{$\,}
\renewcommand{\)}{\,$}

\newcommand{\cc}[1]{\mathcal{#1}}
\newcommand{\bb}[1]{\boldsymbol{#1}}

\renewcommand{\bar}[1]{\overline{#1}}
\renewcommand{\hat}[1]{\widehat{#1}}
\renewcommand{\tilde}[1]{\widetilde{#1}}
\newcommand{\nn}{\nonumber \\}
\numberwithin{equation}{section}
\numberwithin{figure}{section}
\newcounter{example}[section]
\numberwithin{example}{section}
\newcounter{remark}[section]
\numberwithin{remark}{section}
\newtheorem{theorem}{Theorem}[section]
\newtheorem{proposition}[theorem]{Proposition}
\newtheorem{definition}[theorem]{Definition}
\newtheorem{lemma}[theorem]{Lemma}
\newtheorem{corollary}[theorem]{Corollary}
\newtheorem{exmp}[example]{Example}
\newtheorem{rmrk}[remark]{Remark}
\newenvironment{example}{\begin{exmp}\rm}{\end{exmp}}
\newenvironment{remark}{\begin{rmrk}\rm}{\end{rmrk}}
\newcommand{\argmax}{\operatornamewithlimits{argmax}}

\newcommand{\citeasnoun}[1]{\cite{#1}}

\def\argmin{\operatornamewithlimits{argmin}}

\def\Var{\operatorname{Var}}
\def\eqdef{\stackrel{\operatorname{def}}{=}}
\def\eqdistr{\stackrel{\operatorname{d}}{=}}

\def\E{I\!\!E}
\def\P{I\!\!P}
\def\R{I\!\!R}
\def\T{\top}
\newcommand{\diag}{\operatorname{diag}}
\def\ind{\mathbb{I}}

\newcommand{\bbpf}[1]{\tta^{*}_{#1}}
\newcommand{\mle}[2]{\tilde{\theta}_{#1}^{(#2)}}
\newcommand{\mmle}[1]{\tilde{\tta}_{#1}}
\newcommand{\fle}[2]{\tilde{f}_{#1}{(#2)}}
\newcommand{\flej}[3]{\tilde{f}^{(#1)}_{#2}{(#3)}}

\def\A{\bb{A} }
\def\aadapest{\hat{\tta}}

\def\adaplpest{\hat{f}}
\def\adapind{\hat{k}}
\def\dd{\mathrm{d}}
\def\KK{\cc{K}}
\def\LL{\operatorname{L}}

\def\EE{\mathbb{E}}
\def\PPi{\bb{\Pi}}
\def\RR{\mathbb{R}}
\def\RRn{\mathbb{R}^{n}}
\def\RRp{\mathbb{R}^{p}}
\def\RRd{\mathbb{R}^{d}}
\def\YY{\bb{Y}}
\def\Yi{Y_i}
\def\Xi{X_i}
\def\ff{\bb{f}}
\def\ffi{f(\Xi)}
\def\fta{f_{\bb\theta}}
\def\ftai{f_{\bb\theta}(\Xi)}

\newcommand{\ta}[1]{\theta^{(#1)}}
\newcommand{\W}[1]{\mathbf{W}_{#1}}
\newcommand{\w}[2]{w_{{#1},{#2}}}
\newcommand{\B}[1]{\mathbf{B}_{#1}}
\def\KL{\mathbb{K}\mathbb{L}}
\def\M{\mathbf{M}}
\def\S{\mathbf{S}}
\def\Lam{\mathbf{\Lambda}}
\def\tta{\bb\theta}
\def\DD{\mathbf{D}}
\def\TTa{\mathbf{\Theta}}
\def\SSigma{\mathbf{\Sigma}}
\def\VV{\mathbf{V}}
\def\eps{\varepsilon}
\def\eeps{\bb\varepsilon}

\def\epsi{\varepsilon_{i}}

\newcommand{\norm}[2]{\cc{N} \left( {#1},{#2} \right)}
\def\p{p}

\def\cov{\operatorname{cov}}

\def\MSE{\operatorname{MSE}}
\def\MISE{\operatorname{MISE}}

\def\rank{\operatorname{rank}}
\def\dim{\operatorname{dim}}
\def\tr{\operatorname{tr}}
\def\vec{\operatorname{vec}}
\def\PPsi{\bb{\Psi}}

\def\Psii{\Psi_{i}}
\def\z{\mathfrak{z}}

\def\be#1\ee{\begin{equation}#1\end{equation}}

\newcommand{\bea}{\begin{eqnarray}}
\newcommand{\eea}{\end{eqnarray}}
\newcommand{\beaa}{\begin{eqnarray*}}
\newcommand{\eeaa}{\end{eqnarray*}}
\newcommand{\bei}{\begin{itemize}}
\newcommand{\eei}{\end{itemize}}

\newcommand{\al}{\alpha}
\newcommand{\e}{\varepsilon}
\newcommand{\kk}{K}
\newcommand{\s}{\sigma}

\newcommand{\la}{\lambda}
\newcommand{\La}{\Lambda}
\newcommand{\bk}{\mathbf{k}}

\newcommand{\X}{\mathbb{X}}
\newcommand{\PP}{\mathbb{P}}
\newcommand{\N}{\mathbb{N}}

\newcommand{\I}{\mathbb{I}}

\begin{document}

\pagenumbering{roman}

\thesistitle
	{Adaptive estimation in regression \\
	 and complexity of approximation \\
of random fields}

    הה

\begin{flushright}
    To my mother Svetlana
\end{flushright}

\Huge \textbf{Zusammenfassung}
\vspace{10pt}
\linespread{1.5}
\par \normalsize

Gegenstand der vorliegenden Dissertationsschrift sind spezielle Fragen
der nicht-parametrischen Regression mit Misspezifikation der
Rausch-Kovarianz, und der Informationskomplexit\"at der Approximation
von zuf\"alligen Feldern in Abh\"angigkeit der Dimension.
\par Im ersten Abschnitt untersuchen wir Fragen der nichtparametrischen
Regression unter heteroskedastischem Gau\ss schem Rauschen. Wir nutzen die Methode der lokalen Approximation und Lepskis Methode zur Wahl eines Sch\"atzers aus der Menge der linearen Sch\"atzer, die wir durch verschiedene Grade von Lokalisierung erhalten.
Dieser Zugang wird kombiniert mit den ``Propagation Bedingungen''
bei der Wahl der kritischen Werte der Prozedur, wie dies k\"urzlich von Spokoiny und
Vial \citeasnoun{SV} vorgeschlagen wurde. Die ``Propagation Bedingungen'' des Modells mit misspezifizierter Kovarianzstruktur werden abgeschw\"acht.  Insbesondere im Gau\ss schen
Modell mit unbekanntem Mittelwert und unbekannter Kovarianz nutzen wir eine lokal
lineare parametrische Approximation des Mittels und eine inkorrekt
spezifizierte Kovarianzmatrix.
Wir zeigen, dass dieses Verfahren eine Misspezifikation der Kovarianzmatrix mit einem relativen Fehler bis zu \( o\big( \frac{1}{\log n}  \big)  \) erlaubt, wobei \( n \) der Stichprobenumfang ist. Die Qualit\"at der Absch\"atzung wird im Sinne von nichtasymptotischen Orakel-Risikoschranken gemessen.
\par Im zweiten Abschnitt untersuchen wir die Approximation
\( d \)-parametrischer 
zuf\"{a}l-liger Felder vom Tensorprodukt-Typ durch
Partialsummen der Karhunen-Lo\`eve Entwicklung und beschr\"anken
uns auf den mittleren Fehler.
Gegenstand der Analysis ist die
Informationskomplexit\"at \( n(\eps,d)\), die die minimale Anzahl zu verwendender Koeffizienten
der Reihenentwicklung angibt, die n\"otig ist, den Fehler
\( \eps \) zu garantieren. Seit der Untersuchung von
Lifshits und Tulyakova \citeasnoun{LT} ist bekannt, dass dieses
Problem dem ``Fluch der Dimension'' unterliegt. Wir bestimmen
hier die asymptotisch exakte Darstellung der Informationskomplexit\"at.

\Huge \textbf{Abstract}
\vspace{10pt}
\normalsize
\par In this thesis we study adaptive nonparametric regression with noise
misspecification and the complexity of approximation of random fields in
dependence of the dimension.
\par First, we consider the problem of pointwise estimation in
nonparametric regression with heteroscedastic additive Gaussian noise.
We use the method of local approximation applying the Lepski method for
selecting one estimate from the set of linear estimates obtained by the
different degrees of localization. This approach is combined with the
``propagation conditions'' on the choice of critical values of the
procedure, as suggested recently by
Spokoiny and Vial \citeasnoun{SV}. The ``propagation conditions'' are relaxed for the
model with misspecified covariance structure. Specifically, the model with
unknown mean and variance is approximated by the one with the parametric
assumption of local linearity of the mean function and with an
incorrectly specified covariance matrix.
We show that this procedure
allows a misspecification of the covariance matrix with a relative
error up to \( o\big( \frac{1}{\log n}  \big)  \), where \( n \) is the
sample size. The quality of estimation is measured in terms of
nonasymptotic ``oracle'' risk bounds.
\par We then turn to the \( \eps \)-approximation of \( d \)-parametric
random fields of tensor product-type by means of \( n \)-term partial
sums of the Karhunen-Lo\`eve expansion. The analysis is restricted to
the average case setting.
The quantity of interest is the information
complexity \( n(\eps,d) \) describing the minimal number of terms in the
partial sums, which guarantees an error not exceeding a given level
\( \eps \). The behavior of \( n(\eps,d) \) as \( d \to \infty \) is
the subject of our study. It was shown by Lifshits and
Tulyakova \citeasnoun{LT} that this problem inherits the curse of
dimensionality (intractability) phenomenon. We present the exact asymptotic expression for the information complexity \( n(\eps,d) \).

\include{abstract1}
\Huge \textbf{Acknowledgements}
\vspace{20pt}

\par \normalsize
It is a great pleasure to write this page. I would like to thank my supervisor Vladimir Spokoinyi for introducting me to the challenging world of adaptive methods. I am deeply grateful to my
friends and colleagues: Gilles Blanchard, Rada Dakovi\'c (Mati\'c), Le-Minh Ho, Anastasia and Vladislav Kolodko, Nicole Kr\"amer, Volker Kr\"atschmer, Anna Martius (Levina), John G. M. Schoenmakers and Nataliya Togobytska for their sympathy and help. I am greatly indebted to Alexandre B. Tsybakov for his support, important comments and constructive criticism. Many thanks go to Andre Beinrucker and Peter Math\'e for the careful translation of the abstract into German and for valuable comments. I thank the secretary of our research group Cristine Schneider for her help in thousands of administrative problems, and for being always so friendly
and nice. I wish to thank the Weierstrass Institute for Applied Analysis and Stochastics (WIAS) which made
the completion of this thesis possible. Special thanks are due to the WIAS library and especially to
Ulrike Hintze and Ilka Kleinod.

\par The last chapter of this thesis was partially written while the author was visiting
the Institut f\"{u}r Matematische Stochastik, Georg-August-Universit\"at, G\"ottingen, and was supported by the grants RFBR 05-01-00911 and RFBR-DFG 04-01-04000. I am thankful to the supervisor of this part
of the thesis Mikhail A. Lifshits for the formulation of the problem, and to Manfred Denker for his support and for providing excellent working conditions.

\tableofcontents

\pagenumbering{arabic}
\chapter{Introduction}
\label{chap:intro}

\section{Nonparametric versus parametric methods}
In \emph{nonparametric} estimation the balance between the approximation error (bias) and the
 variance of the estimator, the so-called \emph{bias-variance trade-off}, plays a key role. The bias part depends on the regularity properties of the unobserved signal. Often, for example in image denoising, see \citeasnoun{KatkEA2006} and the references therein, this signal has spatially inhomogeneous smoothness. This prompts the idea to \emph{adapt} statistical methods to the spatially varying smoothness of the function to be recovered from the noisy data.

 \par On the other side, there exists the powerful classical theory of \emph{parametric} estimation, see \citeasnoun{Ibragimov and Khasminskii}, where the underlying data distribution \( \P \) belongs to a parametric family \( \cc{P} = (\P_{\tta}, \tta \in \Theta) \) described by a \emph{finite-dimensional} parameter \( \tta \in \Theta \subset \RRp \). Obviously, the assumption that the parametric model holds \emph{globally}, i.e., that there exists a parameter \( \tta_0 \in \Theta \) such that \( \P = \P_{\tta_0} \), is too restrictive. It is hopeless to believe that the real data indeed follow some parametric model or even can be well approximated by it globally.

 \par One way out of this situation is to increase the number of parameters of the model, increasing the dimension of the parameter set \( \Theta \). This increases dramatically the complexity of the model and may, especially for high-dimensional data, make the problem computationally unfeasible. See Chapter \ref{chap:Approx} for an example of a such problem. One can also approximate a high- or infinite-dimensional parameter set \( \Theta \) by a dense sequence of low-dimensional subsets \emph{``sieves''} \( \{ \Theta_p \} \), \( p= 1,2, \ldots \, \). See \citeasnoun{vanderVaartWellner} for details. The simplest example of sieves is given by projection estimators, when the signal \( f \) is considered as a series expansion with respect to some functional basis. One tries to approximate \( f \) by the finite sums of this expansion, that is by its projection on the linear span of the first \( N \) basis functions, see \citeasnoun{Tsybakov}. The crucial problem is to decide how large \( N \) should be in order to provide a satisfactory level of approximation error. Chapter \ref{chap:Approx} of this thesis addresses to the problem of approximation of random fields of specific ``tensor-product'' type by the finite sums of the Karhunen-Lo\`eve expansion.

 \par Another idea to make a parametric model more flexible is to fix a small number of parameters, that is, the dimension \( p \) of the set \( \Theta \), but to reduce the amount of the data. This leads to the \emph{local parametric approach} dating back to the book by Katkovnik \citeasnoun{Katk1985} and papers \citeasnoun{Katk1979}, \citeasnoun{Katk1983}, where he suggested the \emph{method of local approximation}. This approach was further developed with application to image denoising, see \citeasnoun{KatkEA2006}, \citeasnoun{Foi} and the references therein. For local polynomial fitting see \citeasnoun{Fan and Gijbels book}. An interesting development in the direction of local-likelihood estimation, closely connected with the ideas of \citeasnoun{Akaike} and~\citeasnoun{White}, is due to Loader \citeasnoun{Loader}, Polzehl and Spokoiny \citeasnoun{Polzehl and Spokoiny}, Belomestny and Spokoiny~\citeasnoun{Belomestny and Spokoiny}. A fruitful application of this approach is to change-point detection in time series, see \citeasnoun{Spokoiny discontinious}, \citeasnoun{Spokoiny change point detection}, \citeasnoun{Cizek}.

 \par In order to compare the method of local approximation with the projection estimation described above, let us consider the following example. Fix a reference point \( x \in \R \). By the Taylor theorem any function which is \( p \) times differentiable on the closed interval \( [x-h, x+h] \) and \( p +1\) times differentiable on the open interval \( (x-h, x+h) \) can be expanded with respect to the polynomial basis
 \( f(t) \approx f_p(t)= f(x) + f'(x) (t-x) + \cdots + f^{(p)} (t-x)^{p}/p! \) for any \( t \in (x-h, x+h) \). Here \( N=p \) is fixed; we aim to choose the width of the interval \( h \) by the data. If the \emph{bandwidth} \( h \) is sufficiently small, the class of such functions is large and \( f_p(t) \) can serve as a reasonable estimator of the value of the unknown signal \( f(t) \) for \( t \) close to~\( x \). This idea leads to the method of local approximation, see Section \ref{sect:LocPol} for details. Due to the dependence on \( x \) this approach is nonparametric or local parametric.

 \par The most important problem is the detection of the width \( h \) of the interval providing a satisfactory quality of approximation. If the bandwidth is chosen too large it will result in a large
approximation error (bias). Small \( h \) will improve the bias, but because the number of data points falling in this interval will also be small, the variance of the estimator will be large. In the projection estimation framework the number of basis functions \( N \) plays a similar role.
The larger \( N \) is
the smaller is the modeling bias, and the larger is the variance. Thus we come back to the trade-off between bias and variance, that is to the problem of the choice of a ``good'' bandwidth.

\par If the function \( f \) would be known or its smoothness would be given, then the bandwidth \( h \) would be easy to select. Unfortunately, in most real life problems no information about the regularity properties of the underlying signal is available. Thus we need to construct a \emph{data-driven} method which would \emph{adapt} itself \emph{automatically} to the properties of the function \( f \) and, particularly, to its probably spatially inhomogeneous smoothness. One way of doing this is, instead of considering the single bandwidth \( h \), to take a finite grid (usually of geometric type) of bandwidths \( \{ h_k \}_{k=1}^K \) producing a growing sequence of \emph{nested neighborhoods} of the reference point \( x \). This pointwise-adaptive bandwidth (scale, localizing scheme) selection is based on the idea known as Lepski's method. This approach was proposed in a series of papers \citeasnoun{Lep1990}, \citeasnoun{Lep1991}, \citeasnoun{Lep1992}.
The idea is as follows: suppose that a point \( x \) and some method of localization (a smoothing kernel) are fixed. One calculates a sequence of estimators corresponding to different scales, and the procedure searches for the largest local vicinity of the center of approximation \( x \), that is for the largest bandwidth, for which the corresponding estimator is not rejected by the data. The calculated estimators are compared by the algorithm, and the adaptively selected bandwidth is the largest one such that the corresponding estimator does not differ significantly from the estimators with smaller bandwidths. Among other applications, this idea was further applied by Katkovnik as the intersection of the confidence intervals (ICI) rule (see \citeasnoun{KatkEA2006}), by Spokoiny as the fitted log-likelihood (FLL) technique (see \citeasnoun{KatkSpok}) and as a two sample likelihood ratio test with application to change-point detection (see \citeasnoun{Spokoiny discontinious}, \citeasnoun{Spokoiny change point detection}). The interesting recent paper by Rei\ss, Rozenholc, and Cuenod \citeasnoun{Reiss} presents a Lepski-type method based on the Wald-test statistics for robust and quantile regression estimation.

\par It is well known from approximation theory that the smoothness of a function can be expressed via the quality of its approximation by a sufficiently regular kernel smoother (see \citeasnoun{Triebel}). The Taylor theorem can be considered from this point of view as well. Let the degree \( p \) of the Taylor polynomial be fixed. Then the quality of approximation of a function \( f \) by the finite sum of the Taylor expansion and the width \( h \) of the proper vicinity of approximation also express the smoothness of \( f \). Thus the procedure described above intrinsically adapts directly to the local smoothness properties of the unknown function \( f \). One can also select simultaneously a kernel and a bandwidth, see the second part of \citeasnoun{LepSpok97}.

\par Since the seminal paper by Lepski \cite{Lep1990} dating back to 1990, the local pointwise adaptive methods based on Lepski's approach have showed their power being applied to image denoising \citeasnoun{Polzehl and Spokoiny}, \citeasnoun{Foi}, \citeasnoun{KatkEA2006}, robust and quantile regression \citeasnoun{Reiss}, change-point detection and volatility estimation in time series \citeasnoun{Spokoiny discontinious}, \citeasnoun{Mercurio and Spokoiny}, \citeasnoun{Spokoiny change point detection}, \citeasnoun{Cizek}, density estimation \citeasnoun{Butucea} and inverse problems, see \citeasnoun{Mathe} and the references therein. This list is not complete and just shows the possible spectrum of application. A new technique originating from \cite{Lep1990} for spatially adaptive local
constant approximation  employing local-likelihood methods was suggested in \cite{KatkSpok}. This approach is based on the assumption that a regression function can be well
approximated by a constant in a vicinity of a given point. The suggested test statistics \( T_{lk} \), \( 1 \le l <k \le K \) are based on the fitted local-likelihood (FLL), that is on the difference between the value of the local log-likelihood corresponding to the smaller scale at the point of its maximum and the maximum of the local log-likelihood corresponding to the larger scale. These statistics are used for data-driven detection of the size and shape of the homogeneity area. Lepski's selection rule from \cite{Lep1990}, see also \cite{LepMamSpok97}, is applied to the FLL-statistics, whereby chooses an adaptive scale (bandwidth \( h_{\adapind} \)) as the largest for which the values of \( T_{lm} \) are sufficiently small:
\begin{equation}\label{adaptind}
    \adapind = \max \left\{k \leq K :   T_{lm}  \leq \z_l, \,l < m \leq k    \right\}.
 \end{equation}
  The crucial problem for such adaptive methods is the choice of critical values \( \z_1, \ldots ,\z_{K-1} \). A ``propagation approach'' for choosing the parameters in the selection rule \eqref{adaptind} is advocated in \cite{KatkSpok} and \cite{SV}. The idea is to select the critical values to provide the prescribed behavior of the procedure in the simplest parametric situation. Then the procedure should work well even when the parametric assumption is violated.

\par In \cite{KatkSpok} and \cite{SV} the local constant fit is considered. In Chapter \ref{chap:Regression} we generalize the FLL method  to the local linear approximation in regression with heteroscedastic Gaussian noise, and the ``propagation approach'' is justified for the case of misspecified covariance structure.

\section{Local approximation}
\label{sect:LocPol}

\subsection{Local polynomial estimators: basic properties}
Let us consider as a motivation for local polynomial fitting the case of a deterministic design in \( \RR \). By the Taylor theorem any function \( f(\cdot) \) in a H\"older class
\( \Sigma(\beta, L) \), \( \beta >1 \) can be represented, up to a reminder term, as
\( f(t) \approx f(x) + f'(x) (t-x) + \cdots + f^{(p-1)} (t-x)^{\p-1}/(p-1)! \) for \( t \)
sufficiently close to \( x \) and \( p-1 = \lfloor \beta \rfloor  \). This suggests the use of a local polynomial approximation to \( f(t) \) in the form \( \fta(t) = \Psi(t-x)^{\T} \tta \) with
\( \Psi(u) = (1,u ,\ldots,(u)^{\p-1}/(p-1)! )^{\T} \) and the vector of parameters
\( \tta = \tta(x) =(\ta{0},\,\ta{1}, \ldots , \ta{\p-1})^{\T} \) with \( \ta{j}(x) = f^{(j)}(x) \) to be estimated. The main intrinsic issue is to detect an optimal ``vicinity'' of the point \( x \) in order to avoid over- or undersmoothing.

Consider a regression model
\begin{equation*}
     \Yi =\ffi + \sigma \,\epsi , \;\;\;\; i=1, \ldots , n
\end{equation*}
where \( \epsi \) are independent zero mean random variables with \( \EE \epsi^2 =1 \). Given a point \( x \in \RR \), we aim to recover the value \( f(x) \) from the noisy data. Let \( \YY \) be an \( n \)-dimensional vector of observations such that \( \YY = (Y_1, Y_2, \ldots, Y_n)^{\T} \). Denote for any \( i=1, \ldots , n \) by \( \Psii \) the vector of values of the polynomial basis functions at the design points centered at the reference point \( x \):
\begin{equation*}
     \Psii= \Psi(\Xi-x) \eqdef \left(1,\, \Xi -x, \ldots, (\Xi -x)^{\p-1}/(p-1)!\right)^{\T}
\end{equation*}
and by \( \PPsi \) the \( p \times n \) matrix with columns \( \Psii \).
Let \( W(u) \) be a nonnegative localizing function (smoothing kernel) having its maximum at zero and being finite or vanishing at infinity: \( W(u) \to 0 \) as \( |u| \to \infty \).
To shorten the notation denote also by \( \w{h}{i}(x) \eqdef W \big( \frac{\Xi -x}{h} \big) \).
 The localizing scheme corresponding to a bandwidth \( h >0\) then
can be represented as a diagonal matrix of the form:
\begin{equation*}
      \W{h}(x) \eqdef  \diag \{  \w{h}{1}(x), \ldots,\w{h}{n}(x) \}.
  \end{equation*}
The following definition of local polynomial estimators is based on the ones from \citeasnoun{Tsybakov} page 35 and \citeasnoun{Katk1985} pages 28--29.
\begin{definition}\label{loc pol est}
A vector \( \mmle{h}(x) \in \RRp \) defined as a minimizer of the weighted sum of squares
\begin{eqnarray}\label{def loc polynom est}
  \mmle{h}(x) &=&  \argmin_{\tta \in \RRp} \| \W{h}(x)^{1/2} \left( \YY -\PPsi^{\T}\tta\right)\|^2 \nn
  &=& \argmin_{\tta \in \RRp} \sum_{i=1}^{n} |\Yi -\Psii^{\T} \tta|^2  \w{h}{i}(x)
\end{eqnarray}
is called a local polynomial estimator of order \( p-1 \) of \( \tta(x) \). The statistic
\begin{equation*}
    \fle{h}{x} = \bb e^{\T}_1 \mmle{h}(x) = \Psi(0)^{\T} \mmle{h}(x)
\end{equation*}
is called a local polynomial estimator of order \( p-1 \) of \( f(x) \).
Here \( \bb e_1 \in \RRp \) is the first canonical basis vector.
\end{definition}
We will refer to the local polynomial estimators of order \( p-1 \) of \( \tta(x) \)  and of \( f(x) \) as the \(LP (p-1) \) estimator of \( \tta(x) \) or of \( f(x) \) respectively.
It is easy to see that for the properly normalized basis functions the \(LP (p-1) \) estimator of \( \tta(x) \) provides estimators of all derivatives of the function \( f \) of order less or equal \( p-1 \):
\begin{equation*}
    \flej{j}{h}{x} = \bb e^{\T}_{j+1} \mmle{h}(x)\;,\;\; j = 1,\ldots,p-1
\end{equation*}
with the \( j \)th canonical basis vector \( \bb e_j \in \RRp \)

\par The \(LP (p-1) \) estimator \( \mmle{h}(x) \) satisfies the \emph{normal equations}
\begin{equation}\label{normal equations}
    \mathbf{B}(x)  \mmle{h}(x) = \PPsi \W{h}(x) \YY
\end{equation}
where the symmetric \( p \times p \) matrix \( \mathbf{B}(x) \) is given by
\begin{equation}\label{B(x)}
     \mathbf{B}(x) \eqdef   \PPsi  \W{h}(x)\PPsi^{\T} =
        \sum_{i=1}^{n} \Psii \Psii^{\T} \w{h}{i}(x).
\end{equation}
If the matrix \( \mathbf{B}(x) \) is positive definite (\( \mathbf{B}(x) \succ 0 \)), the \( LP(p-1) \) estimator is the unique solution of \eqref{normal equations} and is given by the following formula:
\begin{equation}\label{formula for loc polyn est}
   \mmle{h}(x) = \mathbf{B}(x)^{-1} \PPsi \W{h}(x) \YY
                = \mathbf{B}(x)^{-1} \sum_{i=1}^{n}  \Psii \Yi \w{h}{i}(x).
\end{equation}
In this case the \(LP (p-1) \) estimator \( \fle{h}{x} \) is a \emph{linear estimator} of \( f(x) \):
\begin{equation}\label{linearity of LP estim.}
    \fle{h}{x} = \sum_{i=1}^{n} \Yi W^*_i(x)
\end{equation}
where the weights \( W^*_i(x) \) are given by:
\begin{equation}\label{polynomial weights}
    W^*_i(x) = \bb e^{\T}_1 \mathbf{B}(x)^{-1} \Psii  \w{h}{i}(x).
\end{equation}
\vspace{15pt}
\par Recall the important reproducing polynomials property of the local polynomial estimator (see \citeasnoun{Tsybakov} page 36), and for a more general representation \citeasnoun{Katk1985} page 85.
\begin{proposition}\label{Henderson th}
Let \( x \in \RR \) be such that \( \mathbf{B}(x) \succ 0 \) and let \( P_{p-1} \) be a polynomial of degree less or equal to \( p-1 \). Then the weights defined by \eqref{polynomial weights} satisfy
\begin{equation*}
    \sum_{i=1}^{n} P_{p-1}(\Xi) W^*_i(x) = P_{p-1}(x)
\end{equation*}
for any design points \( \{ X_1, \ldots, X_n \} \). Particularly,
\begin{eqnarray}\label{normalization of polynomial weights}
   & & \sum_{i=1}^{n} W^*_i(x) = 1,  \\ \nonumber
   & & \sum_{i=1}^{n} (\Xi - x)^m W^*_i(x) = 0 \; , \;\; m = 1, \ldots, p-1.
\end{eqnarray}
\end{proposition}
\begin{proof}
By the Taylor expansion
\begin{equation*}
    P_{p-1}(\Xi) = \sum_{m=1}^{p} \frac{P^{(m-1)}_{p-1}(x)}{(m-1)!} (\Xi-x)^{(m-1)} = \Psii^{\T} \tta
\end{equation*}
with \( 0! \eqdef 1 \) and
\( \tta(x) \eqdef (P_{p-1}(x), P'_{p-1}(x), \ldots, P^{(p-1)}_{p-1}(x))^{\T} \). Then by \eqref{polynomial weights} and \eqref{B(x)}
\begin{eqnarray*}
  && \sum_{i=1}^{n} P_{p-1}(\Xi) W^*_i(x)
    = \bb e_1^{\T} \mathbf{B}(x)^{-1} \sum_{i=1}^{n} \Psii \Psii^{\T} \w{h}{i}(x) \tta(x) \\
  &&= \bb e_1^{\T} \tta(x)= P_{p-1}(x).
\end{eqnarray*}
\end{proof}

\subsection{Mean squared error of local polynomial estimators}
\label{Nonadaptive rate}
In this section we show a classical method for obtaining upper bounds for the quadratic risk of the \( LR(p-1) \) estimator under the assumption that the underlying function~\( f \) belongs to a H\"older class \( \Sigma(\beta, L) \) with \( p-1 = \lfloor \beta \rfloor \). This analysis will be done via the traditional bias-variance trade-off. Later on  in Section \ref{section: adaptive rates} it will be shown how this approach can be adjusted for the purpose of pointwise adaptation. In what follows we assume a deterministic design with \( \Xi \in [0,1] \). Fix a point \( x \in \RR \) and the method of localization \( \W{h}(x) \). By \eqref{formula for loc polyn est} the local polynomial estimator \( \mmle{h}(x) \) can be easily decomposed into deterministic and stochastic parts:
\begin{equation*}
    \mmle{h}(x) = \bbpf{h}(x) + \bb \zeta_h(x),
\end{equation*}
where
\begin{eqnarray*}
  \bbpf{h}(x) &=& \mathbf{B}(x)^{-1} \sum_{i=1}^{n}  \Psii  \w{h}{i}(x) \ffi, \\
 \bb \zeta_h(x) &=& \s \, \mathbf{B}(x)^{-1} \sum_{i=1}^{n}  \Psii  \w{h}{i}(x) \epsi.
\end{eqnarray*}
Then
\begin{equation}\label{decomposition of LP estim into det and stoch}
    \fle{h}{x} = \bb e_1^{\T} \mmle{h}(x) = \bb e_1^{\T} \bbpf{h}(x) +  \bb e_1^{\T} \bb \zeta_h(x)
\end{equation}
with
\begin{eqnarray*}
  \bb e_1^{\T} \bbpf{h}(x) &=&  \sum_{i=1}^{n} W^*_i(x) \ffi, \\
  \bb e_1^{\T} \bb \zeta_h(x) &=&  \s \sum_{i=1}^{n} W^*_i(x) \epsi.
\end{eqnarray*}
Denote the variance of the stochastic part \( \bb e_1^{\T} \bb \zeta_h(x) \) of \eqref{decomposition of LP estim into det and stoch} by
  \begin{eqnarray}\label{bound for variance}
    \s^2_h(x)   &\eqdef& \Var_f[\fle{h}{x}] \nn
                &=& \bb e_1^{\T} \EE[\bb \zeta_h(x)\bb \zeta_h(x)^{\T}] \bb e_1 \nn
                &=& \s^2 \sum_{i=1}^{n} (W^*_i(x))^2.
 \end{eqnarray}
Define the bias (the approximation error)
\begin{equation*}
    b_h(x) \eqdef \EE_f[\fle{h}{x} -f(x)]  = \bb e_1^{\T} \bbpf{h}(x) -f(x).
\end{equation*}
Then the \emph{bias-variance decomposition} for the mean squared error at \( x \) is given by
\begin{equation}\label{bias-variance decomposition}
   \MSE(x)  \eqdef \EE_f[|\fle{h}{x} -f(x)|^2] =b^2_h(x) + \s^2_h(x).
\end{equation}
Using Proposition \ref{Henderson th} the bias can be written as follows:
  \begin{equation}\label{formula for bias}
     b_h(x) = \sum_{i=1}^{n} (\ffi - f(x)) W^*_i(x).
 \end{equation}
 Following the line of presentation from \citeasnoun{Tsybakov}, we impose the following assumptions on the localizing schemes and the design.
 \begin{description}
\item[\( \bb{\mathfrak{(Lp1)}} \)]
\emph{There exists a number \( \lambda_0>0  \) such that uniformly in \( x \) the smallest eigenvalue fulfills \( \lambda_p(\mathbf{B}(x)) \ge  nh\lambda_0 \) for all sufficiently large \( n \).
\item[\( \bb{\mathfrak{(Lp2)}} \)]
There exists a real number \( a_0>0 \) such that for any interval \( A \subseteq [0,1] \) and all \( n \ge 1 \)}
\begin{equation*}
    \frac{1}{n} \sum_{i=1}^n \ind \{ \Xi \in A \} \le a_0 \max \big\{ \int_A\dd t, \frac{1}{n} \big\}.
\end{equation*}
\item[\( \bb{\mathfrak{(Lp3)}} \)]
\emph{ The localizing functions (kernels) \( \w{h}{i} \) are compactly supported in \( [0,1] \) with}
\begin{equation*}
    \w{h}{i}(x) = 0 \;\;\;\text{if} \;\;\; |\Xi - x|> h.
\end{equation*}
\noindent This immediately implies a similar property for the local polynomial weights:
\begin{equation*}
    W^*_i(x) = 0 \;\;\;\text{if} \;\;\; |\Xi - x|> h.
\end{equation*}

\item[\( \bb{\mathfrak{(Lp4)}} \)]
\emph{ There exists a finite number \( w_{max} \) such that}
\begin{equation*}
    \sup_{i,x}|\w{h}{i}(x) | \le w_{max}.
\end{equation*}
\end{description}
\begin{lemma}\label{bounds for loc pol weights}
Assume \( \mathfrak{(Lp1)} - \mathfrak{(Lp4)} \). Then for \( n\) sufficiently large and all \( h \ge  \frac{1}{2n} \) and \( x \in[0,1] \) the local polynomial weights \( W^*_i(x) \) are such that:
\begin{eqnarray*}
  \sup_{i,x} |W^*_i(x)| &\le& \frac{C_1}{nh}, \\
  \sum_{i=1}^n|W^*_i(x)|&\le& C_2
\end{eqnarray*}
with \( C_1 =  w_{max} \sqrt{e}/\lambda_0\) and \( C_2 = 2 w_{max} a_0 \sqrt{e}/\lambda_0\).
\end{lemma}
\begin{proof}
Recall that \( \mathbf{B}(x) \) is a symmetric non-degenerate \( p \times p \) matrix. Then by the Schur theorem there exist an orthogonal matrix \( U \) and a diagonal matrix \( \Lambda = \diag\{ \lambda_1^{-2}(\mathbf{B}(x)), \ldots, \lambda_p^{-2}(\mathbf{B}(x)) \} \) such that \( \mathbf{B}(x)^{-2} = U^{\T} \Lambda U  \). Then by Assumption~\( \mathfrak{(Lp1)} \) for any \( \gamma \in \RRp \)
\begin{equation*}
    \gamma^{\T} \mathbf{B}(x)^{-2} \gamma   = \gamma^{\T} U^{\T} \Lambda  U\gamma \le (nh\lambda_0)^{-2} \| \gamma \|^2 ,
\end{equation*}
implying
\begin{equation*}
    \| \mathbf{B}(x)^{-1} \gamma \| \le (nh\lambda_0)^{-1} \| \gamma \|.
\end{equation*}
\noindent By \eqref{polynomial weights}, Assumptions \( \mathfrak{(Lp3)} \) and \( \mathfrak{(Lp4)} \) and using that \( h<1 \), we have
\begin{eqnarray*}
  |W^*_i(x)| &=& |\bb e_1^{\T} \mathbf{B}(x)^{-1} \Psii \w{h}{i}(x)| \\
   &\le & w_{max}\| \mathbf{B}(x)^{-1} \Psii \| \le \frac{w_{max}}{\lambda_0 nh} \| \Psii \| \\
   &\le & \frac{w_{max}}{\lambda_0 nh} \big(1 + h^2 + \frac{h^4}{(2!)^2} + \cdots + \frac{h^{2(p-1)}}{((p-1)!)^2}  \big )^{1/2}\\
   &\le & \frac{w_{max}}{\lambda_0 nh} \big( 1+1+\frac{1}{2!} + \cdots + \frac{1}{(p-1)!}     \big)^{1/2}\\
   &< & \frac{w_{max} \sqrt{e}}{\lambda_0 nh} ,
\end{eqnarray*}
where the upper bound \( w_{max} \sqrt{e} (\lambda_0 nh)^{-1} \) does not depend on \( i \) and \( n \).
\par The second assertion of the lemma is obtained similarly. Condition \( \mathfrak{(Lp2)} \) implies
\begin{eqnarray*}
  \sum_{i=1}^n|W^*_i(x)| &\le& \frac{w_{max}}{\lambda_0 nh} \sum_{i=1}^n \| \Psii \| \,\ind\{ \Xi \in [x-h,x+h] \} \\
  &\le & \frac{w_{max} \sqrt{e}}{\lambda_0 } a_0 \max\{ 2, \frac{1}{nh} \}\\
  &\le & \frac{2 w_{max} \sqrt{e} a_0}{\lambda_0 }
\end{eqnarray*}
for all \( h \ge \frac{1}{2n} \).
\end{proof}
\begin{theorem}\label{theorem upper bound for MSE}
Let \( f \in \Sigma (\beta, L) \) on \( [0,1] \) and let \( \fle{h}{x} \) be the \( LP(p-1) \) estimator of \( f(x) \) with \( p -1 = \lfloor \beta \rfloor\). Then under the conditions of Lemma \ref{bounds for loc pol weights}  for \( n\) sufficiently large and all \( h \ge  \frac{1}{2n} \) and \( x \in[0,1] \),
\begin{eqnarray*}
  |b_h(x)|    &\le & C_2\frac{Lh^{\beta}}{(p-1)!} , \\
  \s^2_h(x)   &\le & \frac{\s^2 C_1 C_2}{nh}
\end{eqnarray*}
with \( C_1 \) and \( C_2 \) as in Lemma \ref{bounds for loc pol weights}.
\par Moreover, the choice of positive bandwidth \( h = h^\star(n) \) given by \eqref{optimal bandwidth classical} such that
\begin{equation*}
    h^\star(n) = \cc O \big(n^{ -\frac{1}{2 \beta +1} }\big)
\end{equation*}
provides the following upper bound for the quadratic risk:
\begin{equation}\label{upper bound for MSE classical}
    \varlimsup_{n \to \infty} \sup_{f \in \Sigma (\beta, L)} \sup_{x \in[0,1]}
        \EE_f[ \psi_n^{-2} |\fle{h}{x} -f(x)|^2] \le  C ,
\end{equation}
where
\begin{equation}\label{nonadaptive rate Obig}
    \psi_n = \cc O \bigg(n^{-\frac{\beta}{2 \beta +1} }\bigg)
\end{equation}
is given by \eqref{nonadaptive rate}
and the constant \( C \) is finite and depends on \( \beta \), \( L \), \( \s^2 \), \( p \), \( w_{max} \) and \( a_0 \) only.
\end{theorem}
\begin{corollary}
Under the conditions of Theorem \ref{theorem upper bound for MSE} we have the same rate for the \( \MISE \) (mean integrated square error):
\begin{equation}\label{upper bound for MISE classical}
    \varlimsup_{n \to \infty} \sup_{f \in \Sigma (\beta, L)}
        \EE_f[ \psi_n^{-2} \int_0^1|\fle{h}{x} -f(x)|^2\dd x] \le  C
\end{equation}
with the rate \( \psi_n \) given by \eqref{nonadaptive rate Obig} and the finite constant \( C \) depending on \( \beta \), \( L \), \( \s^2 \), \( p \), \( w_{max} \) and \( a_0 \) only.
\end{corollary}
\begin{proof}
By \eqref{formula for bias} and the Taylor theorem with \( \tau_i \) such that the points \( \tau_i \Xi \) are between \( \Xi \) and \( x \), we have
\begin{eqnarray*}
  b_h(x) &=& \sum_{i=1}^{n} (\ffi - f(x)) W^*_i(x) \\
    &=& \sum_{j=1}^{p-2} \frac{f^{(j)}(x)}{j!}   \sum_{i=1}^{n}  (\Xi -x)^j W^*_i(x)
        +   \sum_{i=1}^{n} \frac{f^{(p-1)}(\tau_i \Xi)}{(p-1)!}  (\Xi -x)^{p-1} W^*_i(x).
\end{eqnarray*}
The first summand is equal to zero by Proposition \ref{Henderson th}. By the same argumentation the second term can be rewritten as follows:
\begin{eqnarray*}
   b_h(x) &=& \sum_{i=1}^{n} \frac{f^{(p-1)}(\tau_i \Xi)}{(p-1)!}  (\Xi -x)^{p-1} W^*_i(x) \\
          &=& \frac{1}{(p-1)!} \sum_{i=1}^{n} \big( f^{(p-1)}(\tau_i \Xi) -f^{(p-1)}(x)  \big)
                            (\Xi -x)^{p-1} W^*_i(x).
\end{eqnarray*}
Then by Lemma \ref{bounds for loc pol weights}
\begin{eqnarray*}
   |b_h(x)|   &\le & \frac{L}{(p-1)!}
                        \sum_{i=1}^{n}  |\tau_i \Xi -x|^{\beta -(p-1)} |\Xi -x|^{p-1}|W^*_i(x)|\\
            &\le & \frac{L}{(p-1)!}
                        \sum_{i=1}^{n}  |\Xi -x|^{\beta} |W^*_i(x)| \ind\{ |\Xi -x| \le h \}\\
            &\le & C_2\frac{Lh^{\beta}}{(p-1)!}.
\end{eqnarray*}
By formula \eqref{bound for variance} and Lemma \ref{bounds for loc pol weights} the variance is bounded by
\begin{eqnarray*}
  \s^2_h(x) &\le & \s^2  \sup_{i,x} |W^*_i(x)| \sum_{i=1}^{n} |W^*_i(x)| \\
   &\le & \frac{\s^2 C_1 C_2}{nh}.
\end{eqnarray*}
Then by \eqref{bias-variance decomposition}
\begin{equation}\label{balance equation classical}
    \MSE(x) \le \tilde{C}_2 h^{2 \beta} + \frac{\tilde{C}_1}{nh}
\end{equation}
with \( \tilde{C}_1 = \s^2 C_1 C_2 \) and \( \tilde{C}_2 = C_2^2 L^2  ((p-1)!)^{-2} \). Then the optimal bandwidth \( h^\star(n) \) minimizing the upper bound for the \( \MSE \) at \( x \) is given by
\begin{eqnarray}\label{optimal bandwidth classical}
    h^\star(n) &=& \bigg( \frac{\tilde{C}_1}{2 \beta \tilde{C}_2 } \bigg)^{\frac{1}{2 \beta +1} }
                            n^{ -\frac{1}{2 \beta +1} }\nn
                &=&  \bigg( \frac{\s^2((p-1)!)^2 }{4 a_0 \beta L^2 } \bigg)^{\frac{1}{2 \beta +1} } n^{ -\frac{1}{2 \beta +1} }.
\end{eqnarray}
This gives us the rate \( \psi_n \) w.r.t. the squared loss function over a H\"older class \( \Sigma (\beta, L) \):
\begin{eqnarray}\label{nonadaptive rate}
  \psi_n &=& \tilde{C} \bigg(  \frac{L}{(p-1)!} \bigg)^{\frac{1}{2 \beta +1} }
                \bigg(  \frac{\s^2}{n} \bigg)^{\frac{\beta}{2 \beta +1} } \nn
  &=& \cc O \bigg(n^{-\frac{\beta}{2 \beta +1} }\bigg)
\end{eqnarray}
with \( \tilde{C} = 2^{\frac{1}{2 \beta +1} } w_{max} \sqrt{e} \lambda_0^{-1} a_0^{\frac{\beta+1}{2 \beta +1} }  \beta^{-\frac{\beta}{2 \beta +1} }\).
\end{proof}

\subsection{Method of local approximation: general set-up}
\label{subsec:Method of local approximation}

In this section, following up to the notation the book \citeasnoun{Katk1985} and the papers \citeasnoun{Katk1979}, \citeasnoun{Katk1983} we will explain the basic idea of the \emph{method of local approximation} in a more general set-up than in the previous section. Consider for simplicity the regression model
\begin{equation*}
    \Yi = \ffi + \epsi \;, \; \; i = 1, \ldots, n.
\end{equation*}
If we want to recover \( f(x) \) at the point \( x \), we put the \emph{center of localization} at \( x \). Suppose that some basis \( \{ \psi_j(\cdot) \} \) is chosen. Denote by \( \Psi(u) = (\psi_1(u), \ldots, \psi_p(u))^{\T} \) a vector of the basis function.
We believe that for \( t \) close to \( x \) the values \( f(t) \) can be well approximated by the finite sum
\begin{equation}\label{approximation sum from MLA}
    \fta(t) \eqdef \Psi(t-x)^{\T} \tta(x) = \sum_{j=1}^p \theta^{(j)}(x) \psi_j(t-x)
\end{equation}
where \( \Psi(t-x) \) is the vector of values of the basis functions centered at \( x \). Thus, to estimate \( f(x) \), we have to estimate the vector of coefficients
\( \tta(x) = (\theta^{(1)} (x) , \ldots , \theta^{(p)} (x) )^{\T} \).

\par Let \( W(u) \) be a nonnegative localizing function (smoothing kernel) having maximum at zero and being finite or vanishing at infinity: \( W(u) \to 0 \) as \( \| u \| \to \infty \).
Denote also \( \w{h}{i}(x) \eqdef W \big( \frac{\Xi -x}{h} \big) \).
Let \( F : \RR \to \RR_{\ge 0} \)
be a convex \emph{loss function}. Then the solution (solutions) of the following minimization problem
\begin{equation}\label{MLA est}
  \mmle{h}(x)    = \argmin_{\tta \in \RRp} \sum_{i=1}^n F( \Yi - \Psii^{\T} \tta)\w{h}{i}(x)
\end{equation}
with \( \Psii \eqdef \Psi(\Xi -x) \), \( i = 1, \ldots, n \)
is the estimator of the vector \( \tta \) at the point \( x \) obtained by the method of local approximation.
Notice that \( \mmle{h}(x) \) is an \emph{M-estimator}, see \citeasnoun{Huber} or \citeasnoun{vanderVaartWellner}. The estimator
 \begin{equation}\label{MLA est of f}
     \fle{h}{x} \eqdef \Psi(0)^{\T} \mmle{h}(x) = \sum_{j=1}^p \tilde{\theta}_h^{(j)}(x) \psi_j(0)
 \end{equation}
is an estimator of the function \( f \) at the point \( x \) by the method of local approximation.
In the case of the polynomial basis \( \langle 1,u,u^2,  \ldots \rangle \) we have \( \Psi(0) = (1, 0, \ldots, 0)^{\T} \) and \( \fle{h}{x} \) is just the first coordinate of \( \mmle{h}(x) \).

\par It was stressed in \citeasnoun{Katk1985} (see page 29) that the optimal choice of the parameter of locality (bandwidth \( h \)) is one of the most important issues of the nonparametric estimation. Katkovnik \citeasnoun{Katk1985}, see page 16, pointed out that the practical use of the estimators obtained by the method of local approximation, as well as of any estimators, requires to construct them \emph{adaptively}, that is with a tuning of the parameters in accordance with the data in hand. This leads essentially to the traditional problem of testing the hypothesis about the model. The necessity of data-driven treatment motivates the application of the Lepski-type procedure to the selection of the scale (of the bandwidth \( h_{\adapind} \)) and the ``propagation conditions'' approach on the choice of the critical values of the adaptive procedure (see Section \ref{section: adaptive procedure}) suggested in \citeasnoun{KatkSpok} and in \citeasnoun{SV} and developed in the present work.

\par The asymptotic properties of the estimators given by \eqref{MLA est} and \eqref{MLA est of f} were precisely studied in \citeasnoun{Tsybakov82a}, \citeasnoun{Tsybakov82b} and \citeasnoun{Tsybakov86}. In \citeasnoun{Tsybakov86} it was shown that the estimators, constructed by \eqref{MLA est} w.r.t. the convex loss function and the polynomial basis \( \langle 1, \ldots, u^p \rangle \) exhibit the best rate of convergence among all estimators of functions over H\"older classes \( \Sigma(p-1, L) \) on some bounded subset of \( \RR \), as well as among all estimators of their derivatives. The use of a  non-quadratic loss function \( F(\cdot) \) is very important in the theory of robust estimation and allows to treat the noise with unbounded variance, see for instance the classical paper of Huber \citeasnoun{Huber}.

\par If the basis \( \{ \psi_j(\cdot) \} \) is an orthonormal basis in \( L_2(\cc X) \) for some compact \( \cc X \subset \RR^d \) and the loss function is quadratic, i.e., if \( F(y)=y^2 \), then the estimator \( \mmle{h}(x) \) defined by \eqref{MLA est} is the weighted least squares estimator. If the matrix \( \mathbf{B}(x) \eqdef \sum_{i=1}^{n} \Psii \Psii^{\T} \w{h}{i}(x) \) is positive definite then one can write:
\begin{equation*}
    \mmle{h}(x) = \mathbf{B}(x)^{-1} \sum_{i=1}^{n}  \Psii \Yi \w{h}{i}(x)
\end{equation*}
In this case \( \mmle{h}(x) \) and \( \fle{h}{x} \) are linear estimators. Taking the polynomial basis we come back to the \(LP (p-1) \) estimator introduced in the previous section.

\section{Information-based complexity and \\approximation in increasing dimension}
\label{sec:Approx}
\emph{Computational complexity} is a measure of the intrinsic computational resources required to solve a mathematically formulated problem. It depends on the problem, but not on the particularly used algorithm. The notion ``information'' is used in the theory of complexity in the every-day sense of the word. The information is what we know about the problem to be solved. It should be stressed that this term used in Chapter \ref{chap:Approx} has nothing in common with Shannon's definition of information, nor with the Kullback-Leibler information criterion \citeasnoun{KL} used in Chapter~\ref{chap:Regression}. See \citeasnoun{Traub and Werschulz} for an informal introduction, however containing a comprehensive overview of the literature.

\par One can distinguish two different types of complexity. In the first case the information is complete, exact, and free; an example is provided by the traveling salesman problem. This is the so-called \emph{combinatorial complexity}. The \emph{information-based complexity} that we are interested in here, is the computational complexity of (multivariate) continuous mathematical models. This branch of computational complexity deals with the intrinsic difficulty of the approximate solution of a problem for which the information is partial, noisy, and priced, see \citeasnoun{TraubWW}. This is the case when dealing with continuous problems on infinite dimensional spaces. Only partial information such as a finite number of functional values is available. In this case the problem can only be solved approximately implying the presence of error. Usually one requires the problem to be solved with an error not larger than a threshold \( \eps \). The information-based complexity is then defined as the minimal number \( n(\eps,d) \) of information operations (functional values, for example), needed to solve the \( d \)-variate problem with an error not exceeding \( \eps \). In different settings and for different error criteria, \( \eps \) may have different meanings, but always reflects the error tolerance.

\par As pointed out in \citeasnoun{Novak and Wozniakowski}, a central issue is the study of how the information complexity depends on \( \eps^{-1} \) and \( d \). If \( n(\eps,d) \) depends exponentially on \( \eps^{-1} \) and~\( d \), the problem is called \emph{intractable}. Many multivariate problems exhibit exponential dependence on \( d \), called after Bellman \citeasnoun{B} the \emph{curse of dimensionality}. If the information complexity depends on \( \eps^{-1} \) and \( d \) polynomially, the problem is \emph{polynomially tractable}.

\par In spite of the existence of vast literature on the computational complexity of \( d \)-variate problems, most of the papers and books study error bounds without taking into account the dependence on \( d \). Research on \emph{tractability}, requiring the knowledge of dependence on both \( \eps^{-1} \) and~\( d \), was started in the early nineties by Wo\'zniakowski \citeasnoun{W92}, \citeasnoun{W94a}, \citeasnoun{W94b}, who introduced the notion of ``tractability'' and suggested to consider the dependence on \( d \) as \( d \to \infty \). This is important for numerous applications including physics, chemistry, finance, economics, and the computational sciences. For instance, in quantum mechanics, statistical mechanics and mathematical finance, for path integration the number of variables is infinite; approximations to path integrals result in arbitrary large \( d \), see \citeasnoun{R} and \citeasnoun{Novak and Wozniakowski} for details.

\par  In \emph{average case settings} the cost and the error are defined by their average performance. The general theory in the average case settings, among other approaches, was created by Traub, Wasilkowski, and Wo\'zniakowski in \citeasnoun{TraubWW}. The future development is presented by the monographs of Ritter \citeasnoun{R} and Novak and Wozniakowski \citeasnoun{Novak and Wozniakowski}.

\par One of the problems which can be treated in this framework is the approximation (recovery) of
functions. Let \( T = [0,1]^d \) and \( \cc F=C^k(T) \). We identify any \( f \in \cc F \) with its embedding \( id(f) =f \) in the (weighted) \( L_p \)-space over \( T \) with \( 1\le p\le  \infty \). Let the data be the functional values \( f(t_1), \ldots, f(t_n) \). Based on the data \( f(t_i) \) an approximate solution (function) \( \tilde f \) is constructed. The average error of \( \tilde f \)
is defined by \( (\EE\| f-\tilde f \|_p^q )^{1/q}\) with some \( 1 \le q < \infty \), where \( \| \cdot \|_p \) denotes a (weighted) \( L_p \)-norm.

\par Usually, the computational costs are proportional to the total number of functional values, and therefore to the information complexity. One aims at finding a ``good'' method \( \tilde f \) with average cost not exceeding a given bound and with minimal average error. Often one considers methods which use only the functional values \( f(t_i) \). Then the key quantity in the average case settings is the \( n \)th minimal average error
\begin{equation*}
    \inf_{t_i \in T} \inf_{a_i \in L_p(T)} \Big( \EE\| f - \sum_{i=1}^n a_i f(t_i) \|^q_{L_p(T)}\Big)^{1/q}.
\end{equation*}
This minimal error states how well \( f \) can be approximated on average by (affine) linear methods using \( n \) functional values. Chapter \ref{chap:Approx} is devoted to the approximation of \( d \)-parametric random fields of tensor product-type, which is a particular case of linear tensor product problems, see Chapter 6 of \citeasnoun{Novak and Wozniakowski} for a general study.

\chapter{Adaptive estimation under noise misspecification in regression}
\label{chap:Regression}

We consider the problem of pointwise estimation in
nonparametric regression with heteroscedastic additive Gaussian noise.
We use the method of local approximation applying the Lepski method for
selecting one estimator from a set of linear estimators obtained by different degrees of localization. This approach is combined with the
``propagation conditions'' on the choice of critical values of the
procedure, as suggested recently by
Spokoiny and Vial \citeasnoun{SV}. The ``propagation conditions'' are relaxed for the
model with misspecified covariance structure. Specifically, the model with
unknown mean and variance is approximated by the one with the parametric
assumption of local linearity of the mean function and with an
incorrectly specified covariance matrix.
We show that this procedure
allows a misspecification of the covariance matrix with a relative
error up to \( o\big( \frac{1}{\log n}  \big)  \), where \( n \) is the
sample size. The quality of estimation is measured in terms of
nonasymptotic ``oracle'' risk bounds.

\section{Model and set-up}
Consider a regression model
\begin{equation}\label{true model}
    \YY = \ff + \Sigma_0^{1/2} \eeps,\;\;\;\;\eeps \sim \norm{0}{I_n}
\end{equation}
with response vector \( \YY \in \RRn \) and the
covariance matrix \( \Sigma_0 = \diag(\sigma_{0,1}^2, \ldots ,\sigma_{0,n}^2 ) \).
 This model can be written as
 \begin{equation*}
     \Yi =\ffi + \sigma_{0,i} \,\epsi , \;\;\;\; i=1, \ldots , n
\end{equation*}
with design points \( \Xi \in \cc{X} \subset \RR^d \).
Given a point \( x \in \cc X\), the target of estimation is the value of the regression function \( f(x) \). We apply the method of local approximation described in Section \ref{subsec:Method of local approximation}. In view of the representation \eqref{true model}
this means that we believe that at a vicinity of some given point \( x\in \RR^d \) the unknown vector \( \ff \) can be well approximated by \( \ff_{\tta} = \PPsi^{\T}\tta \), where \( \PPsi \) is a given \( \p \times n \) matrix whose columns~\( \Psii \) consist of the values, at the design points, of basis functions centered at \( x \), that is, \( \Psi(u) = (\psi_1(u), \ldots, \psi_p(u))^{\T}\) for some basis \( \{ \psi_j \} \) in \( L_2(\cc X) \) and \( \Psii \eqdef \Psi(\Xi -x) \). The parameter \( \tta =(\ta{0},\,\ta{1}, \ldots , \ta{\p-1})^{\T} \in \Theta \subset \RR^{\p} \) is the target of estimation, and we will choose the appropriate width of the localization window adaptively by application of Lepski's method. The covariance matrix \( \Sigma_0 \) is not assumed to be known exactly and the approximate model used instead of the true one reads as follows:
\begin{equation}\label{PA}
    \YY = \PPsi^{\T}\tta + \Sigma^{1/2} \eeps,
\end{equation}
where \( \Sigma= \diag(\sigma_1^2, \ldots , \sigma_n^2 ) \), \(  \min\{ \sigma_{i}^2 \}>0 \).
Thus the model is misspecified in two places: in the form of the regression
function and in the error distribution. Following the abbreviation from Katkovnik
\citeasnoun{KatkEA2006} we will refer to this model as to ``local polynomial approximation'' or, more generally, ``local parametric approximation'' (LPA), since it is assumed that the ``true'' model \eqref{true model} \emph{locally} can be
replaced by the ``wrong'' parametric one.

\par The model constraint on the form of the regression function includes the important class of polynomial regressions. For example in the univariate case \( x \in \R \), due to the Taylor theorem, the approximation of the unknown function \( f(t) \) for \( t \) close to \( x \) can be written in the following form: \( \fta(t)= \ta{0} + \ta{1} (t-x) + \cdots + \ta{\p-1} (t-x)^{\p-1}/(p-1)! \), with the parameter \( \tta =(\ta{0},\,\ta{1}, \ldots , \ta{\p-1})^{\T} \) corresponding to the values of \( f \) and its derivatives at the point \( x \). The \( \p \times n \) matrix \( \PPsi \) then consists of the columns \( \Psii=\left(1,\, \Xi -x, \ldots, (\Xi -x)^{\p-1}/(p-1)!\right)^{\T} \),
\( i=1, \ldots , n \). If the regression function is sufficiently smooth then, for any \( t \) close to \( x \), up to a reminder term,  \( f(t) \approx  \fta(t) \) and the estimator of \( f(x) \) at the point \( x \) is given  by the first coordinate of \( \mmle{} \), that is by \( \fle{}{x} = f_{\tilde{\tta}} (x) = \tilde{\theta}^{(0)} \). See for further information on local polynomial regression Section \ref{sect:LocPol}, or for more deep insight \citeasnoun{Fan and Gijbels book},
\citeasnoun{KatkEA2006} or \citeasnoun{Loader}.

\par The general approach advocated in here includes also the important case of local constant approximation at a given point \( x \in \R \). In this case the design matrix \(   \PPsi = (1, \ldots ,1) \) and \( \fta(\Xi) = \PPsi_i^{\T} \tta = \theta^{(0)} = \fta(x) ,\;\;\;\;i=1, \ldots ,n. \)

\section{Quasi-maximum local likelihood estimation}
\par Fix a point \( x \in \RR^d \) and an orthogonal basis \( \{ \psi_j \} \) in \( L_2(\cc X) \).
Let the localizing operator be identified by the corresponding matrix. Thus for every \( x \) the sequence of localizing schemes (scales) \( \cc{W}_{k}(x) \), \( k=1, \ldots, K \) is given by the matrices
 \( \cc{W}_{k}(x) = \diag(\w{k}{1}(x), \ldots , \w{k}{n}(x)) \), where the weights
  \( \w{k}{i}(x)\in [0,1]\) can be understood, for instance, as smoothing kernels \( \w{k}{i}(x) = W ((\Xi -x) h^{-1}) \). We assume that a particular localizing function \( w_{(\cdot)} \) is fixed, and we aim to choose the index \( k \) of the optimal bandwidth \( h_k \) based on the available data. To simplify the notation we sometimes suppress the dependence on the reference point \( x \). Denote by
  \begin{equation}\label{def_Wk}
      \W{k} \eqdef  \Sigma^{-1/2} \cc{W}_{k}  \Sigma^{-1/2}  =
        \diag \left( \frac{\w{k}{1}}{\sigma_1^2} , \ldots ,
        \frac{\w{k}{n}}{\sigma_n^2}\right),\;\;\;k=1, \ldots, K.
  \end{equation}
Let \( \Theta \) be a compact subset of \( \RRp \). The LPA means that there exist non-zero weights \( \w{k}{i} \) and a parameter
\( \tta \in \Theta \) such that \( \ffi \approx \ftai = \Psi_{i}^{\T}\theta \) for all \( \Xi \)
providing \(\w{k}{i}>0 \). The notation \( \ffi \approx \Psi_{i}^{\T}\theta \) also has the meaning
that the localized data distribution, obtained by restricting the measures \( \P_{\ff, \Sigma_0} \) and
\( \P_{\PPsi^{\T}\tta, \Sigma_0} \) to the \( \sigma \)-field generated by those data for which \(\w{k}{i}>0 \),
 are close to each other in a certain sense, see {\it modeling bias} in Section~\ref{nearly parametric case}.
\par Under the LPA the corresponding local quasi-log-likelihood has the following form:
\begin{eqnarray}\label{log-likelihood}
      \LL(\W{k},\tta) &=& -\frac{1}{2 }
        \left( \YY -\PPsi^{\T}\tta\right)^{\T}   \W{k}
            \left( \YY -\PPsi^{\T}\tta\right) + R \nn
   &=& -\frac{1}{2} \sum_{i=1}^{n} |\Yi -\Psii^{\T} \tta|^2
    \frac{\w{k}{i}}{ \sigma_i^2} +R,
\end{eqnarray}
where \( R \) stands for the terms not depending on \( \tta \) and
\begin{equation*}
     \Psii = \Psi(\Xi -x) =
(\psi_1(\Xi -x) , \ldots, \psi_p(\Xi -x))^{\T} .
\end{equation*}
Then, due to the assumption of the normality of the errors, for every \( k \) the quasi-maximum likelihood estimator (QMLE) \( \mmle{k} =  \mmle{k}(x) =
(\mle{k}{0}(x),\,\mle{k}{1}(x), \ldots , \mle{k}{\,\p-1}(x) )^{\T} \) coincides with the LSE and  is defined as the minimizer of the weighted sum of squares from \eqref{log-likelihood}:
\begin{eqnarray}
  \mmle{k}  &\eqdef& \argmax_{\tta \in \Theta} \LL(\W{k},\tta) \label{MLE} \nn
            &=&   \argmin_{\tta \in \Theta} \| \W{k}^{1/2} \left( \YY -\PPsi^{\T}\tta\right)\|^2 \nn
            &=& \B{k}^{-1} \PPsi \W{k} \YY
                = \B{k}^{-1} \sum_{i=1}^{n}  \Psii \Yi \frac{\w{k}{i}}{ \sigma_i^2},
\end{eqnarray}
where the \( p \times p \) matrix \( \B{k} = \B{k}(x) \) is given by
\begin{equation}\label{B}
     \B{k} \eqdef   \PPsi  \W{k}\PPsi^{\T} =
        \sum_{i=1}^{n} \Psii \Psii^{\T} \frac{\w{k}{i}}{ \sigma_i^2}.
\end{equation}
That is, by Definition \ref{loc pol est} in the case of the polynomial basis the estimator \( \mmle{k}(x) \) is a
\( LP_k(p-1) \) estimator of \( \tta(x) \) corresponding to \( k \)th scale.
In the following we assume that \( n>p \) and \( \det \B{k}>0 \) for any \( k=1, \ldots , K \).
Because \( p=\rank(\B{k}) \le \min \{ p, \rank(\cc{W}_{k}(x)) \} \) this requires the following
conditions on the design matrix \( \PPsi \) and the minimal localizing scheme \( \cc{W}_{1}(x) \):
\begin{description}
\item[\( \bb{\mathfrak{(D)}} \)]
\emph{
The \( p \times n  \) design matrix \( \PPsi \) has full row rank, i.e.,
 \begin{equation*}
     \dim \cc{C}(\PPsi^{\T}) = \dim \cc{C}(\PPsi^{\T} \PPsi) =p.
 \end{equation*}
}
\end{description}

\begin{description}
\item[\( \bb{\mathfrak{(Loc)}} \)]
\emph{
The smallest localizing scheme \( \cc{W}_{1}(x) \) is chosen to contain at least \( p \) design points such that \( \w{1}{i}(x)>0 \), i.e., \( p \le \#\{ i: \w{1}{i}(x)>0 \} \).
}
\end{description}
The condition \( {\mathfrak{(Loc)}} \) is automatically fulfilled in practise since, for example,
in \( \RR^1 \) it means that for local constant fitting we need at least one observation and so on. Usually it is intrinsically assumed that, starting from the smallest window, at every step of the procedure every new window contains at least \( p \) new design points.

\par The formulas \eqref{MLE} give a sequence of estimators \( \{ \mmle{k}(x) \}_{k=1}^K \). It was noticed in~\citeasnoun{Akaike} that in the case of unknown true data distribution the MLE is a natural estimator for the parameter maximizing the expected log-likelihood. That is, for every \( k=1, \ldots, K \), the estimator \( \mmle{k}(x) \) can be considered as an estimator of
\begin{eqnarray}\label{prameter of BPF}
    \bbpf{k}(x) &\eqdef&
     \argmax_{\tta \in \Theta}  \EE \LL \left( \W{k}, \tta \right) \\
     &=& \argmin_{\tta \in \Theta}  (\ff-\PPsi^{\T}\tta)^{\T} \W{k} (\ff-\PPsi^{\T}\tta)\nn
      &=&   \B{k}^{-1} \PPsi \W{k}\ff
        =\B{k}^{-1} \sum_{i=1}^{n}  \Psii f(\Xi) \frac{\w{k}{i}}{ \sigma_i^2}.
\end{eqnarray}
Recall that we do not assume that the regression function \( f \) even locally satisfies the LPA.
It is known from \citeasnoun{White} that in the presence of model misspecification for every~\( k \)
the QMLE \( \mmle{k} \) is a strongly consistent estimator for \( \bbpf{k}(x) \), which is the minimizer of the localized Kullback-Leibler \citeasnoun{KullbackLeibler} information criterion:
\begin{eqnarray*}
  \bbpf{k}(x) &=&
     \argmin_{\tta \in \Theta} \sum_{i=1}^n
        \KL \left( \norm{\ffi}{\s_i} , \norm{\Psii^{\T} \tta}{\s_i}\right) \w{k}{i}(x) \\
  &=& \argmin_{\tta \in \Theta} \sum_{i=1}^n |\ffi - \Psii^{\T} \tta|^2 \frac{\w{k}{i}(x)}{\s_i^2}
\end{eqnarray*}
with \( \KL(P,P_{\tta}) \eqdef \EE_P \big[\log \big(\frac{\dd P}{\dd P_{\tta}}  \big)\big] \). For the properties of the Kullback-Leibler divergence see, for example, \citeasnoun{Tsybakov}.

\par It follows from the above definition of \( \bbpf{k}(x) \)  and from \eqref{MLE}
that the QMLE \( \mmle{k} \) admits a decomposition into
deterministic and stochastic parts:
\begin{eqnarray}
  & & \mmle{k}
  = \B{k}^{-1}\PPsi \W{k}  (\ff + \Sigma_0^{1/2} \eeps)
  = \bbpf{k} + \B{k}^{-1}\PPsi \W{k}
    \Sigma_0^{1/2} \eeps \label{linearity if quasiMLE}\\
  & & \EE \mmle{k}
    = \bbpf{k},
  \end{eqnarray}
  where \( \eeps \sim \norm{0}{I_n} \). Notice that if the regression function indeed follows the LPA, that is if
\( \ff \equiv \PPsi^{\T}\tta \), then \( \bbpf{k} \equiv \tta \)
for any \( k \) and the classical parametric set-up is recovered.

\section{Adaptive procedure}
\label{section: adaptive procedure}
Let a point \( x \in \cc X \subset \RRn \), an orthogonal basis \( \{ \psi_j \} \) in \( L_2(\cc X) \)  and the method of localization \( w_{(\cdot)} \) be fixed.
The crucial assumption for the procedure under consideration to work is that the localizing schemes (scales) \( \cc{W}_{k}(x) = \diag(\w{k}{1} , \ldots , \w{k}{n} ) \) are nested. Specifically, we say that the localizing schemes are nested if the following {\it ordering condition} is fulfilled:
\begin{description}
\item[\( \bb{(\cc{W})} \)]
    \emph{ For any fixed \( x \) and the method of localization \( w_{(\cdot)} \) the following relation holds:
    \begin{equation*}
    \cc{W}_{1}(x) \le  \ldots \le  \cc{W}_{k}(x) \le  \ldots \le  \cc{W}_{K}(x).
            \end{equation*}    }
\end{description}
For kernel smoothing this condition means the following. Let the sequence of bandwidths \( \{ h_k \} \) be ordered by increasing magnitude, i.e., \( h_1 <  \ldots < h_{K} \), and
let \( \cc{W}_{k}(x)  = \diag (\w{k}{1}, \ldots , \w{k}{n}  ) \) be the localizing matrix,
corresponding to the bandwidth \( h_k \). Here the weights \( \w{k}{i}  = \w{k}{i}(x) = W ( (\Xi-x) h_k^{-1} ) \in [0,1]\) are nonnegative functions such that for any \( 0<h_l < h_k<1 \) it holds \( W ( u h_l^{-1} ) \leq W ( u h_k^{-1} ) \) and \( W(u) \to 0 \) as \( |u| \to \infty \), or even are compactly supported.

\par Recall that given a center of localization \( x \in \cc X \), a basis \( \{ \psi_j \} \) and the method of localization \( w_{(\cdot)} \), we look for the estimator of \( f(x) \) having the form
\begin{equation*}
    \fle{k}{x} = \sum_{j=1}^p\tilde{\theta}^{(j)}_k(x) \psi_j(0).
\end{equation*}
The parameters \( \tilde{\theta}^{(j)}_k(x) \), \( j=1, \ldots, p \) are the components of the QMLE given by~\eqref{MLE}. The use of the adaptively chosen \( \adapind \) gives the adaptive estimator \( \fle{\adapind}{x}  \) of \( f(x) \) corresponding to the adaptive window choice \( \w{\adapind}{\cdot}(x)  \). In the case of the polynomial basis \( \psi_1(0)=1 \) and \( \psi_j(0) = 0 \) for \( j=2, \ldots, p \). Then the estimator of \( f(x) \) is just the first coordinate \( \tilde{\theta}^{(1)}_{\adapind}(x) \). In this case we also can get the estimators for the derivatives of \( f \) at the point \( x \).
\par The index \( \adapind \in \left\{ 1, \ldots, K \right\}\) corresponds to the adaptive choice of the degree of localization (of the width of the window), and it will be obtained by application of Lepski's method, see below.
Then the adaptive estimator of the parameter vector is
\begin{equation}\label{aadapest}
  \aadapest(x) \eqdef  \mmle{\adapind}(x) =
        (\tilde{\theta}^{(1)}_{\adapind}(x), \ldots, \tilde{\theta}^{(p)}_{\adapind}(x))^{\T}.
\end{equation}
In a non-formal way the idea of the adaptive procedure used for selection of \( \adapind \) can be described as follows. Let a point \( x \) and the method of localization \( W \) be fixed. For
 \( k=1, \ldots , K \), let \( \mmle{k} = \mmle{k}(x) = ( \mle{k}{0}(x), \mle{k}{1}(x),
\ldots , \mle{k}{\;\p-1}(x)  )^{\T} \) be the linear estimator defined
 by \eqref{MLE}. We aim to choose an adaptive estimator \( \aadapest(x) =  \mmle{\adapind}(x) \) from the set \( \{\mmle{1}, \ldots,  \mmle{K}   \} \), that is to pick the adaptive index
  \( \adapind \) from \( \left\{ 1, \ldots, K \right\} \).
Following the Lepski method (see \citeasnoun{Lep1990}), we will proceed with the multiple testing of homogeneity: starting with the smallest scheme \( \cc W_1(x) \) and enlarging it step by step so long as the estimators \( \mmle{l}(x) \) do not differ from each other significantly. More precisely, to describe the test statistic, define for any \( \tta \), \( \tta' \in \Theta\) the corresponding log-likelihood ratio:
  \begin{equation}\label{log-likelihood ratio}
     \LL(\W{k},\tta, \tta') \eqdef \LL(\W{k},\tta) -\LL(\W{k},\tta').
  \end{equation} Then, using the approach suggested in \citeasnoun{KatkSpok}, for every \( l = 1, \ldots, K \), \emph{the fitted log-likelihood (FLL)} ratio is defined as follows:
\begin{equation*}
    \LL(\W{l},\mmle{l}, \tta') \eqdef \max_{\tta \in \Theta} \LL(\W{l},\tta, \tta').
  \end{equation*}
By Theorem \ref{Th. Spokoiny_fitted likelihood}, for any \( l \) and \( \tta \), the FLL is a quadratic form:
\begin{equation*}
    2 \LL(\W{l}, \mmle{l}, \tta) = (\mmle{l} - \tta)^{\T}
        \B{l} (\mmle{l} - \tta).
\end{equation*}

\par Define the confidence set corresponding to \( \mmle{l} \) as
\begin{eqnarray}\label{confidence set}
    \cc{E}_l (\z_l)
    &\eqdef&
    \left\{ \tta: 2 \LL ( \W{l}, \mmle{l}, \tta ) \leq \z_l  \right\}\nn
    &=&
    \left\{ \tta: (\mmle{l} - \tta)^{\T}
        \B{l} (\mmle{l} - \tta) \leq \z_l  \right\} .
\end{eqnarray}
In terms of this definition ``the estimator \( \mmle{k} \) does not differ significantly from \( \mmle{l} \)'' means that \(  \mmle{k}  \in \cc{E}_l (\z_l)  \). This prompts to use (see \citeasnoun{KatkSpok}) the \emph{FLL-statistics}:
 \begin{eqnarray}\label{Tlk}
    T_{lk} &\eqdef&
                2 \LL(\W{l}, \mmle{l}, \mmle{k})\nn
           & = &
                (\mmle{l} - \mmle{k})^{\T}  \B{l} (\mmle{l} - \mmle{k})\;,\;\;\;\; l < k.
 \end{eqnarray}
 If \( T_{lk} \) is significantly large, say \( T_{lk} > \z_l \) for some sufficiently big value
 \( \z_l \), then the discrepancy between \( \mmle{l} \) and \( \mmle{k} \) is not negligible and the corresponding hypothesis of homogeneity should be rejected in favor of the smaller one. Notice that this simple approach works only due to the condition \( (\cc{W} ) \),
  because the hypotheses are nested.

  \par A justification of the FLL approach is given by the fact that the fitted log-likelihood ratio
  \( \LL(\W{k},\mmle{k}, \bbpf{k}) \) can be used to measure the quality of estimation of~\( \bbpf{k} \) by its empirical counterpart \( \mmle{k} \) at each level of localization
   (see \citeasnoun{Akaike} and \citeasnoun{White}).
\subsection{Algorithm}
Given the set of linear estimators \( \{ \mmle{1}, \ldots, \mmle{K}  \} \) and the set of critical values \( \{ \z_1, \ldots, \z_{K-1} \} \), see the \emph{``propagation conditions''} from the next subsection for details, one aims to select in a data-driven way the estimator \( \aadapest = \mmle{\adapind} \) with
\( \adapind \in \{1,\ldots, K\} \). The selection procedure originating from~\citeasnoun{Lep1990} is described as follows:
\begin{equation*}
    \mmle{1}
\end{equation*}
\begin{equation*}
    \vdots
\end{equation*}
\begin{equation*}
     \mmle{k} \;\;  \text{is accepted iff} \;\;  \mmle{k-1} \;\;  \text{was accepted and}
\end{equation*}
\begin{equation*}
        \mmle{k} \in \bigcap_{l<k} \cc{E}_l(\z_l)
        \iff
        \bigcap_{l<k} \{ T_{lk} \le  \z_l \} \not= \emptyset.
\end{equation*}
That is, we use Lepski's selection rule with the FLL test statistics \( \{ T_{lm} \} \):
 \begin{equation}\label{adaptind}
    \adapind = \max \left\{k \leq K :   T_{lm}  \leq \z_l, \,l < m \leq k    \right\}.
 \end{equation}
 \subsection{Choice of the critical values}
Let \( \aadapest_{k} \) denote the last accepted estimator after the
 first \( k \) steps of the procedure:
  \begin{equation}\label{the last accepted}
     \aadapest_{k} \eqdef \mmle{\min \{k,\adapind\}}.
  \end{equation}
Denote for some \( \kappa \le  K \) the hypothesis
\( H_{\kappa} \): \( \bbpf{1} = \cdots = \bbpf{\kappa} = \tta \), which means that the LPA is
fulfilled up to the step \( \kappa \). Clearly, by Assumption \( {(\cc{W})} \) for any
\( k < \kappa \) the hypothesis \( H_k \) is included in \( H_{\kappa} \).

\par Following the idea proposed in \citeasnoun{SV} we will choose the critical values \( \z_1, \ldots ,\z_{K-1} \)
 of the procedure using a kind of ``level'' conditions under the LPA
 (homogeneity hypothesis). In other words, the procedure is optimized to provide
 the desired error level in the local parametric situation. As it will
 be shown later (see Theorem \ref{oracle result}), if the procedure is tuned well under the LPA, it will
 perform well even when this assumption is violated.

\par The Wilks-type Theorem \ref{Wilks theorem} below gives the bound for the expected fitted
log-likelihood ratio:
\begin{equation}\label{bound for the expected fitted
log-likelihood ratio}
    \EE |2 \LL( \W{k}, \mmle{k}, \bbpf{k} )  |^r  \le (1 + \delta)^r C(p,r)
\end{equation}
where the constant \( C(p,r) \) does not depend on the degree of localization and is given by:
\begin{equation}\label{def C(p,r)}
     C(p,r) = \EE|\chi^2_p|^r = 2^r \frac{\Gamma (r+ \frac{p}{2} )}{\Gamma(\frac{p}{2})},
\end{equation}

Take some ``confidence level'' \( \alpha \in (0,1] \). Then the set of \( K-1 \) conditions on the choice of the critical values \( \z_1, \ldots ,\z_{K-1} \) can be defined to provide at each step of the procedure a risk of the adaptive estimators of at most an \( \alpha \)-fraction of the best possible (parametric) risk
\eqref{bound for the expected fitted log-likelihood ratio}. These conditions are given by the following formulas:
 \begin{definition} (Propagation conditions (PC))
\par The critical values \( \z_1, \ldots ,\z_{K-1} \) satisfy the following set of conditions:
 \begin{equation}\label{PC}
    \EE_{0, \Sigma} |(\mmle{k} - \aadapest_{k})^{\T} \B{k} (\mmle{k} - \aadapest_{k})|^{r}
        \leq \alpha  C(p,r)\; \; \;        \text{for all}\;\; k=2, \ldots, K,
 \end{equation}
where \( C(p,r) \) is defined by \eqref{def C(p,r)}, \( \alpha \in (0,1] \) and \( \EE_{0, \Sigma} \) stands for the expectation w.r.t.
the measure \( \norm{0}{\Sigma} \).
 \end{definition}
\begin{remark}
Lemma \ref{Pivotality property} (see Section \ref{section:Auxiliary results}) shows that under the LPA the Gaussian distribution
provides a nice pivotality property: the actual value of the parameter \( \tta \) is not important for
the risk of adaptive estimator, so one can put \( \tta = 0 \) in \eqref{PC}.
\end{remark}
\begin{remark}
Since the procedure is fitted in the parametric situation, ideally
(while the LPA holds) it should not terminate. If it does, then the critical values are too small. This event will be referred to as a ``false alarm''. Therefore by the  \((PC)\) we require that at each level of localization the risk associated with the type I error is at most an \( \alpha \)-fraction of the corresponding
 risk in the parametric situation.
\end{remark}

\section{Theoretical study}

\subsection{Local parametric risk bounds}

To justify the statistical properties of the considered procedure we need the following simple observation. Let for any \( \tta \), \( \tta' \in \Theta\) the corresponding log-likelihood ratio \( \LL(\W{k},\tta, \tta') \) be defined by \eqref{log-likelihood ratio}.
       Then
\begin{equation*}
    2 \LL(\W{k},\tta, \tta')=
        \left( \YY -\PPsi^{\T} \tta'  \right)^{\T}
        \W{k}
        \left( \YY -\PPsi^{\T} \tta'  \right)
     - \left( \YY -\PPsi^{\T} \tta  \right)^{\T}
        \W{k}
        \left( \YY -\PPsi^{\T} \tta  \right).
     \end{equation*}
\begin{theorem}\label{Th. Spokoiny_fitted likelihood}
(Quadratic shape of the fitted log-likelihood)

\par Let for every \( k = 1, \ldots, K\) the fitted log likelihood (FLL) be defined as follows:
\begin{equation*}
    \LL(\W{k},\mmle{k}, \tta') \eqdef \max_{\tta} \LL(\W{k},\tta, \tta').
  \end{equation*}
Then
\begin{equation}
    2 \LL(\W{k},\mmle{k}, \tta)
    =  ( \mmle{k} -\tta )^{\T} \B{k} ( \mmle{k} -\tta ).
\end{equation}
\end{theorem}

\begin{proof}
Notice that \( \LL(\W{k}, \tta) \) defined by \eqref{log-likelihood} is quadratic in \( \tta \). The assertion follows from
the Taylor expansion of the second order at the point \( \mmle{k} \) because it is the point of maximum and the second derivative is a constant matrix \( \B{k} \).
\end{proof}
\par In order to control the admissible level of misspecification for the ``model'' covariance matrix from \eqref{PA} we need to introduce the
following condition on the relative variability in errors:
\begin{description}
\item[\( \bb{\mathfrak{(S)}} \)]
\emph{
There exists \( \delta \in [0,1) \) such that
\begin{equation*}
1-\delta \le \sigma_{0,i}^2/\sigma_{i}^2 \le 1+ \delta \; \;\; \text{for all} \;\; \; i = 1, \ldots, n.
\end{equation*}
}
\end{description}
Let the matrix \( \S \) be defined as follows:
\begin{equation}\label{def S}
    \S \eqdef \Sigma_0^{1/2} \W{k} \PPsi^{\T} \B{k}^{-1}
        \PPsi \W{k} \Sigma_0^{1/2} .
\end{equation}
Then for the distribution of \( \LL(\W{k},\mmle{k}, \bbpf{k}) \) one observes the so-called ``Wilks phenomenon'' (see \citeasnoun{FanZhangZhang} ) described by the following theorem:
\begin{theorem}\label{Wilks theorem}
Let the regression model be given by \eqref{true model} and the parameter maximizing the expected local log-likelihood \( \bbpf{k} = \bbpf{k}(x) \) be defined by \eqref{prameter of BPF}. Then for any \( k = 1, \ldots, K \) the following equality in distribution takes place:
\begin{equation}\label{likelihood = weightsum of sqnormal }
    2 \LL ( \W{k}, \mmle{k}, \bbpf{k} ) \eqdistr
    \lambda_{1}(\S)\bar\varepsilon_{1}^{2} + \cdots +
    \lambda_{p}(\S)\bar\varepsilon_{p}^{2},
\end{equation}
where \( p = \rank (\B{k}) = \dim \Theta = \p \), \( \lambda_{1}(\S) , \ldots , \lambda_{p}(\S) \)
are the non-zero eigenvalues of the matrix \( \S  \) and \( \bar\varepsilon_{i} \) are independent
standard normal random variables.

\par Moreover, under Assumption \( \mathfrak{(S)} \) it holds that the maximal eigenvalue fulfills \( \lambda_{max}(\S) \le 1+ \delta \) and for any \( \z > 0 \)
\begin{equation}\label{chi squared domination}
    \P \left\{2 \LL ( \W{k}, \mmle{k}, \bbpf{k} ) \ge  \z\right\} \le
    \P \left\{\eta \ge  \z/(1+ \delta)\right\},
\end{equation}
where \( \eta  \) is a random variable distributed according to the \( \chi^2 \)
law with \( p \) degrees of freedom.
\end{theorem}
\begin{remark}
Generally, if the matrix \( \B{k} \) is degenerated in \eqref{likelihood = weightsum of sqnormal }
the number of terms \( p \le \dim \Theta \).
\end{remark}
\begin{proof}
    By Theorem \ref{Th. Spokoiny_fitted likelihood} and the decomposition \eqref{linearity if quasiMLE}
     it holds that:
\begin{eqnarray*}
 2 \LL ( \W{k}, \mmle{k}, \bbpf{k} )
 &=& ( \mmle{k} -\bbpf{k} )^{\T} \B{k}
        ( \mmle{k} -\bbpf{k} )\\
 &=& (\B{k}^{-1} \PPsi \W{k} \Sigma_0^{1/2} \eeps)^{\T}
    \B{k}
    (\B{k}^{-1} \PPsi \W{k} \Sigma_0^{1/2} \eeps)\\
 &=& \eeps^{\T} \S \eeps,
\end{eqnarray*}
where the symmetric matrix \( \S \) is defined by \eqref{def S}. Then by the Schur theorem there exist an orthogonal matrix \( \M \) and a diagonal matrix \( \Lam \) composed of the eigenvalues of \( \S \) such that \(     \S = \M^{\T} \Lam \M \). For \( \eeps \sim \norm{0}{I_n} \) and an orthogonal matrix \( \M \) it holds that \( \bar{\eeps} \eqdef \M\eeps \sim \norm{0}{I_n} \). Indeed, \( \EE\M \eeps = \EE \eeps =0 \) and
\begin{equation*}
     \Var \M\eeps = \EE \M\eeps (\M\eeps)^{\T} =
     \M \EE(\eeps \eeps^{\T}) \M\eeps = \M \M^{\T}  = I_n.
\end{equation*}
Therefore
\begin{equation*}
    2 \LL ( \W{k}, \mmle{k}, \bbpf{k} )
   \eqdistr \bar{\eeps}^{\T} \Lam \bar{\eeps}\; ,\;\;\;\;
   \bar{\eeps}  \sim \norm{0}{I_n}.
\end{equation*}
On the other hand, the matrix \( \S =\Sigma_0^{1/2}\W{k} \PPsi^{\T} \B{k}^{-1} \PPsi \W{k} \Sigma_0^{1/2} \) can be rewritten as:
\begin{equation*}
    \S = \Sigma_0^{1/2}\W{k}^{1/2}\PPi_k \W{k}^{1/2}\Sigma_0^{1/2},
\end{equation*}
with \( \PPi_k = \W{k}^{1/2} \PPsi^{\T} \B{k}^{-1} \PPsi  \W{k}^{1/2}\). Notice that \( \PPi_k \) is an orthogonal projector onto the linear subspace of dimension \( p = \rank (\B{k}) \) spanned by the rows of matrix \( \PPsi \). Indeed, \( \PPi_k \) is symmetric and idempotent, i.e.,\( \PPi_k^2 = \PPi_k \).

\par Moreover, \( \rank(\PPi_k) = \tr(\PPi_k) = \tr(\W{k}^{1/2} \PPsi^{\T} \B{k}^{-1} \PPsi  \W{k}^{1/2}) = \tr(\B{k}^{-1} \PPsi  \W{k} \PPsi^{\T}) =\tr(\B{k}^{-1} \B{k})= \tr(I_p) = p \). Therefore \( \PPi_k \) has only \( p \) unit eigenvalues and \( n-p \) zero eigenvalues. Notice also that the \( n \times n \) matrix \( \S \) has \( \rank(\S) = \rank(\PPi_k \W{k}^{1/2}\Sigma_0^{1/2} ) = \rank(\PPi_k)=p \) as well. Thus  \( 2 \LL ( \W{k}, \mmle{k}, \bbpf{k} ) \eqdistr  \lambda_{1}(\S)\bar\varepsilon_{1}^{2} + \cdots +
    \lambda_{p}(\S)\bar\varepsilon_{p}^{2}\), where  \( \lambda_{1}(\S) , \ldots , \lambda_{p}(\S) \)
are the non-zero eigenvalues of the matrix \( \S  \).

\par Define the \( L_2  \)-norm of a matrix \( \A \) via its maximal eigenvalue
\begin{equation}\label{def om matrix norm}
    \| A \| \eqdef \sqrt{\lambda_{max}(A^{\T}A) }.
\end{equation}
Thus, taking into account Assumption \( (\mathfrak{S}) \), the induced \( L_2  \)-norm of the matrix \( \S \) can be estimated as follows:
\begin{eqnarray*}
  \| \S \| &=& \| \Sigma_0^{1/2}\W{k}^{1/2}\PPi_k \W{k}^{1/2}\Sigma_0^{1/2} \|\\
     &\le& \| \Sigma_0^{1/2}\W{k}^{1/2} \| \| \PPi_k \| \| \W{k}^{1/2}\Sigma_0^{1/2} \| \\
     &=&  \lambda_{max}(\W{k}\Sigma_0) \lambda_{max}(\PPi_k)  \\
    &=&   \max_{i}\{ \w{k}{i} \frac{\sigma_{0,i}^2}{\sigma_{i}^2}  \} \\
    &\le & (1 + \delta) \max_{i}\{ \w{k}{i} \} \le  1 + \delta.
\end{eqnarray*}
Therefore the largest eigenvalue of the matrix \( \S \) is bounded:
 \(\lambda_{max}(\S) \le  1 + \delta \).

 \par The last assertion of the theorem follows from the simple observation that
 \begin{equation*}
    \P \left\{ \lambda_{1}(\S)\bar\eps_1^2 + \cdots
        + \lambda_{p}(\S) \bar\eps_p^2  \ge  \z \right\}
    \le  \P \left\{\lambda_{max}(\S) (\bar\eps_1^2 + \cdots +
        \bar\eps_p^2  ) \ge  \z  \right\}.
 \end{equation*}
\end{proof}

\begin{corollary} (Quasi-parametric risk bounds) \label{param. risk bounds}
\par Let the model be given by \eqref{true model} and \( \bbpf{k} = \bbpf{k}(x) \)
 be defined by \eqref{prameter of BPF}. Assume \( (\mathfrak{S} ) \). Then for any \( \mu <1/( 1 + \delta) \)
\begin{eqnarray}
   \EE \exp \{\mu \LL (\W{k}, \mmle{k}, \bbpf{k} ) \}
    &\le&  \left[1-\mu (1 + \delta) \right]^{-p/2} \label{param. risk bounds exp}\\
  \EE |2 \LL( \W{k}, \mmle{k}, \bbpf{k} )  |^r
  & \le& (1 + \delta)^r C(p,r)\;, \label{param. risk bounds polinom}
\end{eqnarray}
where
\begin{equation}\label{C(p,r)}
     C(p,r) = \EE|\chi^2_p|^r = 2^r \frac{\Gamma (r+ \frac{p}{2} )}{\Gamma(\frac{p}{2})}.
\end{equation}
\end{corollary}
\begin{proof}
By \eqref{likelihood = weightsum of sqnormal }
and independence of \( \bar{\eps}_i \)
\begin{eqnarray*}
    \EE \exp \{ \mu \LL( \W{k}, \mmle{k}, \bbpf{k} ) \}
    &=&  \EE \exp \left\{ \frac{\mu}{2}
             \sum_{i=1}^{p} \lambda_i(\S) \bar{\eps}_i^2 \right \}\\
    &=&  \prod_{i=1}^{p} \EE \exp \left\{ \frac{\mu}{2} \,
        \lambda_i(\S) \bar{\eps}_i^2  \right \}\\
    &=& \prod_{i=1}^{p} \left[ 1 - \mu \,\lambda_i(\S)  \right]^{-1/2}\\
    &\le & \left[ 1 - \mu \,\lambda_{max}(\S)  \right]^{-p/2} \\
    &\le & [ 1- \mu (1+\delta)]^{-p/2}.
\end{eqnarray*}
Let \( \eta \sim \chi^2_p \). Integrating by parts yields the second inequality:
\begin{eqnarray*}
  \EE | 2 \LL( \W{k}, \mmle{k}, \bbpf{k} ) |^r
    &=& \int_0^{\infty}
        \P\left\{2\LL( \W{k}, \mmle{k}, \bbpf{k} )
        \ge \z\right\} r \z^{r-1} \dd\z \\
  &\le& r \int_0^{\infty}
        \P\left\{ \eta \ge \z/(1+\delta)\right\} \z^{r-1} \dd \z \\
  &=& (1 + \delta)^r \,\EE|\eta|^r.
\end{eqnarray*}
\end{proof}

\subsection{Upper bound for the critical values}

\par  Let us recall the \emph{(partial) L\"{o}wner ordering} of matrices: for any real symmetric matrices \( A \) and \( B \) we will write \( A \preceq B \) if and only if \( \vartheta^{\T} A \, \vartheta  \leq \vartheta^{\T} B \, \vartheta  \) for all vectors~\( \vartheta \), or, equivalently if and only if the matrix \(  B -  A \) is nonnegative definite.

Assuming \( {\mathfrak{(S)}} \) the true covariance matrix fulfills \( \Sigma_0 \preceq \Sigma (1+\delta) \)
 and the variance of the estimator \( \mmle{k} \) is bounded above by \( \B{k}^{-1} \):
\begin{eqnarray}
     V_{k} \eqdef \Var \mmle{k}
        &=&    \B{k}^{-1}\PPsi\W{k} \Sigma_0 \W{k}  \PPsi^{\T}\B{k}^{-1} \label{Vk def}
\\        &\preceq &  (1+\delta)\B{k}^{-1}\PPsi\W{k} \Sigma  \W{k} \PPsi^{\T}\B{k}^{-1}\nn
        &=&  (1+\delta) \B{k}^{-1}\PPsi  \Sigma^{-1/2} \cc W_{k}^2 \Sigma^{-1/2}  \PPsi^{\T}\B{k}^{-1}\nn
        &\preceq & (1+\delta) \B{k}^{-1}\PPsi  \Sigma^{-1/2} \cc W_{k} \Sigma^{-1/2} \PPsi^{\T}\B{k}^{-1}\nn
       &=& (1+\delta) \B{k}^{-1} \PPsi \W{k} \PPsi^{\T} \B{k}^{-1} \nn
       &=& (1+\delta) \B{k}^{-1}.\label{Vk bound}
       \end{eqnarray}
The last inequality follows from the observation that all the entries of the ``weight''
matrix \( \cc{W}_{k} \) do not exceed one, implying \(\cc W_{k}^2 \preceq \cc W_{k} \). Strict equality
occurs if the \( \{ \w{k}{i} \} \) are boxcar (rectangular) kernels and the noise is known, i.e.,
 \( \delta = 0 \). To justify the procedure one needs to show that the critical values chosen by the \((PC)\) are finite.
The upper bound for the critical values is obtained under the following assumption:
\begin{description}
\item[\( \bb{\mathfrak{(B)}} \)]
\emph{
Let the matrices \( \B{k} \) satisfy
\begin{equation*}
   u_0 I_{\p} \preceq  \B{k-1}^{-1/2} \, \B{k} \, \B{k-1}^{-1/2} \preceq u I_{\p}
\end{equation*}
for some constants \( u_0 \) and \( u \) such that \( 1< u_0 \le  u \)
for any \( 2\le k\le K \)
}
\end{description}
 \begin{remark}
 In the ``one dimensional case'' \( \p=1 \), that is for local constant approximation, the ``matrix''
 \( \B{k} = \sum_{i=1}^{n} \w{k}{i} \sigma_i^{-2} \ge  \B{k-1} \) is just a weighted
``local design size''. Assume for simplicity that \( \sigma_i^{2} \equiv \sigma^2 \), the weights are rectangular kernels \( \w{k}{i}(x) = \ind\{ |\Xi - x| \le h_k/2 \} \), and the design is equidistant. Then for \( n \) sufficiently large
 \begin{equation*}
     \frac{1}{n} \B{k} = \frac{1}{n\sigma^2} \sum_{i=1}^{n} \ind\{ |\frac{i}{n} - x| \le h_k/2 \} \approx  \frac{h_k}{\sigma^2},
 \end{equation*}
and the condition \( \mathfrak{(B)} \) means that the bandwidths grow geometrically: \( h_k =u h_{k-1} \).
 \end{remark}
Denote for any \( l<k \) the variance of the difference \( \mmle{k} - \mmle{l} \) by \( V_{lk} \):
\begin{equation}\label{def_Vlk}
     V_{lk} \eqdef \Var (\mmle{k} - \mmle{l}) \succ 0 .
\end{equation}
Then there exists a unique matrix \( V_{lk}^{1/2} \succ 0 \) such that \(( V_{lk}^{1/2})^2 = V_{lk} \).

\begin{lemma}\label{Tlk large div}
\par Assume \( \mathfrak{(S)} \), \( \cc{(W)} \) and \( (\mathfrak{B}) \). If for some \( k \le  K \) the
LPA is fulfilled, that is
if \( \bbpf{1} = \cdots = \bbpf{k} = \tta \), then for any \( l <  k \) it holds that:
\begin{eqnarray*}
   \P \left\{2 \LL ( \W{l}, \mmle{l}, \mmle{k} ) \ge \z\right\}
        &\le&
            \P \left\{\eta \ge  \z/\lambda_{max}(V_{lk}^{1/2} \B{l} V_{lk}^{1/2})\right\}\\
        &\le&
            \P \left\{\eta \ge  \z/t_0\right\}\\
   \P \left\{2 \LL ( \W{k}, \mmle{k}, \mmle{l} ) \ge  \z\right\}
   &\le&
     \P \left\{\eta \ge  \z/\lambda_{max}(V_{lk}^{1/2} \B{k} V_{lk}^{1/2})\right\}\\
   &\le&
      \P \left\{\eta \ge  \z/t_1\right\},
 \end{eqnarray*}
 where \( t_0=2(1+\delta) (1 + u_0^{-(k-l)}) \), \( t_1 =2(1+\delta) (1 + u^{(k-l)}) \)
 and   \( \eta \) is a \( \chi^2_p \)-distributed random variable.
\end{lemma}

\begin{proof}
The LPA and \eqref{linearity if quasiMLE} imply
\begin{equation*}
    \mmle{l}-\mmle{k}
        =
            \B{l}^{-1}\PPsi\W{l}\Sigma_0^{1/2}\eeps-\B{k}^{-1}\PPsi\W{k}\Sigma_0^{1/2}\eeps
        \eqdistr
            V_{lk}^{1/2}\xi,
\end{equation*}
where \( \xi \) is a standard normal vector in \( \R^{\p} \). Thus by Theorem
 \ref{Th. Spokoiny_fitted likelihood} under the LPA for any \( l<k \)
\begin{equation*}
  2\LL(\W{l}, \mmle{l}, \mmle{k}) =  \| \B{l}^{1/2} (\mmle{l}  - \mmle{k}) \|^2 \\
  \eqdistr  \xi^{\T} V_{lk}^{1/2} \B{l} V_{lk}^{1/2} \xi.
\end{equation*}
 By the Schur theorem there exists an orthogonal matrix \( M \) such that
\begin{equation*}
    \xi^{\T} V_{lk}^{1/2} \B{l} V_{lk}^{1/2} \xi
    \eqdistr  \bar\varepsilon^{\T} M^{\T} \Lambda_{lk} M \bar\varepsilon,
\end{equation*}
where \( \bar\varepsilon \) is standard normal vector,
\( \Lambda = \diag (\lambda_{1}(V_{lk}^{1/2} \B{l} V_{lk}^{1/2})) ,
 \cdots , \lambda_{\p}(V_{lk}^{1/2} \B{l} V_{lk}^{1/2}))\) and \( p = \rank (\B{l})\).
 Therefore
\begin{equation*}
  2 \LL ( \W{l}, \mmle{l}, \mmle{k} ) \eqdistr
    \lambda_{1}(V_{lk}^{1/2} \B{l} V_{lk}^{1/2})\bar\varepsilon_{1}^{2} + \cdots +
    \lambda_{p}(V_{lk}^{1/2} \B{l} V_{lk}^{1/2})\bar\varepsilon_{p}^{2},
 \end{equation*}
 where \( \lambda_{j}(V_{lk}^{1/2} \B{l} V_{lk}^{1/2}) \), \( j=1,\ldots , p \) are nonzero eigenvalues
 of \( V_{lk}^{1/2} \B{l} V_{lk}^{1/2} \).

 By a similar argument:
 \begin{equation*}
    2 \LL ( \W{k}, \mmle{k}, \mmle{l} ) \eqdistr
    \lambda_{1}(V_{lk}^{1/2} \B{k} V_{lk}^{1/2})\bar\varepsilon_{1}^{2} + \cdots +
    \lambda_{p}(V_{lk}^{1/2} \B{k} V_{lk}^{1/2})\bar\varepsilon_{p}^{2}.
 \end{equation*}
Recalling that \( \eta \) is a \( \chi^2_p \)-distributed random variable, we have
 \begin{eqnarray*}
   \P \left\{2 \LL ( \W{l}, \mmle{l}, \mmle{k} ) \ge \z\right\} &\le&
    \P \left\{\eta \ge  \z/\lambda_{max}(V_{lk}^{1/2} \B{l
    } V_{lk}^{1/2})\right\},\\
   \P \left\{2 \LL ( \W{k}, \mmle{k}, \mmle{l} ) \ge  \z\right\} &\le&
    \P \left\{\eta \ge  \z/\lambda_{max}(V_{lk}^{1/2} \B{k} V_{lk}^{1/2})\right\}.
 \end{eqnarray*}
Notice that for any square matrices \( A \) and \( B \),
\begin{equation*}
    (A-B)(A^{\T}-B^{\T}) \preceq 2  (AA^{\T}+BB^{\T}).
\end{equation*}
Application of this bound to the variance of the difference of estimators yields
\begin{eqnarray*}
  V_{lk} &=&    (\B{l}^{-1}\PPsi\W{l}\Sigma_0^{1/2}-\B{k}^{-1}\PPsi\W{k}\Sigma_0^{1/2} )
                (\B{l}^{-1}\PPsi\W{l}\Sigma_0^{1/2}-\B{k}^{-1}\PPsi\W{k}\Sigma_0^{1/2})^{\T} \\
&\preceq& 2 (\B{l}^{-1}\PPsi\W{l} \Sigma_0 \W{l}  \PPsi^{\T}\B{l}^{-1} + \B{k}^{-1}\PPsi\W{k} \Sigma_0 \W{k}  \PPsi^{\T}\B{k}^{-1} ) \\
&=&  2V_l +2V_k,
\end{eqnarray*}
where \( V_l=\Var\mmle{l} \), \( l\le k \). By the upper bound \eqref{Vk bound} for the variance
 \( V_{l} \) (resp. of \( V_{k} \)) and by Assumption \( \mathfrak{(B)} \):
 \begin{eqnarray*}
V_{l} & \preceq&  (1+ \delta)\B{l}^{-1},\\
V_{k} & \preceq&  (1+ \delta)\B{k}^{-1} \preceq  (1+ \delta)  u_0^{-(k-l)}  \B{l}^{-1}, \\
 V_{lk} & \preceq& 2(1+\delta) (1 + u_0^{-(k-l)}) \B{l}^{-1}.
\end{eqnarray*}
Therefore
\begin{equation}\label{bound for B_l}
  \B{l} \preceq 2(1+\delta) (1 + u_0^{-(k-l)}) V_{lk}^{-1}.
 \end{equation}
Thus by \eqref{bound for B_l} the upper bound for the induced \( L_2\) matrix norm
reads as follows:
 \begin{eqnarray}
\lambda_{max}(V_{lk}^{1/2} \B{l} V_{lk}^{1/2}) &=& \|  \B{l}^{1/2} V_{lk}^{1/2}\|^2 \nn
&=& \sup_{\| \gamma \| = 1} \gamma^{\T}  V_{lk}^{1/2} \B{l} V_{lk}^{1/2} \gamma  \nn
&\le& 2(1+\delta) (1 + u_0^{-(k-l)}) \sup_{\| \gamma \| = 1}
                  \gamma^{\T}  V_{lk}^{1/2} V_{lk}^{-1} V_{lk}^{1/2} \gamma \nn
&\le & 2(1+\delta) (1 + u_0^{-(k-l)})\label{lambda max l}.
\end{eqnarray}
Similarly:
\begin{eqnarray}
 V_{lk} &\preceq& 2(1+\delta) (1 + u^{(k-l)}) \B{k}^{-1} ,\nn
  \lambda_{max}(V_{lk}^{1/2} \B{k} V_{lk}^{1/2})&\le & 2(1+\delta) (1 + u^{(k-l)})
                                                                    \label{lambda max k}.
\end{eqnarray}
These bounds imply
\begin{eqnarray*}
   \P \left\{2 \LL ( \W{l}, \mmle{l}, \mmle{k} ) \ge \z\right\}
        &\le&
            \P \left\{\eta \ge  \z/\lambda_{max}(V_{lk}^{1/2} \B{l} V_{lk}^{1/2})\right\}\\
        &\le&
            \P \left\{\eta \ge  \z[2(1+\delta) (1 + u_0^{-(k-l)})]^{-1}\right\},\\
   \P \left\{2 \LL ( \W{k}, \mmle{k}, \mmle{l} ) \ge  \z\right\}
   &\le&
     \P \left\{\eta \ge  \z/\lambda_{max}(V_{lk}^{1/2} \B{k} V_{lk}^{1/2})\right\}\\
   &\le&
      \P \left\{\eta \ge  \z[2(1+\delta) (1 + u^{(k-l)})]^{-1}\right\}.
 \end{eqnarray*}
\end{proof}

\begin{lemma} \label{expmoment corol_new}
Under the conditions of the preceding lemma for any \( \mu_0 < t_0^{-1} \), or
\( \mu_1 < t_1^{-1} \) respectively, the exponential moments are bounded:
\begin{eqnarray*}\label{expmoment_new}
\EE \exp \{ \mu_0 \LL(\W{l}, \mmle{l}, \mmle{k}) \}
        &\le&
        [1 - \mu_0 t_0]^{-p/2}\\
        \EE \exp \{ \mu_1 \LL(\W{k}, \mmle{k}, \mmle{l}) \}
        &\le &
        [1 - \mu_1 t_1]^{-p/2},
\end{eqnarray*}
where \( t_0=2(1+\delta) (1 + u_0^{-(k-l)}) \) and \( t_1 =2(1+\delta) (1 + u^{(k-l)}) \).
\end{lemma}
\begin{proof}
The proof of lemma is similar to the proof of Corollary~\ref{param. risk bounds}. The bounds \eqref{lambda max l} and \eqref{lambda max k} imply the following bounds for the corresponding moment generating functions:
\begin{eqnarray*}
\EE \exp \{ \mu \LL(\W{l}, \mmle{l}, \mmle{k}) \}
        &=& \prod_{j=1}^{p} \EE \exp \{ \frac{\mu}{2} \lambda_{j}(V_{lk}^{1/2} \B{l} V_{lk}^{1/2})
            \bar\varepsilon_{j}^2\}\\
        &=&    \prod_{j=1}^{p} [1- \mu  \lambda_{j}(V_{lk}^{1/2} \B{l} V_{lk}^{1/2}) ]^{-1/2}\\
        &\le&
        [1 - \mu \lambda_{max}(V_{lk}^{1/2} \B{l} V_{lk}^{1/2}) ]^{-p/2}\\
          &\le&
        [1 - 2\mu (1+\delta) (1 + u_0^{-(k-l)})]^{-p/2}\;,\\
\EE \exp \{ \mu \LL(\W{k}, \mmle{k}, \mmle{l}) \}
        &\le &
        [1 - \mu \lambda_{max}(V_{lk}^{1/2} \B{k} V_{lk}^{1/2})]^{-p/2}\\
         &\le & [1 - 2\mu (1+\delta) (1 + u^{(k-l)})]^{-p/2}.
\end{eqnarray*}
\end{proof}
\begin{lemma}\label{polmoment corol_new}
 Under the conditions of the preceding lemma it holds that:
\begin{eqnarray*}
    \EE| 2\LL(\W{l}, \mmle{l}, \mmle{k}) |^r &\le& 2^r C(p,r) (1+\delta)^r (1 + u_0^{-(k-l)})^r,\\
    \EE| 2\LL(\W{k}, \mmle{k}, \mmle{l}) |^r &\le& 2^r C(p,r) (1+\delta)^r (1 + u^{(k-l)})^r,
\end{eqnarray*}
where
\begin{equation*}
     C(p,r) = \EE|\chi^2_p|^r = 2^r \frac{\Gamma (r+ \frac{p}{2} )}{\Gamma(\frac{p}{2})}.
\end{equation*}
\end{lemma}
\begin{proof} Integration by parts and Lemma \ref{Tlk large div} yield for the second assertion
\begin{eqnarray*}
  \EE| 2\LL(\W{k}, \mmle{k}, \mmle{l}) |^r
    &=&
    r \int_{0}^{\infty}  \P \left\{2 \LL ( \W{k}, \mmle{k}, \mmle{l} ) \ge   \z\right\} \z^{r-1} \dd \z \\
   &\le &
   r \int_{0}^{\infty}    \P \left\{\eta \ge  \z \left[2(1+\delta)
        (1 + u^{(k-l)})\right]^{-1}  \right\} \z^{r-1} \dd \z \\
   &=& 2^r  (1+\delta)^r (1 + u^{(k-l)})^r \E|\eta|^r,
\end{eqnarray*}
where \( \eta \sim  \chi^2_p\). The first assertion is proved similarly.
\end{proof}
\begin{theorem} (The theoretical choice of the critical values)
\label{upper bound}
\par Assume \( (\mathfrak{D}) \), \( (\mathfrak{Loc}) \), \( (\mathfrak{B}) \) and \( (\cc{W}) \).
The adaptive procedure \eqref{adaptind} in the considered set-up is well-defined in the sense that the choice of the critical values
\begin{equation}
    \z_k = \frac{4}{\mu} \left\{ r(K-k) \log u + \log{(K/\alpha)} - \frac{p}{4} \log(1-4\mu )
            - \log(1-u^{-r}) + \bar{C}(p,r) \right\} \label{zk}
\end{equation}
provides the conditions \eqref{PC} for all \( k\le K \). Here \( \bar{C}(p,r) =\log\left\{ \frac{2^{2r} [\Gamma(2r + p/2) \Gamma(p/2)]^{1/2}}{\Gamma(r + p/2)} \right \} \) and \( \mu \in (0, 1/4) \). In particular,
\begin{equation}
    \EE_{0, \Sigma} |(\mmle{K} - \aadapest)^{\T}
        \B{K} (\mmle{K} - \aadapest)|^{r}
        \leq \alpha  C(p,r).
 \end{equation}
\end{theorem}
\begin{proof} 
The risk corresponding to the adaptive estimator can be represented as a sum of risks of
the false alarms at each step of the procedure:
\begin{equation*}
\EE_{0, \Sigma} |(\mmle{k} - \aadapest_k )^{\T} \B{k} (\mmle{k} - \aadapest_k ) |^r
    =
    \sum_{m=1}^{k-1} \EE_{0, \Sigma} |(\mmle{k} - \mmle{m} )^{\T} \B{k} (\mmle{k} - \mmle{m} ) |^r \I{\{ \aadapest_{k} = \mmle{m}\}}.
\end{equation*}
By the definition of the last accepted estimator \( \aadapest_{k} \)
the event \( \{ \aadapest_{k} = \mmle{m} \} \) with \( m =1, \ldots, k-1 \)
occurs if for some \( l = 1, \ldots ,m \) the statistic \( T_{l, m+1} >\z_l \). Thus
\begin{equation*}
    \{ \aadapest_{k} = \mmle{m} \} \subseteq \bigcup_{l=1}^{m} \{ T_{l, m+1} > \z_l \}.
\end{equation*}
It holds also that for any positive \( \mu \)
\begin{eqnarray*}
    \I{\{ T_{l, m+1} > \z_l \}}
        &=&
         \I{\{ 2 \LL(\W{l}, \mmle{l}, \mmle{m+1}) - \z_l > 0\}}\\
        &\le & \exp \{ \frac{\mu}{2} \LL(\W{l}, \mmle{l}, \mmle{m+1}) - \frac{\mu}{4} \z_l\}.
\end{eqnarray*}
Application of this simple fact and the Cauchy-Schwarz inequality implies for \( m = 1, \ldots, k-1 \) the following bound:
\begin{eqnarray*}
 && \EE_{0, \Sigma} |(\mmle{k} - \mmle{m} )^{\T} \B{k} (\mmle{k} - \mmle{m} ) |^r \I{\{ \aadapest_{k} = \mmle{m} \}}\\
   &=& \EE_{0, \Sigma} |2 \LL (\W{k}, \mmle{k}, \mmle{m})|^r \I{\{ \aadapest_{k} = \mmle{m}\}} \\
  &\le & \sum_{l=1}^{m} e^{-\frac{\mu}{4}\z_l} \EE_{0, \Sigma} \left[|2 \LL (\W{k}, \mmle{k}, \mmle{m})|^r
    \exp{\{ \frac{\mu}{2} \LL(\W{l}, \mmle{l}, \mmle{m+1}) \}}\right]\\
    &\le & \sum_{l=1}^{m} e^{-\frac{\mu}{4}\z_l}
        \left\{  \EE_{0, \Sigma} \left[ |2 \LL (\W{k}, \mmle{k}, \mmle{m})|^{2r} \right]
          \right\}^{1/2}
    \left\{ \EE_{0, \Sigma} \left[ \exp{\{ \mu \LL(\W{l}, \mmle{l}, \mmle{m+1}) \}} \right]\right\}^{1/2}.
\end{eqnarray*}
By Lemma \ref{expmoment corol_new} with \( \delta = 0 \)
\begin{equation*}
    \EE_{0, \Sigma} \left[ \exp{\{ \mu \LL(\W{l}, \mmle{l}, \mmle{m+1}) \}} \right]
        < ( 1- 4\mu)^{-p/2}.
\end{equation*}
This together with the bound from Lemma \ref{polmoment corol_new} gives
\begin{eqnarray*}
   && \EE_{0, \Sigma} |(\mmle{k} - \aadapest_k )^{\T} \B{k} (\mmle{k} - \aadapest_k ) |^r \\
  &\le & 2^r \sqrt{C(p,2r)}   ( 1- 4\mu)^{-p/4}
   \sum_{m=1}^{k-1} \sum_{l=1}^{m} e^{-\frac{\mu}{4}\z_l}(1+u^{(k-m)})^r\\
   &=& 2^r \sqrt{C(p,2r)}  ( 1- 4\mu)^{-p/4}
        \sum_{l=1}^{k-1} e^{-\frac{\mu}{4}\z_l} \sum_{m=l}^{k-1}(1+u^{(k-m)})^r\\
   &\le & 2^{2r} \sqrt{C(p,2r)}   ( 1- 4\mu)^{-p/4} (1-u^{-r})^{-1}
        \sum_{l=1}^{k-1} e^{-\frac{\mu}{4}\z_l} u^{r(k-l)},
\end{eqnarray*}
because \( -(k-l) < -(m-l) \) and
\begin{eqnarray*}
    \sum_{m=l}^{k-1}(1+u^{(k-m)})^r  &=& u^{r(k-l)} \sum_{m=l}^{k-1}(u^{-(k-l)}+u^{-(m-l)})^r\\
    &<& 2^r u^{r(k-l)} \sum_{m=l}^{k-1} u^{-r(m-l)}\\
    &<& 2^r u^{r(k-l)} (1-u^{-r})^{-1}.
    \end{eqnarray*}
Since \( u^{r(k-l)} \le u^{r(K-l)} \) for any \( l<k\le K \) the choice
\begin{equation*}
    \z_l = \frac{4}{\mu} \left\{ r(K-l) \log u + \log{(K/\alpha)} - \frac{p}{4} \log(1-4\mu )
    - \log(1-u^{-r}) + \bar{C}(p,r) \right\}
\end{equation*}
with
\begin{equation*}
    \bar{C}(p,r) =\log\left\{ \frac{2^{2r} [\Gamma(2r + p/2) \Gamma(p/2)]^{1/2}}{\Gamma(r + p/2)} \right \}
\end{equation*}
provides the required bound
\begin{equation*}
    \EE_{0,\Sigma} |(\mmle{l} - \aadapest_{l})^{\T}
        \B{l} (\mmle{l} - \aadapest_{l})|^{r}
        \leq \alpha  C(p,r)\; \; \;
        \text{for all}\;\; l=2, \ldots, K.
 \end{equation*}
\end{proof}

\subsection{Quality of estimation in the nearly parametric case:\\ small modeling bias and propagation property}
\label{nearly parametric case}

The critical values \( \z_1, \ldots , \z_{K-1} \) were selected by the propagation conditions \eqref{PC} under the hypothesis of homogeneity of the theta's with a probably misspecified error distribution, i.e. under the measure \( \norm{\tta}{\Sigma} \). Now \( \bbpf{1} \approx \cdots \approx \bbpf{k} \approx \tta\) up to some \( k\leq K \) and the covariance matrix is \( \Sigma_0 \). The aim is to formalize the meaning of ``\( \approx \)'' and to justify the use of the critical values in this situation. For this purposes we will take into account the discrepancy between the joint distributions of the linear estimators \( \mmle{1}, \ldots , \mmle{k} \) for \( k = 1, \ldots , K \) under the null (homogeneity) hypothesis corresponding to the distributions with mean \( \bbpf{1} = \cdots = \bbpf{k} = \tta\) and ``wrong'' covariance matrix \( \Sigma \) and in the general situation (under the alternative) with \( \bbpf{1} \ne \cdots \ne \bbpf{k} \) and covariance matrix \( \Sigma_0 \). Denote the expectations w.r.t. these measures by \(\EE_{\tta, \Sigma} := \EE_{k,\tta, \Sigma} \) and \( \EE_{\ff, \Sigma_0} := \EE_{k,\ff, \Sigma_0} \) respectively. Denote a \( p\times k \) matrix of the first \( k \) estimators by
 \begin{equation*}
    \tilde\TTa_k   \eqdef   (\mmle{1}, \ldots, \mmle{k}).
 \end{equation*}
Its mean under the alternative (the matrix of the parameters minimizing the expected local log-likelihoods) is given by
 \begin{equation*}
    \TTa^*_k      \eqdef   \EE_{\ff, \Sigma_0} \tilde \TTa_k =(\bbpf{1}, \ldots, \bbpf{k}),
 \end{equation*}
 and the mean under the null (the ``true'' parameter in the parametric set-up) is:
\begin{equation*}
    \TTa_k       \eqdef   \EE_{\tta, \Sigma} \tilde \TTa_k= (\tta, \ldots, \tta).
\end{equation*}
Let \( A \otimes B \) stands for the Kronecker product of \( A \) and \( B \) defined as
\begin{equation*}
    A \otimes B =\left(
                   \begin{array}{cccc}
                     a_{11}B & a_{12}B & \cdots & a_{1n}B \\
                     a_{21}B & a_{22}B & \cdots & a_{2n}B \\
                     \cdot   & \cdot   & \cdots & \cdot \\
                     a_{m1}B & a_{m2}B & \cdots & a_{mn}B \\
                   \end{array}
                 \right).
\end{equation*}

Denote the \( pk \times pk \) covariance matrices of
\( \vec \tilde \TTa_k^{\T} = (\mmle{1}^{\T}, \ldots , \mmle{k}^{\T}) \in \RR^{pk}\) by
\begin{eqnarray}
 \SSigma_k    &\eqdef&  \Var_{\tta, \Sigma}[\vec \tilde \TTa_k] =\DD_k (J_k \otimes \Sigma) \DD_k^{\T}, \label{def SSigma}\\
\SSigma_{k,0}   &\eqdef&  \Var_{\ff, \Sigma_0} [\vec \tilde \TTa_k] =\DD_k (J_k \otimes \Sigma_0) \DD_k^{\T} \label{def SSigma0},
\end{eqnarray}
where the matrix \(J_k \) is a \( k \times k \) matrix with all its elements equal to \(1\) and the
 \( pk \times nk \) matrix \( \DD_k \) is defined as follows:
\begin{eqnarray}
 \DD_k    &\eqdef&    D_1 \oplus \cdots \oplus D_k = \diag(D_1, \ldots, D_k), \label{def DD_K}\nn
  D_l     &\eqdef&    \B{l}^{-1} \PPsi \W{l}, \;\;\; l = 1, \ldots ,k.
\end{eqnarray}
By Lemma \ref{simidefinitness of SSigma } from Section \ref{section:Auxiliary results} under Assumption \( \mathfrak{(S)}\) with the same \( \delta \), a relation similar to \( \mathfrak{(S)}\) holds for the covariance matrices \( \SSigma_{k}   \) and \( \SSigma_{k,0} \) of the linear estimators:
\begin{equation}\label{multi cond sigma}
  (1-\delta) \SSigma_{k} \preceq  \SSigma_{k,0} \preceq (1+\delta)  \SSigma_{k}\;,\;\; k \le K.
\end{equation}

\par Even though the moment generating function of \( \vec \tilde \TTa_K \) has a form corresponding to the multivariate normal distribution (see Lemma \ref{MGF for joint distribution} in Section \ref{section:Auxiliary results}) this representation makes sense only if \( \SSigma_K \) is nonsingular. Notice that \( \rank(J_K \otimes \Sigma) =n \). From \( J_K \otimes \Sigma \succeq 0\) it follows only that \( \SSigma_{K} \succeq 0\), similarly, \(  \SSigma_{K,0} \succeq 0 \). However, without any additional assumptions it is easy to show (see Lemma \ref{nonsingularity of SSigma rectang} in Section \ref{section:Auxiliary results}) that for rectangular kernels \( \SSigma_{K} \succ 0 \). On the other hand, due to \eqref{multi cond sigma}, it is enough to require nonsingularity only for the matrix \( \SSigma_K \) corresponding to the approximate model \eqref{PA}, and its choice belongs to a statistician. In what follows we assume that \( \SSigma_{K} \succ 0\).

\par Denote by \( \P_{\tta, \Sigma}^k = \norm{\vec\TTa_k}{\SSigma_k} \) and by \( \P_{\ff, \Sigma_0}^k = \norm{\vec\TTa^*_k}{\SSigma_{k,0}} \), \( k=1, \ldots, K \), the distributions of \( \vec \tilde \TTa_k \) under the null and under the alternative. Denote also the Radon-Nikodym derivative by
\begin{equation}\label{def Zk}
    Z_k \eqdef \frac{\dd \P_{\ff, \Sigma_0}^k}{\dd \P_{\tta, \Sigma}^k}.
\end{equation}
Then by Lemma \ref{KL for joint distr} from Section \ref{section:Auxiliary results} the Kullback-Leibler divergence between these measures has the following form:
\begin{eqnarray}
  & &2\KL(\P_{\ff, \Sigma_0}^k,\P_{\tta, \Sigma}^k) \eqdef 2\EE_{\ff, \Sigma_0}
  \log\bigg( \frac{\dd \P_{\ff, \Sigma_0}^k}{\dd \P_{\tta, \Sigma}^k}  \bigg) \label{KL joint}\nn
  &=&  \Delta(k)  +\log\bigg(\frac{\det \SSigma_k}{\det \SSigma_{k,0}} \bigg)
            + \tr (\SSigma_k^{-1} \SSigma_{k,0}) -pk,
\end{eqnarray}
where
\begin{eqnarray}
  b(k) &\eqdef& \vec \TTa^*_k - \vec \TTa_k \label{def_b(k)}, \\
  \Delta(k) & \eqdef & b(k)^{\T} \SSigma_k^{-1} b(k)  \label{def_Delta(k)}.
\end{eqnarray}

\par If there would be no ``noise misspecification'', i.e., if \( \delta \equiv 0 \) implying
\( \Sigma = \Sigma_0 \), then
\( \Delta(k) = b(k)^{\T} \SSigma_k^{-1} b(k) = 2 \KL(\P_{\ff, \Sigma}^k,\P_{\tta, \Sigma}^k)\). Therefore this quantity can be used to indicate the deviation between the mean values in the true
 \eqref{true model} and the approximate \eqref{PA} models. Clearly, under \( (\cc{W}) \) the quantity \( \Delta(k) \) grows with \( k \), so following the terminology suggested in \citeasnoun{SV}, we
introduce the \emph{small modeling bias condition}:
\begin{description}
\item[\( \bb{(SMB)} \)]
\emph{
Let there exist for some \( k \le K \) and some \( \tta \) a constant \( \Delta  \ge 0 \) such that
\begin{equation*}
  \Delta(k) \le  \Delta.
\end{equation*}
}
\end{description}
Monotonicity of \( \Delta(k) \) and Assumption \( (SMB) \) immediately imply that
\begin{equation*}
  \Delta(k') \le \Delta \;\;  \text {for all } \; k' \le k .
\end{equation*}
The conditions \eqref{multi cond sigma} yield \( - p k \delta \le  \tr (\SSigma_k^{-1} \SSigma_{k,0}) -p k \le  p k \delta \). Thus \eqref{KL joint} implies the bound for the Kullback-Leibler divergence in terms of \( \delta \):
\begin{equation}\label{bounds for KL}
    -\frac{pk}{2} \log(1+ \delta) + \frac{\Delta(k)}{2}  -\frac{ pk \delta}{2}
    \le
    \KL(\P_{\ff, \Sigma_0}^k,\P_{\tta, \Sigma}^k)
    \le
    -\frac{pk}{2} \log(1- \delta) + \frac{\Delta(k)}{2}  +\frac{ pk \delta}{2}.
\end{equation}
Moreover, as \( \delta \to 0+\)
\begin{equation}\label{bounds for KL asymptotics}
     \Delta(k) -2 pk\delta + o(\delta)
    \le
    2\KL(\P_{\ff, \Sigma_0}^k,\P_{\tta, \Sigma}^k)
    \le
    \Delta(k) + 2 pk\delta + o(\delta).
\end{equation}
This means that if for some \( k \) Assumption \( (SMB) \) is fulfilled and \( \delta = o\big( \frac{1}{K} \big)  \), then the Kullback-Leibler divergence between \( \P_{\tta, \Sigma}^k  \) and  \( \P_{\ff, \Sigma_0}^k  \) is bounded by a small constant.
\par Now one can state the crucial property for obtaining the final oracle result.

\begin{theorem}\label{Propagation result theorem}(Propagation property)
   \par Assume \( (\mathfrak{D}) \), \( (\mathfrak{Loc}) \), \( (\mathfrak{S}) \), \( (\cc{W}) \), \( (\mathfrak{B}) \) and \( ({PC}) \). Then for any \( k \le K \) the following upper bounds hold:
    \begin{eqnarray*}
    && \EE|(\mmle{k} - \tta)^{\T} \B{k} (\mmle{k} - \tta)|^{r/2}\\
       & \le &  (\EE|\chi^2_p|^r)^{1/2} (1+\delta)^{pk/4}(1-\delta)^{-3pk/4}
        \exp\left\{ \varphi(\delta)
        \frac{\Delta(k)}{2(1-\delta)}\right\},\\
       && \EE|(\mmle{k} - \aadapest_{k})^{\T} \B{k} (\mmle{k} - \aadapest_{k})|^{r/2}\\
       & \le & (\alpha \EE|\chi^2_p|^r)^{1/2} (1+\delta)^{pk/4}(1-\delta)^{-3pk/4}
        \exp\left\{ \varphi(\delta)
        \frac{\Delta(k)}{2(1-\delta)}\right\},
 \end{eqnarray*}
 where \(\varphi(\delta) \eqdef
  \begin{cases}
1 & \mathrm{for \; homogeneous \;errors,} \\
 \frac{2(1+\delta)}{(1-\delta)^2} -1  &\mathrm{otherwise}.
\end{cases} \)
\par Here \( \mmle{k} = \mmle{k}(x) \) is the QMLE defined by \eqref{MLE} and \( \aadapest_{k}(x) = \mmle{\min\{k, \adapind\}}(x) \) is the adaptive estimator at the \( k \)th step of the procedure.
\end{theorem}
\begin{remark}
Bounds \eqref{bound for Zk homog} and \eqref{bound for Zk} below give a kind of condition on the relative error in the noise misspecification.
As \( \delta \to 0+ \) it holds for every \( k \le K \)
\begin{equation*}
    \varphi(\delta)
        \frac{\Delta(k)}{1+\delta} -2pk\delta +o(\delta) \le
        \log \EE_{\tta, \Sigma} [Z_k^2]\le
       \varphi(\delta)
        \frac{\Delta(k)}{1-\delta} +2pk\delta + o(\delta),
\end{equation*}
where \( Z_k \) is defined by \eqref{def Zk}.

\par This  bound implies, up to the additive constant
\(  \log\big( \alpha \EE|\chi^2_p|^r \big)/2\), the same asymptotic behavior for the logarithm of the risk of adaptive estimator at each step of the procedure. Because by \( (SMB) \) the quantity \( \Delta(k) \) is bounded by a small constant and \( K \) is of order \( \log n \), \( \EE_{\tta, \Sigma} [Z_k^2] \) is small if \( \delta = o \big( \frac{1}{\log n} \big) \). This means that for the case when \( \Sigma \) is an estimator for \( \Sigma_0 \), only logarithmic in sample size accuracy is needed. This observation is of particular importance, since it is known from~\citeasnoun{Spokoiny variance} that the rate \( n^{-1/2} \) of variance estimation is achievable only for dimensions \( d \le 8 \) over classes of functions with bounded second derivative.
\end{remark}
\begin{remark}
The propagation property guaranties that the adaptive procedure does not stop with high probability while \( \Delta(k) \) is small, i.e. under \( (SMB) \), and if the relative error \( \delta \) in the noise is sufficiently small.
\end{remark}

\begin{proof} 
\par
Notice that for any nonnegative measurable function \( g = g(\tilde \TTa_k) \) the Cauchy-Schwarz inequality implies
\begin{equation}\label{changing measure}
    \EE_{\ff, \Sigma_0} [g] = \EE_{\tta, \Sigma} [g Z_k]
    \le \big(\EE_{\tta, \Sigma} [g^2] \big)^{1/2} \big(\EE_{\tta, \Sigma} [Z_k^2] \big)^{1/2}
\end{equation}
with the Radon-Nikodym derivative
\begin{equation*}
    Z_k = \frac{\dd \P^k_{\ff, \Sigma_0}}{\dd \P^k_{\tta, \Sigma}}.
\end{equation*}
Taking \( g=|(\mmle{k} - \tta)^{\T} \B{k} (\mmle{k} - \tta)|^{r/2} \) one gets the first assertion applying ``the parametric risk bound'' with \( \delta=0 \) from \eqref{param. risk bounds polinom}:
\begin{eqnarray*}
  \EE[g] &\le& \big(\EE_{\tta, \Sigma} |(\mmle{k} - \tta)^{\T} \B{k} (\mmle{k} - \tta)|^{r} \big)^{1/2} \big(\EE_{\tta, \Sigma} [Z_k^2] \big)^{1/2} \\
   &=& \big(\EE_{\tta, \Sigma} |2\LL(\W{k}, \mmle{k}, \tta )|^r \big)^{1/2} \big(\EE_{\tta, \Sigma} [Z_k^2] \big)^{1/2}\\
   &\le&  (\EE|\chi^2_p|^r)^{1/2} \big(\EE_{\tta, \Sigma} [Z_k^2] \big)^{1/2}.
\end{eqnarray*}
The second assertion is treated similarly by applying the pivotality property (Lemma~\ref{Pivotality property}) and the propagation conditions \eqref{PC}.

To calculate \( \EE_{\tta, \Sigma} [Z_k^2] \) let us consider \( \log Z_k \) given by
\begin{eqnarray*}
    \log\big(Z_k(y)\big) =
        \frac{1}{2} \log\bigg(\frac{\det \SSigma_k}{\det \SSigma_{k,0}} \bigg)
        &-&\frac{1}{2} \| \SSigma_{k,0}^{-1/2} (y - \vec \TTa^*_k) \|^2  \nn
        &+& \frac{1}{2}\| \SSigma_k^{-1/2} (y - \vec \TTa_k) \|^2
\end{eqnarray*}
as a function of \( \vec \TTa^*_k \). Application of the Taylor expansion at the point \(  \vec \TTa_k \) yields
\begin{eqnarray*}
  2\log Z_k &=& \log \frac{\det \SSigma_k}{\det \SSigma_{k,0}}
                 - \| \SSigma_{k,0}^{-1/2} (y - \vec \TTa_k) \|^2
                 + \| \SSigma_k^{-1/2} (y - \vec \TTa_k) \|^2 \\
            &+&  2 b(k)^{\T} \SSigma_{k,0}^{-1}(y - \vec \TTa_k)
                 - b(k)^{\T} \SSigma_{k,0}^{-1} b(k).
\end{eqnarray*}
With \( \xi \sim \norm{0}{I_{pk}} \) the second moment of the Radon-Nikodym derivative under the null hypothesis reads as follows:
\begin{eqnarray}
   & & \EE_{\tta, \Sigma} [Z_k^2] \nn
   &=&     \frac{\det \SSigma_k}{\det \SSigma_{k,0}}
            \exp\{ - b(k)^{\T} \SSigma_{k,0}^{-1} b(k)\}
            \EE \exp \{ -\| \SSigma_{k,0}^{-1/2} \SSigma_k^{1/2} \xi \|^2
                        + \| \xi \|^2
                        + 2 b(k)^{\T} \SSigma_{k,0}^{-1} \SSigma_k^{1/2} \xi \}\nn
   &=&    \frac{\det \SSigma_k}{\det \SSigma_{k,0}}
                \big[\det \big(2 \SSigma_k^{1/2}\SSigma_{k,0}^{-1}\SSigma_k^{1/2} - I_{pk} \big )\big]^{-1/2}\nn
   &\times&       \exp\{ 2 b(k)^{\T} \SSigma_{k,0}^{-1} \SSigma_k^{1/2} \big(2 \SSigma_k^{1/2} \SSigma_{k,0}^{-1}  \SSigma_k^{1/2} -I_{pk}\big)^{-1} \SSigma_k^{1/2} \SSigma_{k,0}^{-1} b(k)
           - b(k)^{\T} \SSigma_{k,0}^{-1} b(k)\} \nn
   &=&    \frac{\det \SSigma_k}{\det \SSigma_{k,0}}
                \big[ \prod_{j=1}^{pk}
                    \{ 2 \lambda_j(\SSigma_k^{1/2}\SSigma_{k,0}^{-1}\SSigma_k^{1/2}) -1\} \big]^{-1/2} \label{exp pokazatel}\\
   &\times& \nonumber \exp
                \{ b(k)^{\T} \SSigma_{k,0}^{-1/2}
                        \big[ 2 \SSigma_{k,0}^{-1/2} \SSigma_k^{1/2}
                                \big(2 \SSigma_k^{1/2} \SSigma_{k,0}^{-1}  \SSigma_k^{1/2}
                                -I_{pk}\big)^{-1} \SSigma_k^{1/2} \SSigma_{k,0}^{-1/2}  - I_{pk}\big]
                \SSigma_{k,0}^{-1/2} b(k)
                \} .
\end{eqnarray}
To estimate the obtained expression in terms of the level of noise misspecification \( \delta \)
 notice that the condition \eqref{multi cond sigma} implies
\begin{equation*}
    \left( \frac{1}{1+\delta} \right)^{pk}
        \le
            \frac{\det \SSigma_k}{\det \SSigma_{k,0}}
        \le
    \left( \frac{1}{1-\delta} \right)^{pk},
\end{equation*}
\begin{equation*}
     \left(\frac{1-\delta}{1+\delta} \right)^{\frac{pk}{2}}
        \le
            \big[ \prod_{j=1}^{pk}
                    \{ 2 \lambda_j(\SSigma_k^{1/2}\SSigma_{k,0}^{-1}\SSigma_k^{1/2}) -1\} \big]^{-1/2}
        \le
     \left( \frac{1+\delta}{1-\delta} \right)^{\frac{pk}{2}}.
\end{equation*}
\begin{equation*}
    \frac{1-\delta}{1+\delta} I_{pk}
        \preceq
            \left( 2 \SSigma_k^{1/2}\SSigma_{k,0}^{-1}\SSigma_k^{1/2} - I_{pk}  \right)^{-1}
        \preceq
    \frac{1+\delta}{1-\delta} I_{pk}.
\end{equation*}
Therefore the quantity in the exponent in \eqref{exp pokazatel} is bounded by:
\begin{eqnarray*}
 &  &  \;\;   \left( 2 \frac{1-\delta}{(1+\delta)^2 } -1 \right)  b(k)^{\T} \SSigma_{k,0}^{-1} b(k)\\
 && \le    b(k)^{\T} \SSigma_{k,0}^{-1/2}
                        \big[ 2 \SSigma_{k,0}^{-1/2} \SSigma_k^{1/2}
                                \big(2 \SSigma_k^{1/2} \SSigma_{k,0}^{-1}  \SSigma_k^{1/2}
                                -I_{pk}\big)^{-1} \SSigma_k^{1/2} \SSigma_{k,0}^{-1/2}  - I_{pk} \big]
                \SSigma_{k,0}^{-1/2} b(k) \\
 & & \le   \left( 2 \frac{1+\delta}{(1-\delta)^2} -1 \right) b(k)^{\T} \SSigma_{k,0}^{-1} b(k).
\end{eqnarray*}
Moreover,
\begin{eqnarray*}
&& \frac{\Delta(k)}{1+\delta}
    =     \frac{1}{1+\delta} b(k)^{\T} \SSigma_k^{-1} b(k) \\
&& \le b(k)^{\T} \SSigma_{k,0}^{-1} b(k)\\
&&  \le \frac{1}{1-\delta} b(k)^{\T} \SSigma_k^{-1} b(k)
    = \frac{\Delta(k)}{1-\delta}.
\end{eqnarray*}
Finally,
\begin{eqnarray}\label{bound for Zk}
  && 
        \left( \frac{1- \delta}{(1+\delta)^3} \right)^{\frac{pk}{2} }
        \exp\left\{ \left(  \frac{2(1-\delta)}{(1+\delta)^2} -1 \right) \frac{\Delta(k)}{1+\delta}\right\}\nn
   &&\le  \EE_{\tta, \Sigma } [Z_k^2]
        \le 
            \left( \frac{1+ \delta}{(1-\delta)^3} \right)^{\frac{pk}{2}}
            \exp\left\{ \left( \frac{2(1+\delta)}{(1-\delta)^2} -1 \right) \frac{\Delta(k)}{1-\delta}\right\}.
\end{eqnarray}
In the \emph{case of homogeneous errors} the  expression for \( \log Z_k \) reads as
\begin{eqnarray*}
  \log Z_k  &=&  p k \log \big(\frac{\sigma}{\sigma_0} \big)
                    + \frac{1}{2} \big(\frac{1}{ \sigma^2}-\frac{1}{ \sigma^2_0}\big)
                                            \| \VV_k^{-1/2} (y - \vec \TTa_k) \|^2 \\
            &+& \frac{1}{ \sigma^2_0} b(k)^{\T} \VV_k^{-1} (y-\vec \TTa_k)
                    - \frac{1}{2 \sigma^2_0} b(k)^{\T} \VV_k^{-1} b(k),
\end{eqnarray*}
implying
\begin{equation*}
    \EE_{\tta, \sigma} [Z_k^2] = \left( \frac{\sigma^2}{\sigma_0^2} \right)^{pk}
        \left( \frac{\sigma_0^2}{2 \sigma^2 -\sigma_0^2}  \right)^{\frac{pk}{2}}
        \exp \left\{  \frac{b(k)^{\T} \VV_k^{-1} b(k) }{ 2 \sigma^2 -\sigma_0^2 } \right\}.
\end{equation*}
By the condition \( (\mathfrak{S})  \)
\begin{eqnarray}\label{bound for Zk homog}
&&
                \left( \frac{1-\delta}{(1+\delta)^3}  \right)^{\frac{pk}{2}}
                \exp \left\{ \frac{\Delta_1(k)}{\sigma^2 (1+\delta)}   \right\}\nn
 \le \EE_{\tta, \sigma} [Z_k^2]
    &\le&  
                \left( \frac{1+\delta}{(1-\delta)^3}  \right)^{\frac{pk}{2}}
                \exp \left\{ \frac{\Delta_1(k)}{\sigma^2 (1-\delta)}   \right\} ,
\end{eqnarray}
where \( p \) is the dimension of the parameter set and \( k \) is the degree of the localization.
\end{proof}

\subsection{Quality of estimation in the nonparametric case: the oracle result}
\label{subsection:oracle result}
Define the {\it oracle index} as the largest index \( k \le K \) such that the small modeling bias condition \( (SMB) \) holds, that is
\begin{equation}\label{oracle index}
    k^* \eqdef \max \{ k \le K : \Delta(k) \le \Delta \}.
\end{equation}

\begin{theorem}\label{oracle result}
Let \( \Delta(1) \le \Delta \), i.e., the first estimator is always accepted by the testing procedure.
Let \( k^* \) be the oracle index. Then under the conditions \( (\mathfrak{D}) \), \( (\mathfrak{Loc}) \), \( (\mathfrak{S}) \), \( (\cc{W}) \), \( (\mathfrak{B}) \)  the risk between the adaptive estimator and the oracle is bounded by the following expression:
\begin{eqnarray}
       && \EE|(\mmle{k^*} - \aadapest)^{\T} \B{k^*} (\mmle{k^*} - \aadapest)|^{r/2}\\
       & \le &   \z_{k^*}^{r/2} +
       (\alpha \EE|\chi^2_p|^r)^{1/2} (1+\delta)^{pk^*/4}(1-\delta)^{-3pk^*/4}
        \exp\left\{ \varphi(\delta)
        \frac{\Delta}{2(1-\delta)}\right\} ,\nonumber
\end{eqnarray}
where \( \varphi(\delta) \) is as in Theorem \ref{Propagation result theorem}.
\end{theorem}

\begin{proof}
By the definition of the adaptive estimator \( \aadapest = \mmle{ \adapind}  \). Because the events \( \{ \adapind \le k^*\} \) and \( \{ \adapind > k^*\} \) are disjunct one can write
\begin{eqnarray*}
       && \EE|(\mmle{k^*} - \aadapest)^{\T} \B{k^*} (\mmle{k^*} - \aadapest)|^{r/2}\\
&=& \EE|(\mmle{k^*} - \mmle{ \adapind})^{\T} \B{k^*} (\mmle{k^*} - \mmle{ \adapind})|^{r/2} \ind\{ \adapind \le k^* \} \\
&+&
\EE|(\mmle{k^*} - \mmle{ \adapind})^{\T} \B{k^*} (\mmle{k^*} - \mmle{ \adapind})|^{r/2} \ind \{ \adapind > k^*\}.
\end{eqnarray*}
If \(  \adapind \le k^* \) then \( \aadapest_{k^*} \eqdef \mmle{\min \{ k^*, \adapind \} } = \mmle{ \adapind} \). Thus to bound the first summand it is enough to apply Theorem~\ref{Propagation result theorem} with \( k = k^* \).
\par To bound the second expectation, i.e. to bound fluctuations of the adaptive estimator \( \aadapest \) at the steps of the procedure for which the SMB condition is not fulfilled anymore, just notice that for \( \adapind > k^* \) the quadratic form coincides with the test statistic \( T_{k^*,\adapind} \)
\begin{eqnarray*}
  && (\mmle{k^*} - \aadapest)^{\T} \B{k^*} (\mmle{k^*} - \aadapest) \\
  &=& (\mmle{k^*} - \mmle{\adapind})^{\T} \B{k^*} (\mmle{k^*} - \mmle{\adapind})
    \eqdef T_{k^*,\adapind}.
\end{eqnarray*}
But the index \( \adapind \) was accepted, this means that \( T_{l,\adapind} \le \z_l \) for all \( l < \adapind \) and therefore for \( l=k^* \). Thus
\begin{equation*}
    \EE |(\mmle{k^*} - \aadapest)^{\T} \B{k^*} (\mmle{k^*} - \aadapest)|^{r/2} \ind \{ \adapind > k^*\} \le \z_{k^*}^{r/2}.
\end{equation*}
\end{proof}

\subsection{Oracle risk bounds for estimators of the regression function and its derivatives}
\label{subsection:estimators of the regression function}
Theorem \ref{oracle result} provides an oracle risk bound for the adaptive estimator
\( \aadapest(x) = \mmle{\adapind} (x)\) of the parameter vector \( \tta(x) \in \RRp \)
of the finite-rank expansion from the method of local approximation, see Section
~\ref{subsec:Method of local approximation} for details. This is equivalent to the estimation of the parameter of the local linear fit of the form \( \PPsi^{\T} \tta \) at the point \( x \) to the model~\eqref{true model} under misspecification together with the adaptive choice of the degree of localization (of the bandwidth). If the basis is polynomial and the regression function \( f(\cdot) \) is sufficiently smooth in
a neighborhood of \( x \), then \( \aadapest(x) \) is the adaptive local polynomial estimator \( LP^{ad}(p-1) \) of the vector \( ( f^{(0)} (x) , \ldots , f^{(p-1)} (x))^{\T} \) of the values of
\( f \) and its derivatives (if they exist) at the reference point \( x \in \RRd \) under the model misspecification.

\par Now we are going to obtain a similar oracle result for the components of the vector \( \aadapest(x) \), particularly for \(  \bb e_j^{\T} \aadapest(x) \), \( j = 1, \ldots, p \), where
\( \bb e_j = (0, \ldots, 1 , \ldots, 0)^{\T} \) is the \( j \)th canonical basis vector
in \( \RRp \). As a corollary of this general result in the case of the polynomial basis we get an oracle risk bound for  \( LP^{ad}(p-1) \) estimators
of the function \( f \) and its derivatives at the point \( x \).
\par Denote the \( LP_k(p-1) \) estimator of \( f^{(j-1)}(x) \) corresponding to the \( k \)th scale by
\begin{eqnarray}\label{def of estimators of f and its derivatives}
  \flej{j-1}{k}{x} &=& e_j^{\T}\mmle{k}(x), \; j= 1, \ldots,p, \\ \nonumber
  \fle{k}{x} &=& \flej{0}{k}{x} = e_1^{\T}\mmle{k}(x).
\end{eqnarray}
Then the adaptive local polynomial estimators are defined as follows:
\begin{eqnarray}\label{def of adaptive estimators of f and its derivatives}
  \adaplpest^{(j-1)}(x) &=& e_j^{\T}\aadapest(x), \; j= 1, \ldots,p, \\ \nonumber
  \adaplpest(x) &=&  e_1^{\T}\aadapest(x).
\end{eqnarray}
Similarly, the adaptive estimators of the function \( f \) and its derivatives corresponding to the \( k \)th step of the procedure are given by
\begin{equation}\label{def of k-th adaptive estimators of f and its deriv}
    \adaplpest_k^{(j-1)}(x) \eqdef e_j^{\T}\aadapest_k(x), \; j= 1, \ldots,p.
\end{equation}
Thus, if the basis is polynomial, the estimator \( \adaplpest(x) \eqdef \adaplpest^{(0)}(x) \) is the \( LP^{ad}(p-1) \)
estimator of the value \( f(x) \), and \( \adaplpest^{(j-1)}(x) \) with \( j = 2, \ldots, p \)
are, correspondingly, the \( LP^{ad}(p-1) \) estimators of the values of its derivatives.
We will use the polynomial basis to obtain the rate of convergence, but it should be stressed that the results of Theorems \ref{oracle result} and \ref{oracle result componentwise} hold for any basis satisfying the conditions of the theorems.

\par We need the following assumptions:
\begin{description}
\item[\( \bb{\mathfrak{(S1)}} \)]
\emph{There exist  \( 0 < \s_{min} \le  \s_{max} < \infty  \) such that for any \( i=1, \ldots, n \) the variance of the errors in the ``approximate'' model \eqref{PA} is uniformly bounded:}
\begin{equation*}
   \s^2_{min}\le  \s^2_{i} \le \s^2_{max}.
\end{equation*}

\item[\( \bb{\mathfrak{(Lp1^{\mathrm{d} })}} \)]
\emph{Let assumption \( \mathfrak{(S1)} \) be satisfied. There exists a number \( \Lambda_0>0  \) such that for any \( k=1, \ldots, K \) the smallest eigenvalue fulfills \( \lambda_p(\B{k}) \ge  nh_k^d\Lambda_0 \s^{-2}_{max}\) for \( n \) sufficiently large.}
\end{description}
Then, because \( \B{k} \succ 0 \), for any \( k= 1, \ldots K \) we have
\begin{equation}\label{bound for Bk via the smallest eigenvalue in Rd}
    \gamma^{\T} \B{k}^{-1} \gamma \le \frac{\s^2_{max}}{nh_k^d\Lambda_0} \| \gamma \|^2
\end{equation}
for any \( \gamma \in \RRp \), and we obtain the following lemma:
\begin{lemma}\label{bound for the componentwise differences}
Let \( \mathfrak{(S1)} \) and \( \mathfrak{(Lp1^{\mathrm{d} })} \) be satisfied. Then for any \( j = 1, \ldots, p \) and \( k, \, k'= 1, \ldots K \) the following upper bound holds:
\begin{equation*}
    \left( \frac{nh_k^d\Lambda_0}{\s^2_{max}} \right)^{1/2}
        |\bb e_j^{\T}\mmle{k} - \bb e_j^{\T}\mmle{k'}|
            \le
        \| \B{k}^{1/2} (\mmle{k} - \mmle{k'}) \|.
\end{equation*}
\end{lemma}
\begin{proof}
By \eqref{bound for Bk via the smallest eigenvalue in Rd} taking \( \gamma = \B{k}^{1/2} (\mmle{k} - \mmle{k'})\) we have
\begin{eqnarray*}
  |\bb e_j^{\T}\mmle{k} - \bb e_j^{\T}\mmle{k'}|^2 &\le & \| \mmle{k} -\mmle{k'} \|^2 \\
    &=&
        \|\B{k}^{-1/2} \B{k}^{1/2} (\mmle{k} - \mmle{k'}) \|^2 \\
    &\le &
        \frac{\s^2_{max}}{nh_k^d\Lambda_0}  \| \B{k}^{1/2} (\mmle{k} - \mmle{k'}) \|^2.
\end{eqnarray*}
\end{proof}

\par To obtain the ``componentwise'' oracle risk bounds we need to recheck the ``propagation property''. First, notice that the ``propagation conditions'' \eqref{PC} on the choice the critical values \( \z_1, \ldots, \z_{K-1} \) imply the similar bounds for the components \( \bb e_j^{\T} \aadapest_k(x) \). Recall that \( \aadapest_k \eqdef \mmle{\min\{ k, \adapind \}} \). Then, by \eqref{PC}, Lemma \ref{bound for the componentwise differences} and the pivotality property (Lemma \ref{Pivotality property}) we have the following simple observation:
\begin{lemma}\label{PC componentwise} Let \( \mathfrak{(S1)} \) and \( \mathfrak{(Lp1^{\mathrm{d} })} \) be satisfied. Under the propagation conditions~\( (PC) \) for any \( \tta \in \RRp \) and all \( k=2, \ldots, K \) we have:
    \begin{eqnarray*}
     \left( \frac{nh_{k}^d\Lambda_0}{\s^2_{max}} \right)^{r}
               \EE_{\tta, \Sigma} |\bb e_j^{\T}\mmle{k}(x) - \bb e_j^{\T}\aadapest_k(x)|^{2r}
  &\le&  \EE_{0, \Sigma} \| \B{k}^{1/2} (\mmle{k} - \aadapest_{k}) \|^{2r}\\
  &\le & \alpha  C(p,r).
   \end{eqnarray*}
Here \( \EE_{0, \Sigma} \) stands for the expectation w.r.t.
the measure \( \norm{0}{\Sigma} \) and \( C(p,r) = \EE|\chi^2_p|^r \).
\end{lemma}

\par As in the first parts of this chapter to make the notation shorter we will suppress the dependence on \( x \). To get the propagation property we study for \( k=1, \ldots, K \) the joint distributions of \( \bb e_j^{\T} \mmle{1} , \ldots , \bb e_j^{\T} \mmle{k}\), that is the distribution of \( \bb e_j^{\T} \tilde{\TTa}_k \), the \( j \)th row of the matrix \( \tilde{\TTa}_k \), under the null and under the alternative.
Obviously,
 \begin{equation*}
     \EE_{\ff, \Sigma_0} [\bb e_j^{\T}\tilde \TTa_k]  = \bb e_j^{\T} \TTa^*_k
        =(\bb e_j^{\T}\bbpf{1}, \ldots, \bb e_j^{\T} \bbpf{k}),
 \end{equation*}
 and the mean under the null (the true parameter in the parametric set-up) is:
\begin{equation*}
\EE_{\tta, \Sigma} [\bb e_j^{\T} \tilde \TTa_k]  =   \bb e_j^{\T} \TTa_k
            = (\bb e_j^{\T} \tta, \ldots,\bb e_j^{\T} \tta).
\end{equation*}
Recall that the matrices \( \SSigma_{k,0} \) and \( \SSigma_{k} \) have a block structure. Now, for instance, to study the estimator of the first coordinate of the ``best parametric fit'' vector (or of   \( f(x) \) in the case of the polynomial basis) we take the first elements of each block and so on.
Denote the \( k \times k \) covariance matrices of \( j \)th elements of the vectors
\( \mmle{1}, \ldots , \mmle{k} \) by
\begin{eqnarray}
    \SSigma_{k,j}
        &\eqdef& \big\{ \cov_{\tta, \Sigma}\big[\mle{l}{j},\mle{m}{j}   \big] \big\}_{1\le l\le m\le k}\nn
         &=&     \DD_{k,j} (J_k \otimes \Sigma) \DD_{k,j}^{\T}
                    \;\; \text{under the null}, \label{def SSigma j}\\
    \SSigma_{k,0,j}
        &\eqdef& \big\{ \cov_{\ff, \Sigma_0}\big[\mle{l}{j},\mle{m}{j}   \big] \big\}_{1\le l\le m\le k}\nn
         &=&     \DD_{k,j} (J_k \otimes \Sigma_0) \DD_{k,j}^{\T}
                \;\; \text{under the alternative,} \label{def SSigma 0 j}
\end{eqnarray}
where \(J_k \) is a \( k \times k \) matrix with all its elements equal to \(1\), and the \( k \times nk \) block diagonal matrices \( \DD_{k,j} \) is defined by
\begin{eqnarray}
 \DD_{k,j}    &\eqdef&   \bb e_j^{\T} D_1 \oplus \cdots \oplus\bb e_j^{\T} D_k , \label{def DD_K}
            = \big( I_k \otimes \bb e_j^{\T} \big) \DD_k\nn
  D_l     &\eqdef&    \B{l}^{-1} \PPsi \W{l}, \;\;\; l = 1, \ldots ,k.
\end{eqnarray}
Moreover, the following representation holds:
\begin{eqnarray}\label{SSigma_kj via SSigma_k}
  \SSigma_{k,j} &=& \big( I_k \otimes \bb e_j^{\T} \big) \DD_k
                        \big( J_k \otimes \Sigma  \big) \DD_k^{\T}
                            \big( I_k \otimes \bb e_j^{\T} \big)^{\T} \nn
  &=& \big( I_k \otimes \bb e_j \big)^{\T} \SSigma_k \big( I_k \otimes \bb e_j \big),
\end{eqnarray}
where \( \SSigma_k \) is defined by \eqref{def SSigma}. Similarly,
\begin{equation}\label{SSigma_k0j via SSigma_k0}
    \SSigma_{k,0,j} = \big( I_k \otimes \bb e_j \big)^{\T} \SSigma_{k,0} \big( I_k \otimes \bb e_j \big).
\end{equation}
Thus, the important relation \eqref{multi cond sigma} is preserved for \( \SSigma_{k,j} \) and
\( \SSigma_{k,0,j} \) obtained by picking the \( (j,j) \)th elements of each block of \( \SSigma_{k} \) and \( \SSigma_{k,0} \) respectively.

\par With usual notation \( \gamma^{(j)} \) for the \( j \)th component of \( \gamma \in \RR^k \), denote by
\begin{eqnarray}
  b_j(k) &\eqdef& (\bb e_j^{\T} (\bbpf{1} - \tta), \ldots,\bb e_j^{\T} (\bbpf{k} - \tta) )^{\T} \nn
  &=& ( (\bbpf{1} - \tta)^{(j)}, \ldots, (\bbpf{k} - \tta)^{(j)} )^{\T} \in \RR^k \label{def_bj(k)},  \\
  \Delta_j(k) & \eqdef & b_j(k)^{\T} \SSigma_{k,j}^{-1} \, b_j(k)  \label{def_Delta j(k)}.
\end{eqnarray}
\begin{theorem}\label{Propagation result componentwise theorem}(``Componentwise'' propagation property)
   \par Under the conditions \( (\mathfrak{D}) \), \( (\mathfrak{Loc}) \), \( (\mathfrak{S}) \), \( (\mathfrak{S1}) \), \( (PC) \), \( (\mathfrak{B}) \), \( (\cc{W}) \) and \( \mathfrak{(Lp1^{\mathrm{d} })} \) for any \( k \le K \) the following upper bound holds:
    \begin{eqnarray*}
      && \left( \frac{nh_{k}^d\Lambda_0}{\s^2_{max}} \right)^{r/2}
               \EE |\bb e_j^{\T}\mmle{k}(x) - \bb e_j^{\T}\aadapest_k(x)|^r\\
       & \le & (\alpha \EE|\chi^2_p|^r)^{1/2} (1+\delta)^{pk/4}(1-\delta)^{-3pk/4}
        \exp\left\{ \varphi(\delta)
        \frac{\Delta_j(k)}{2(1-\delta)}\right\},
 \end{eqnarray*}
 where \(\varphi(\delta) \eqdef
  \begin{cases}
1 & \mathrm{for \; homogeneous \;errors,} \\
 \frac{2(1+\delta)}{(1-\delta)^2} -1  &\mathrm{otherwise}.
\end{cases} \)
\end{theorem}
\begin{corollary}\label{corr:PP for pol bas}
Let the basis be polynomial. Then under the conditions of the preceding theorem the following upper bound holds:
\begin{eqnarray*}
      && \left( \frac{nh_{k}^d\Lambda_0}{\s^2_{max}} \right)^{r/2}
               \EE |\flej{j-1}{k}{x} - \adaplpest^{(j-1)}_k(x) |^r\\
       & \le & (\alpha \EE|\chi^2_p|^r)^{1/2} (1+\delta)^{pk/4}(1-\delta)^{-3pk/4}
        \exp\left\{ \varphi(\delta)
        \frac{\Delta_j(k)}{2(1-\delta)}\right\},
 \end{eqnarray*}
 with \( \varphi(\delta) \) as before.
\end{corollary}

\begin{proof}
The proof essentially follows the line of the proof of Theorem \ref{Propagation result theorem}.
If the distributions of \( \vec \tilde \TTa_k \) under the null and under the alternative were
Gaussian, then any subvector is also Gaussian. Denote by \( \P_{\tta, \Sigma}^{k,j} = \norm{(\bb e_j^{\T}\tta, \ldots, \bb e_j^{\T} \tta)^{\T}}{\SSigma_{k,j}} \) and by \( \P_{\ff, \Sigma_0}^{k,j} = \norm{(\bb e_j^{\T}\bbpf{1}, \ldots, \bb e_j^{\T} \bbpf{k})^{\T}}{\SSigma_{k,0,j}} \), \( k=1, \ldots, K \), the distributions of \( e_j^{\T} \tilde \TTa_k \) under the null and under the alternative.

\par By the Cauchy-Schwarz inequality and Lemma \ref{PC componentwise}
\begin{equation*}
    \left( \frac{nh_{k}^d\Lambda_0}{\s^2_{max}} \right)^{r/2}
               \EE |\bb e_j^{\T}\mmle{k}(x) - \bb e_j^{\T}\aadapest(x)|^r
  \le  (\alpha \EE|\chi^2_p|^r)^{1/2} \big( \EE_{\tta, \Sigma} [Z^2_{k,j}] \big)^{1/2}
\end{equation*}
with the Radon-Nikodym derivative given by
\begin{equation}\label{def Z_kj}
    Z_{k,j} \eqdef \frac{\dd \P_{\ff, \Sigma_0}^{k,j}}{\dd \P_{\tta, \Sigma}^{k,j}}.
\end{equation}
By inequalities \eqref{SSigma_kj via SSigma_k} and \eqref{SSigma_k0j via SSigma_k0} the analog of Assumption \( (\mathfrak{S}) \) is preserved for \( \SSigma_{k,0,j} \) and \( \SSigma_{k,j} \), that is, there exists \( \delta \in [0,1  ) \) such that
\begin{equation}\label{condition Sj}
 (1-\delta)   \SSigma_{k,j}\preceq  \SSigma_{k,0,j} \preceq (1+\delta) \SSigma_{k,j}
\end{equation}
for any \( k \le K \) and \( j=1, \ldots ,p \).
By the Taylor expansion at the point \( (\bb e_j^{\T}\tta, \ldots, \bb e_j^{\T} \tta)^{\T} \)
with \( \xi_j \sim \norm{0}{I_k} \)
\begin{eqnarray*}
 && \EE_{\tta, \Sigma} [Z^2_{k,j}]\\
 &=&
        \frac{\det \SSigma_{k,j}}{\det \SSigma_{k,0,j}}
        \exp\{ -\| \SSigma_{k,0,j}^{-1/2} b_j(k) \|^2 \}  \\
     &\times& \EE\left[\exp\{ -\| \SSigma_{k,0,j}^{-1/2}\SSigma_{k,j}^{1/2} \xi_j \|^2
                            + \| \xi_j \|^2
                            + 2 b_j(k)^{\T} \SSigma_{k,0,j}^{-1}\SSigma_{k,j}^{1/2} \xi_j \}  \right]\\
  &=& \frac{\det \SSigma_{k,j}}{\det \SSigma_{k,0,j}}
        \prod_{l=1}^k
        \left[ 2 \lambda_l(\SSigma_{k,j}^{1/2}\SSigma_{k,0,j}^{-1}\SSigma_{k,j}^{1/2}) -1 \right]^{-1/2}\\
  &\times& \exp \left\{ b_j(k)^{\T} \SSigma_{k,0,j}^{-1/2}
        \left[ 2 \SSigma_{k,0,j}^{-1/2}\SSigma_{k,j}^{1/2}
            (2 \SSigma_{k,j}^{1/2}\SSigma_{k,0,j}^{-1}\SSigma_{k,j}^{1/2} -I_k )^{-1}
            \SSigma_{k,j}^{1/2}\SSigma_{k,0,j}^{-1/2} -I_k
            \right]
        \SSigma_{k,0,j}^{-1/2} b_j(k) \right\}.
\end{eqnarray*}
Now utilizing \eqref{condition Sj} we get
\begin{equation*}
    \frac{1-\delta}{1+\delta} I_{k}
        \preceq
            \left( 2 \SSigma_{k,j}^{1/2}\SSigma_{k,0,j}^{-1}\SSigma_{k,j}^{1/2} - I_{k}  \right)^{-1}
        \preceq
    \frac{1+\delta}{1-\delta} I_{k},
\end{equation*}
\begin{eqnarray*}
   &&     2 \SSigma_{k,0,j}^{-1/2}\SSigma_{k,j}^{1/2}
            (2 \SSigma_{k,j}^{1/2}\SSigma_{k,0,j}^{-1}\SSigma_{k,j}^{1/2} -I_k )^{-1}
            \SSigma_{k,j}^{1/2}\SSigma_{k,0,j}^{-1/2} -I_k \\
  &\preceq&
    2 \frac{1+\delta}{1-\delta} \SSigma_{k,0,j}^{-1/2}\SSigma_{k,j} \SSigma_{k,0,j}^{-1/2} -I_k\\
  &\preceq& \left(2 \frac{1+\delta}{(1-\delta)^2} -1\right)I_k, \\
  && \\
&& \frac{\det \SSigma_{k,j}}{\det \SSigma_{k,0,j}}
        \le
    \left( \frac{1}{1-\delta} \right)^{k},\\
&&    \prod_{l=1}^k
        \left[ 2 \lambda_l(\SSigma_{k,j}^{1/2}\SSigma_{k,0,j}^{-1}\SSigma_{k,j}^{1/2}) -1 \right]^{-1/2}    \le
     \left( \frac{1+\delta}{1-\delta} \right)^{\frac{k}{2}}.
\end{eqnarray*}
Finally, because \( b_j(k)^{\T} \SSigma_{k,0,j}^{-1} \,b_j(k) \le \Delta_j(k) (1-\delta)^{-1}\), we
obtain the bound for the second moment of the Radon-Nikodym derivative:
\begin{equation*}
    \EE_{\tta, \Sigma} [Z^2_{k,j}]
        \le
        \left( \frac{1+\delta}{(1-\delta)^3} \right)^{\frac{k}{2}}
        \exp\left\{ \varphi(\delta)  \frac{\Delta_j(k)}{1-\delta} \right\}
\end{equation*}
which completes the proof.
\end{proof}

At this point we introduce the following ``componentwise'' small modeling dias conditions:
\begin{description}
\item[\( \bb{(SMBj)} \)]
\emph{
Let there exist for some \( j = 1, \ldots, p \), some \( k(j) \le K \) and some \( \theta^{(j)} = e_j^{\T} \tta \) a constant \( \Delta_j  \ge 0 \) such that
\begin{equation} \label{smbj}
  \Delta_j(k(j)) \le  \Delta_j,
\end{equation}
where \( \Delta_j(k) \) is defined by \eqref{def_Delta j(k)}.
}
\end{description}
\begin{definition}
For each \( j = 1, \ldots, p \) the oracle index \( k^*(j) \) is defined as the largest index in the scale for which the \( (SMBj) \) condition holds, that is,
\begin{equation}\label{def oracle index kj}
    k^*(j) = \max \{ k\le K : \Delta_j(k)\le  \Delta_j \}.
\end{equation}
\end{definition}
\begin{theorem}\label{oracle result componentwise}
Let the smallest bandwidth \( h_1 \) be such that the first estimator \( \bb e_j^{\T}\mmle{1}(x) \) be always accepted in the adaptive procedure.
Let \( k^*(j) \) be the oracle index defined by \eqref{def oracle index kj}, \( j = 1, \ldots, p \) . Assume \( (\mathfrak{D}) \), \( (\mathfrak{Loc}) \), \( (\mathfrak{S}) \), \( (\mathfrak{B}) \), \( (PC) \), \( (\cc{W}) \), \( (\mathfrak{S1}) \) and \( \mathfrak{(Lp1^{\mathrm{d} })} \). Then the risk between the \( j \)th coordinates of the adaptive estimator and the oracle is bounded with the following expression:
\begin{eqnarray}
         && \left( \frac{nh_{k^*(j)}^d\Lambda_0}{\s^2_{max}} \right)^{r/2}
               \EE |\bb e_j^{\T}\mmle{k^*(j)}(x) - \bb e_j^{\T}\aadapest(x)|^r\\
       & \le &   \z_{k^*(j)}^{r/2} +
       (\alpha \EE|\chi^2_p|^r)^{1/2} (1+\delta)^{pk_j^*/4}(1-\delta)^{-3pk_j^*/4}
        \exp\left\{ \varphi(\delta)
        \frac{\Delta_j}{2(1-\delta)}\right\} \nonumber
\end{eqnarray}
where \( \varphi(\delta) \) as in Theorem \ref{Propagation result componentwise theorem}.
\end{theorem}

\begin{corollary}\label{corr:oracle risk bound for pol bas}
Let the basis be polynomial. Under the conditions of the preceding theorem, the risk between the adaptive estimator \( LP^{ad}(p-1) \) of the value of the \( j \)th derivative of \( f \) at \( x \) and the oracle is bounded with the following expression:
 \begin{eqnarray*}
 &&  \left( \frac{nh_{k^*(j)}^d\Lambda_0}{\s^2_{max}} \right)^{r/2}
       \EE |\flej{j-1}{k^*(j)}{x} - \adaplpest^{(j-1)}(x)|^r\nn
    &\le &  \z_{k^*(j)}^{r/2} +
       (\alpha \EE|\chi^2_p|^r)^{1/2} (1+\delta)^{pk_j^*/4}(1-\delta)^{-3pk_j^*/4}
        \exp\left\{ \varphi(\delta)
        \frac{\Delta_j}{2(1-\delta)}\right\}
 \end{eqnarray*}
 with \( \varphi(\delta) \) as before.
 \end{corollary}

\begin{proof} To simplify the notation we suppress the dependence on \( j \) in the index \( k \). Similarly to the proof of Theorem \eqref{oracle result} we consider disjunct events
\( \{ \adapind \le k^*\} \) and \( \{ \adapind > k^*\} \). Therefore,
\begin{eqnarray*}
       && \EE|\bb e_j^{\T}\mmle{k^*}(x) - \bb e_j^{\T}\aadapest(x)|^r\\
&=& \EE|\bb e_j^{\T}\mmle{k^*}(x) - \bb e_j^{\T}\aadapest(x)|^r \,\ind\{ \adapind \le k^* \} \\
&+&
\EE|\bb e_j^{\T}\mmle{k^*}(x) - \bb e_j^{\T}\aadapest(x)|^r \, \ind \{ \adapind > k^*\}.
\end{eqnarray*}
By Lemma \ref{bound for the componentwise differences} and the definition of the test statistic \( T_{k^*, \adapind} \) the second summand can be easily bounded:
\begin{eqnarray*}
    && \left( \frac{nh_{k^*}^d\Lambda_0}{\s^2_{max}} \right)^{r/2}
          \EE|\bb e_j^{\T}\mmle{k^*}(x) - \bb e_j^{\T}\aadapest(x)|^r \,\ind \{ \adapind > k^*\} \\
    &\le & \EE \| \B{k^*}^{1/2} (\mmle{k^*}(x) - \aadapest(x)) \|^r \,\ind \{ \adapind > k^*\}\\
    &\le &  \z_{k^*}^{r/2}.
\end{eqnarray*}

\par To bound the first summand we use the ``componentwise'' analog of Theorem~\ref{Propagation result theorem}, particularly Theorem~\ref{Propagation result componentwise theorem}, and this completes the proof.
\end{proof}

\section{Rates of convergence}
\label{section: adaptive rates}

\subsection{Minimax rate of spatially adaptive local polynomial estimators}
\label{subsect:mimimax rates}
In this section we give some basic information on spatial adaptation and present the rate of
convergence of the adaptive local polynomial estimator \( LP^{ad}(p-1) \) of \( f(x) \).
Let us recall that Donoho and Johnstone in 
 \citeasnoun{Donoho and Johnstone} suggested how to measure the quality of adaptive estimators. The authors called this approach \emph{``the ideal spatial adaptation''} and defined it as a level of performance which would be achieved by smoothing with knowledge of the best ``oracle'' scheme. The estimator corresponding to this scheme is called an ``oracle''. The adaptive methods try to construct an estimator which mimics the performance of the oracle in some sense, for example, in terms of the risk of estimation. Inequalities relating the risk of the adaptive estimator to the risk of the oracle are usually referred to as \emph{``oracle inequalities''}. The results obtained in Section \ref{subsection:oracle result} belong to this family.

\par To simplify the representation in this section we consider a univariate design in~\( [0,1] \).
 The generalization to the multidimensional case is straightforward. Fix a point \( x \in [0,1] \) and a method of localization \( w_{(\cdot)} \). In this section we also assume that the basis is polynomial and centered at \( x \), that is \( \psi_1 \equiv 1 \) and \( \psi_{j}(t) = (t-x)^{j-1}/(j-1)! \) with \( j = 2, \ldots, p \). As in Section \ref{sect:LocPol} we denote for any \( k= 1, \ldots, K \) by
\begin{equation}\label{lp estmator}
    \fle{k}{x} \eqdef \bb e^{\T}_1 \mmle{k}(x)
\end{equation}
the local polynomial estimator of order \( p-1 \) of \( f(x) \) corresponding to the \( k \)th scale with the bandwidth \( h_k = h_k(x) \), or just the \( LP_k(p-1) \) estimator of \( f(x) \) for short. Here \( \bb e_1 \in \RRp \) is the first canonical basis vector \( (1,0,\ldots,0)^{\T} \). As before we assume that
\begin{equation*}
    \frac{1}{n} <  h_1 < \ldots < h_k< \ldots  h_K \le 1
\end{equation*}
and therefore that the ordering condition \( (\cc W) \) is satisfied. Denote the \emph{adaptive local polynomial estimator} \( LP^{ad}(p-1) \) of \( f(x) \) by
\begin{equation}\label{adapt_lp estimator}
    \adaplpest(x) \eqdef \fle{\adapind}{x} = \bb e^{\T}_1 \aadapest(x).
\end{equation}
with \( \aadapest(x) \) defined by \eqref{aadapest}. To obtain bounds for the risk of the adaptive estimator in \citeasnoun{Donoho and Johnstone}, \citeasnoun{Goldenshluger and Nemirovski} and \citeasnoun{LepMamSpok97} it was suggested to compare the \( \MSE(x) \) (the \( L_r \)-risk in \citeasnoun{Goldenshluger and Nemirovski}) corresponding to the adaptive estimator \( \adaplpest(x) \) with the infimum over all scales of the mean squared risks (the \( L_r \)-risks, respectively) of nonadaptive estimators \( \fle{l}{x} \), \( l= 1, \ldots, K \). That is we compare \( \EE_f[|\adaplpest(x) -f(x)|^2] \) with the ``best'' risk of the form \( \EE_f[|\fle{l}{x} -f(x)|^2] \). Clearly, for any \( l \) by the bias-variance decomposition and  by \eqref{decomposition of LP estim into det and stoch} we have
\begin{equation*}
    \EE_f[|\fle{l}{x} -f(x)|^2] = b^2_{l,f}(x) + \s^2_l(x),
\end{equation*}
where the variance term is defined by
\begin{equation*}
    \s^2_l(x) \eqdef \EE_f[|\bb e^{\T}_1 \mmle{l}(x)  - \bb e^{\T}_1 \bbpf{l}(x) |^2]
\end{equation*}
and the bias is given by
\begin{equation*}
  b_{l,f}(x) \eqdef  \bb e^{\T}_1 \bbpf{l}(x) - f(x).
\end{equation*}
Here
\begin{equation*}
    \bb e^{\T}_1 \bbpf{l}(x) = \EE_f [\fle{l}{x}] = \sum_{i=1}^{n}  W^*_{l,\,i}(x) \ffi
\end{equation*}
is a \emph{local linear smoother} of the function \( f  \) at the point \( x \) corresponding to the
\( l \)th scale, see Section~\ref{sect:LocPol} for details. The polynomial weights \( W^*_{l,\,i} \) now are defined by
\begin{equation}\label{polynomial weights_multiscaled}
    W^*_{l,\,i}(x) = \bb e^{\T}_1 \B{l}^{-1} \Psii  \frac{\w{l}{i}(x)}{\s^2_i}
\end{equation}
with \( \B{l}\) defined by \eqref{B}. The columns of the ``design'' matrix \( \PPsi \) are given by:
\begin{equation*}
     \Psii= \Psi(\Xi-x) = \left(1,\, \Xi -x, \ldots, (\Xi -x)^{\p-1}/(p-1)!\right)^{\T}.
\end{equation*}
 With this notation the ideal spatial adaptation can be expressed as follows (see \citeasnoun{LepMamSpok97}):
\begin{equation}\label{adaptive balance equation without correction term}
    \MSE^{id}(x) = \inf_{1\le k\le K} \{ \bar{b}^2_{k,f}(x) + \s^2_k(x)\}
\end{equation}
where for any \( k \) the first summand
\begin{equation}\label{adaptive bias}
    \bar{b}_{k,f}(x) \eqdef \sup_{1\le l \le k } |b_{l,f}(x)|
                     =      \sup_{1\le l \le k } |\bb e^{\T}_1 \bbpf{l}(x) - f(x)|
\end{equation}
reflects the \emph{local smoothness} of \( f \) within the largest interval \( [x-h_k , x+h_k] \), containing intervals \( [x-h_l , x+h_l] \) with \( 1\le l<k \).
Indeed, the smoothness of a function can be defined via the quality of its approximation by polynomials, see \citeasnoun{Fan and Gijbels book} for example.
The bandwidth \( h^{\star} = h^{\star}(x, w_{(\cdot)},f(\cdot)) \)
providing a trade-off between \( \bar{b}^2_{k,f}(x) \) and the variance term could be called an ``ideal'' or ``oracle'' bandwidth. Unfortunately, as it is generally in nonparametric estimation, we cannot minimize the right-hand side of \eqref{adaptive balance equation without correction term} directly because it depends on the unknown function \( f \). The lack of information about \( f \) can be compensated by the assumption that \( f \) belongs to some smoothness class, see \citeasnoun{Ibragimov and Khasminskii}. This technique in the nonadaptive set-up under the assumption that \( f \in \Sigma (\beta, L) \) on \( [0,1] \) is demonstrated in Section~\ref{Nonadaptive rate} for the local polynomial approximation of order \( p -1 = \lfloor \beta \rfloor\). Here the use of \eqref{adaptive bias}  or of the SMB conditions allows to adapt not to the functional class but to the smoothness properties of the function \( f \) itself.

\par In the pointwise adaptation framework due to Lepski \citeasnoun{Lep1990}, see also \citeasnoun{Brown and Low}, it was discovered that the relation \eqref{adaptive balance equation without correction term} ``does not work''. This means that an adaptive estimator satisfying \eqref{adaptive balance equation without correction term} does not exist. In pointwise estimation one has to pay an additional logarithmic factor \( d(n) \) for proceeding without knowledge of the regularity properties of \( f \). It was proved in \citeasnoun{Lep1990} and \citeasnoun{Lep1992} that this factor \( d(n) \) is unavoidable and is of order \( \log n \), where \( n \) is the sample size. In \citeasnoun{LepSpok97} for kernel smoothing in the Gaussian white noise model (in our set-up under regularity assumptions on the design this is the case of \( p=1  \), \( \delta=0 \) and \( \s_i \equiv \s \)), it was shown that \( d(n) \)  depends on the range of adaptation, that is on the ratio of the largest bandwidth to the smallest one and that \( d(n) \) is not larger in order than \( \log n \). This phenomenon can be expressed as an increase of the noise level leading to the \emph{adaptive} upper bound for the squared risk (see \citeasnoun{LepMamSpok97}) in the following form:
\begin{eqnarray}\label{adaptive balance equation}
  \MSE^{ad}(x) &=& \inf_{1\le k\le K} \{ \bar{b}^2_{k,f}(x) + \s^2_k(x)d(n)\}.
\end{eqnarray}
This relationship (see \citeasnoun{LepSpok97}) can be written in the form of a \emph{``balance equation''}:
\begin{equation}\label{adaptive balance equation equation}
    \bar{b}_{k,f}(x) = C(w) \s_k(x) \sqrt{d(n)}
\end{equation}
with
\begin{equation}\label{adaptive factor}
    d(n) = \log \bigg(\frac{h_K}{h_1}  \bigg).
\end{equation}
The optimal selection of the constant \( C(w) \) provides sharp oracle results. The bandwidth \( h^\star = h_{k^\star} \) such that
\begin{equation}\label{k star}
    k^\star = \max\{ k\le K : \bar{b}_{k,f}(x) \le  C(w) \s_k(x) \sqrt{d(n)} \}
\end{equation}
is called the \emph{``ideal adaptive bandwidth''}
 or just the \emph{``oracle bandwidth''}.

\par Before the proceeding with the analysis of the convergence rate, let us point out that the weights \( W^*_{l,\,i}(x) \) defined by \eqref{polynomial weights_multiscaled}
preserve the reproducing polynomials property:
\begin{proposition}\label{Henderson th multiscaled}
Let \( x \in \RR\) be such that \( \B{1}  = \sum_{i=1}^n \Psii \Psii^{\T} \w{1}{i}(x)\s_i^{-2} \succ 0 \). Then the weights defined by \eqref{polynomial weights_multiscaled} satisfy
\begin{eqnarray}\label{normalization of polynomial weights multiscaled}
   & & \sum_{i=1}^{n} W^*_{l,\,i}(x) = 1,  \\ \nonumber
   & & \sum_{i=1}^{n} (\Xi - x)^m W^*_{l,\,i}(x) = 0 \; , \;\; m = 1, \ldots, p-1,
\end{eqnarray}
for all \( l=1, \ldots , K \) and design points \( \{ X_1, \ldots, X_n \} \).
\end{proposition}
\begin{proof}
By Assumption \( (\cc W) \), if \( \B{1} \succ 0 \) at some point \( x \), then \( \B{l} = \B{l}(x) \succ 0 \)
for all \( l=1, \ldots , K \), and the assertion follows from the proof of Proposition \ref{Henderson th}.
\end{proof}
To simplify the study of \eqref{adaptive balance equation} we need to introduce
the following assumptions:
 \begin{description}
\item[\( \bb{\mathfrak{(Lp1')}} \)]
\emph{Assume \( \mathfrak{(S1)} \). There exists a number \( \lambda_0>0  \) such that for any \( k=1, \ldots, K \) the smallest eigenvalue fulfills \( \lambda_p(\B{k}) \ge  nh_k\lambda_0 \s^{-2}_{max}\) for sufficiently large \( n \).}

\item[\( \bb{\mathfrak{(Lp2')}} \)]
\emph{There exists a real number \( a_0>0 \) such that for any interval \( A \subseteq [0,1] \) and all \( n \ge 1 \)}
\begin{equation*}
    \frac{1}{n} \sum_{i=1}^n \ind \{ \Xi \in A \} \le a_0 \max \big\{ \int_A\dd t, \frac{1}{n} \big\}.
\end{equation*}
\item[\( \bb{\mathfrak{(Lp3')}} \)]
\emph{ The localizing functions (kernels) \( \w{k}{i} \) are compactly supported in \( [0,1] \) with}
\begin{equation*}
    \w{k}{i}(x) = 0 \;\;\;\text{if} \;\;\; |\Xi - x|> h_k.
\end{equation*}
\noindent This immediately implies the similar property for the local polynomial weights:
\begin{equation*}
    W^*_{k,i}(x) = 0 \;\;\;\text{if} \;\;\; |\Xi - x|> h_k.
\end{equation*}

\item[\( \bb{\mathfrak{(Lp4')}} \)]
\emph{ There exists a finite number \( w_{max} \) such that}
\begin{equation*}
    \sup_{k,i}|\w{k}{i}(x) | \le w_{max}.
\end{equation*}

\end{description}
\begin{remark}
Assumption \( \mathfrak{(Lp1')} \) is weaker than \( \mathfrak{(Lp1)} \) because it does not require the uniformity in \( x \).
\end{remark}
\begin{remark}
Assumption \( \mathfrak{(S1)} \) implies that the conditional number
\begin{equation}\label{kond_number}
    \kappa(\Sigma) \eqdef \frac{\s^2_{max}}{\s^2_{min}}
\end{equation}
of the covariance matrix in the known ``wrong'' model \eqref{PA} is finite.
\end{remark}
\begin{theorem}\label{nonuniform upper bounds for bias and variance}
 Assume \( (\cc W) \), \( \mathfrak{(S)} \), \( \mathfrak{(S1)} \),  \( \mathfrak{(Lp1')} \)--\( \mathfrak{(Lp4')} \) and that the smallest bandwidth \( h_1 \ge  \frac{1}{2n} \). Let \( f \in \Sigma (\beta, L) \) on \( [0,1] \) and let \( \{ \fle{k}{x} \}_{k=1}^K\) be the \( LP_k(p-1) \) estimators of \( f(x) \) with \( p -1 = \lfloor \beta \rfloor\). Then for sufficiently large \( n \) and any \( h_k \) satisfying \( h_K > \ldots > h_k > \ldots > h_1 \), \( k=1, \ldots, K \), the following upper bounds hold:
\begin{eqnarray*}
  |\bar{b}_{k,f}(x)|    &\le & C_2  \kappa(\Sigma) \frac{Lh_k^{\beta}}{(p-1)!}, \\
  \s^2_k(x)   &\le & (1+\delta) \frac{\s^2_{max} }{nh_k \lambda_0},
\end{eqnarray*}
with  \( C_2 = 2 w_{max} a_0 \sqrt{e}/\lambda_0 \) and \( \delta \in [0,1) \).
\par Moreover, the choice of a positive bandwidth \( h = h^\star(n) \) (see \eqref{h star with const} for the precise formula) in the form:
\begin{equation*}
    h^\star(n) = \cc O \left(\left( \frac{d(n)}{n} \right)^{ \frac{1}{2 \beta +1} }\right)
\end{equation*}
provides the following upper bound for the risk of adaptive estimator:
\begin{equation}\label{upper bound for MSE adaptive}
    \varlimsup_{n \to \infty} \sup_{f \in \Sigma (\beta, L)}
        \EE_f[ \psi_n^{-2} | \adaplpest(x) -f(x)|^2] \le  C,
\end{equation}
where
\begin{equation}\label{adaptive rate Obig}
    \psi_n = \cc O \left( \left( \frac{d(n)}{n} \right)^{\frac{\beta}{2 \beta +1} }\right)
\end{equation}
is given by \eqref{adaptive rate}
and the constant \( C \) is finite and depends on \( \beta \), \( L \), \( \s^2_{min} \), \( \s^2_{max} \), \( p \), \( w_{max} \) and \( a_0 \) only.
\end{theorem}
\begin{remark}
The bound for \( \s^2_k(x) \) is simple than the corresponding one from Theorem \ref{theorem upper bound for MSE} due to the assumption of normality of the vector of errors  (\( \eeps \sim \norm{0}{I_n} \)) in the models \eqref{true model}--\eqref{PA}.
\end{remark}
\begin{remark}
Recall that in \citeasnoun{LepSpok97} it was shown that the ``adaptive factor'' \( d(n) \) cannot be less in order than \( \log\big(  h_K h_1^{-1} \big) \).
\end{remark}

\begin{proof}
The bound for \( |b_{l,f}(x)| \) at the point \( x \) is obtained as in the proof of Theorem~\ref{theorem upper bound for MSE} by application of the second assertion of Lemma \ref{bounds for loc pol weights}, so we skip some details.
By Proposition \ref{Henderson th multiscaled} and the Taylor theorem with \( \tau_i \) such that the points \( \tau_i \Xi \) are between \( \Xi \) and \( x \), and utilizing Assumption \( \mathfrak{(Lp3)} \) we have:
\begin{eqnarray*}
   |b_{l,f}(x)|   &\le & \frac{1}{(p-1)!}
                        \sum_{i=1}^{n} | f^{(p-1)}(\tau_i \Xi) -f^{(p-1)}(x)|                         |\Xi -x|^{p-1} |W^*_{l,\,i}(x)|\\
   &\le & \frac{L}{(p-1)!}
                        \sum_{i=1}^{n}  |\tau_i \Xi -x|^{\beta -(p-1)} |\Xi -x|^{p-1}|W^*_{l,\,i}(x)|\\
            &\le & \frac{Lh_l^{\beta}}{(p-1)!} \sum_{i=1}^{n}  |W^*_{l,\,i}(x)|.
\end{eqnarray*}
Under the assumptions of the theorem the sum of the polynomial weights can be bounded as follows:
\begin{eqnarray*}
  \sum_{i=1}^{n}  |W^*_{l,\,i}(x)|
        &\le & w_{max}
            \sum_{i=1}^{n} \s^{-2}_i \| \B{l}^{-1} \Psii \| \\
        &\le &
            \kappa(\Sigma) \frac{w_{max}}{\lambda_0 nh_l}  \sum_{i=1}^n \| \Psii \| \,
                        \ind\{ \Xi \in [x-h_l,x+h_l] \}\\
       &\le &
            \kappa(\Sigma)
                \frac{w_{max} \sqrt{e}}{\lambda_0 } a_0 \max\{ 2, \frac{1}{nh_l} \}\\
        &\le &
            \kappa(\Sigma)\frac{2 a_0 w_{max} \sqrt{e} }{\lambda_0 },
\end{eqnarray*}
and the first assertion is justified in view of
\begin{equation}
    \bar{b}_{k,f}(x) \eqdef \sup_{1\le l \le k } |b_{l,f}(x)|
                            \le \kappa(\Sigma)\frac{2 a_0 w_{max} \sqrt{e} a_0}{\lambda_0 } \frac{Lh_k^{\beta}}{(p-1)!}.
\end{equation}
To bound the variance just notice that, because \( \B{k} \) is symmetric and non-degenerate,
by \( \mathfrak{(Lp1')} \) for any \( \gamma \in \RRp \) it holds:
\begin{equation*}
    \gamma^{\T} \B{k}^{-1} \gamma \le \frac{\s^2_{max} }{nh_k \lambda_0} \| \gamma \|^2.
\end{equation*}
Then under Assumption \( \mathfrak{(S)} \) by \eqref{Vk bound} for the variance term we have:
\begin{eqnarray*}
  \s^2_k(x) &=& \bb e_1^{\T} \Var\mmle{k} \, \bb e_1 \\
  &\le & (1+\delta) \bb e_1^{\T} \B{k}^{-1} \bb e_1 \\
  &\le & (1+\delta) \frac{\s^2_{max} }{nh_k \lambda_0}.
\end{eqnarray*}

\par By \eqref{adaptive balance equation},
\begin{equation*}
    \MSE^{ad}(x) \le \inf_{1\le k\le K} \big\{ \tilde{C}_2 h_k^{2 \beta} + \frac{\tilde{C}_1 d(n)}{nh_k} \big\}
\end{equation*}
with \( \tilde{C}_2 = (C_2 L\, \kappa(\Sigma) /(p-1)!)^2 \) and \( \tilde{C}_1 = (1+\delta) \s^2_{max} \lambda_0^{-1}\). The choice of a bandwidth of the form:
\begin{equation}\label{h star with const}
    h^\star(n) = \tilde{C} \left(  \frac{(p-1)!}{ L \, \kappa(\Sigma) }  \right)^{ \frac{2}{2 \beta +1} } \left((1+\delta ) \s^2_{max}\frac{ d(n)}{n}  \right)^{ \frac{1}{2 \beta +1} }
\end{equation}
minimizes the upper bound for the \( \MSE^{ad}(x) \) and provides the rate \( \psi_n \) w.r.t. the square loss function and over a H\"older class \( \Sigma(\beta, L) \):
\begin{equation}\label{adaptive rate}
    \psi_n = C \left(  \frac{ L \, \kappa(\Sigma) }{(p-1)!} \right)^{ \frac{1}{2 \beta +1} }
    \left(  (1+\delta )\s^2_{max} \frac{d(n)}{n}  \right)^{ \frac{ \beta}{2 \beta +1} } .
\end{equation}
Here \( \tilde{C} \) and \( C \) depend only on \( w_{max} \), \( a_0 \), \( \lambda_0 \) and \( \beta \).
\end{proof}

\subsection{SMB, the bias-variance trade-off and the rate of convergence}
\label{subsect:SMB and the bias-variance trade-off}
The choice of the ``ideal adaptive bandwidth'' usually can be done by \eqref{k star}. In \citeasnoun{SV}
it was shown that the small modeling bias \( (SMB1) \) condition \eqref{smbj} can be obtained from the ``bias-variance trade-off'' relations. Unfortunately, to have the ``modeling bias''
\( \Delta(k) = \cc O (1)\)
(this is \( \Delta_1(k) \) in the present framework) one should apply the balance equation \eqref{adaptive balance equation equation} or \eqref{k star} without the ``adaptive factor'' \( d(n) \), see equation \( (3.5) \) in \citeasnoun{SV}. In the Gaussian regression set-up (example \( 1.1 \) in \citeasnoun{SV} ) under smoothness assumptions on the regression function \( f \in \Sigma(\beta,L) \)
this results in a suboptimal rate in the upper bound for the \( \MSE(x) \):
\begin{eqnarray}\label{rate with gamma}
  \psi_n &=&  \cc O \left( L^{\frac{1}{2 \beta +1} } n^{-\frac{\beta}{2 \beta +1} } \sqrt{\log n} \right)   \nn
   &=& \cc O \left( L^{\frac{1}{2 \beta +1} }  \left(
                \frac{\log^\gamma n}{n}  \right)^{\frac{\beta}{2 \beta +1} } \right)
\end{eqnarray}
with \( \gamma = \frac{2 \beta +1}{2\beta}>1 \). Notice that, due to the normalization by \( \sqrt{\Var[\tilde{\theta}_l]} \), the adaptive procedure used in \citeasnoun{SV} coincides with Lepski's selection rule from \citeasnoun{Lep1990} and \citeasnoun{LepSpok97}. Because local constant Gaussian regression under a regularity assumption on the design is equivalent to the Gaussian white noise model, it is known from these papers that this procedure is rate optimal with the minimax rate \( \psi_n = \cc O \left( L^{\frac{1}{2 \beta +1} }  \left(\frac{\log n}{n}  \right)^{\frac{\beta}{2 \beta +1} } \right) \), that is, \( \gamma \) should be equal to \( 1 \). This shows that the method of obtaining the upper bounds from \citeasnoun{SV} and generalized in the present work should be refined. This lack of optimality was also independently noticed in \citeasnoun{Reiss}.

\par Now we will demonstrate that: (1) the definition of the ``ideal adaptive bandwidth''~\eqref{k star} with \( d(n) =1 \) implies the \( (SMBj) \) conditions; (2) for Lepski's selection rule in our framework we have the same rate for the upper bound of the risk as in equation \eqref{rate with gamma}.

\par Notice that for the method of local approximation using of the polynomial basis centered at \( x \) the definition of the ``ideal adaptive bandwidth''~\eqref{k star} can be easily generalized for the estimators of the derivatives of \( f \) defined by \eqref{def of estimators of f and its derivatives}. Then, given a point \( x \) and the method of localization \( w_{(\cdot)} \), for any \( j = 1, \ldots, p \) the formula~\eqref{k star} reads as follows:
\begin{equation}\label{kj star}
    k^\star(j) = \max\{ k\le K : \bar{b}_{k,f^{(j-1)}}(x) \le  C_j(w) \s_k(x) \sqrt{d(n)} \},
\end{equation}
where \( C_j(w) \) is a constant depending on the choice of the smoother \( w_{(\cdot)} \),
\begin{eqnarray*}
  \bar{b}_{k,f^{(j-1)}}(x) &=& \sup_{1\le l\le k}|\bb e_j^{\T} \bbpf{l}(x) -  f^{(j-1)}(x)| , \\
  \s^2_k(x) &=& \Var_{\ff, \Sigma_0} [\bb e_j^{\T} \mmle{k}(x) ] ,
\end{eqnarray*}
and \( f^{(0)} \) stands for the function \( f \) itself. To bound the ``modeling bias'' \( \Delta_j(k) \) we need the following assumption:
\begin{description}
\item[\( \bb{(\mathfrak{S}kj)} \)]
\emph{
There exists  a constant \( s_j > 0 \) such that for all \( k \le K \)
\begin{equation}\label{sigma kj}
     \SSigma^{-1}_{k,j} \preceq s_j \SSigma^{-1}_{k,j,diag} ,
\end{equation}
}
\end{description}
where \( \SSigma_{k,j,diag} = \diag\big(\Var_{\tta, \Sigma}[\bb e_j^{\T} \mmle{1}(x)] ,
 \ldots , \Var_{\tta, \Sigma}[\bb e_j^{\T} \mmle{k}(x)] \big) \) is a diagonal matrix composed of the diagonal elements of \( \SSigma_{k,j} \). Thus we have the following result:
\begin{theorem}\label{th:SMB from the balance equation}
Assume \( (\mathfrak{B}) \), \( (\mathfrak{S}) \) and \( (\mathfrak{S}kj) \). Let the weights \( \{ \w{k}{i}(x) \} \) satisfy~\eqref{binar product}. Then for any given point \( x \), smoothing function \( w_{(\cdot)} \) and \( j = 1, \ldots, p \) the choice of \( k(j) = k^\star(j) \) defined by the relation \eqref{kj star} with \( d(n)=1 \) implies the \( (SMBj) \) condition \( \Delta_j(k(j)) \le \Delta_j \) with
the constant \( \Delta_j = s_j C^2_j(w) (1+\delta) (1-u_0^{-1})^{-1} \).
\end{theorem}
\begin{proof}
Consider the quantity \( b_j(k)^{\T} \SSigma^{-1}_{k,j,diag} b_j(k)\). Suppose that \( \bb e_j^{\T} \tta(x) = f^{(j-1)} (x) \). In view of relation  \eqref{binar product} for the weights \( \{ \w{l}{i}(x) \} \) the form of the matrix \( \SSigma_{k,j,diag} \) is particularly simple:
\begin{equation*}
    \SSigma_{k,j,diag} = \diag (\bb e_j^{\T} \B{1}^{-1} \bb e_j , \ldots,  \bb e_j^{\T} \B{k}^{-1} \bb e_j).
\end{equation*}
Then by \( (\mathfrak{B}) \) and \eqref{Vk bound}
\begin{eqnarray*}
  b_j(k)^{\T} \SSigma^{-1}_{k,j,diag} b_j(k)
        &=&
            \sum_{l=1}^k   \frac{|\bb e_j^{\T} (\bbpf{l} - \tta)|^2}{\bb e_j^{\T} \B{l}^{-1} \bb e_j}
        \\
        &\le &
            \big(\bar{b}_{k,f^{(j-1)}}(x)\big)^2
                \sum_{l=1}^k   \frac{1}{\bb e_j^{\T} \B{l}^{-1} \bb e_j}   \\
        &\le & \frac{\big(\bar{b}_{k,f^{(j-1)}}(x)\big)^2}{\bb e_j^{\T} \B{k}^{-1} \bb e_j}
                \sum_{l=1}^k u_0^{-(k-l)}\\
        &\le & \frac{\big(\bar{b}_{k,f^{(j-1)}}(x)\big)^2 (1+\delta)}{\s^2_k(x) (1-u_0^{-1})}.
\end{eqnarray*}
By \eqref{kj star} with \( d(n) =1 \) the choice of \( k=k^\star(j) \)
implies \( \big(\bar{b}_{k,f^{(j-1)}}(x)\big)^2 \le  C^2_j(w) \s^2_k(x)\). Thus
\begin{equation*}
     b_j(k)^{\T} \SSigma^{-1}_{k,j,diag} b_j(k) \le  (1+\delta) C^2_j(w)(1-u_0^{-1})^{-1}
\end{equation*}
and
\begin{equation*}
    \Delta_j(k) =  b_j(k)^{\T} \SSigma^{-1}_{k,j} b_j(k) \le s_j C^2_j(w) (1+\delta) (1-u_0^{-1})^{-1}.
\end{equation*}
\end{proof}

\par Now we will show that the ``oracle'' risk bound from Corollary \ref{corr:oracle risk bound for pol bas} delivers at least the suboptimal \eqref{rate with gamma} rate of convergence for the upper bound of the risk w.r.t. the polynomial loss function and over a H\"older class \( \Sigma(\beta,L) \). For simplicity we restrict ourselves to the case of the univariate design. We study the quality of the \( LP^{ad} (p-1) \) estimator \( \adaplpest (x) \) of \( f(x) \) under the assumption that \( f \in \Sigma(\beta,L)  \) on~\( [0,1] \) with \( \lfloor \beta \rfloor = p-1 \).

\par Denote by \( k^\star \) the index \( k^\star(j) \) with \( j=1 \) from Theorem \ref{th:SMB from the balance equation} and the corresponding bandwidth \( h_{k^\star} \) by \( h^{\star} \). Then the following asymptotic result holds:

\begin{theorem}
Assume \( (\mathfrak{B}) \), \( (\mathfrak{D}) \), \( (\mathfrak{Loc}) \), \( (PC) \), \( (\mathfrak{S}) \), \( \mathfrak{(S1)} \), \( (\mathfrak{S}k^\star1) \), \( \mathfrak{(Lp1')} \)--\( \mathfrak{(Lp4')} \), \( (\cc W) \), and that the smallest bandwidth fulfills \( h_1 \ge  \frac{1}{2n} \) and is such that the first estimator \( \fle{1}{x} = \bb e_1^{\T}\mmle{1}(x) \) is always accepted by the adaptive procedure. Let the weights \( \{ \w{k}{i}(x) \}_{k=1}^K \) satisfy~\eqref{binar product}. Assume that for \( x \in (0,1) \) there exists \( \tta(x) \in \RRp \) such that \( f(x) = \bb e_1^{\T} \tta(x) \). Let \( f \in \Sigma (\beta, L) \) on \( [0,1] \)  with \( p -1 = \lfloor \beta \rfloor\). Then for the risk of the adaptive \( LP^{ad}(p-1) \) estimator \( \adaplpest (x) \) of the function \( f(x) \) at the point \( x \in (0,1) \) the following upper bound holds:
\begin{equation*}
    \EE |f(x)  - \adaplpest (x)|^r
        \le
             C L^{\frac{r}{2 \beta +1} }
                  \left( \frac{\log^\gamma n}{n}   \right)^{\frac{r\beta}{2 \beta +1} }
                   (1 + o(1))  , \; n \to \infty
\end{equation*}
with \( \gamma = \frac{2 \beta +1}{2 \beta} \) and the constant \( C \) depending on
\( \beta \), \( \s^2_{min} \), \( \s^2_{max} \), \( p \), \( w_{max} \), \( \lambda_0 \) and \( a_0 \) only.
\end{theorem}

\begin{proof}
By the triangle inequality and the inequality \(  (a+b)^r \le C_r (a^r + b^r) \) with \( C_r = 2^{r-1} \), \( r\ge 1 \) and \( C_r = 1 \) for \( r \in (0,1) \), for any \( k=1, \ldots, K \) we have
\begin{equation*}
    |f(x)  - \adaplpest (x)|^r
        \le
            C_r \left[ |f(x) - \fle{k}{x} |^r + |\fle{k}{x} -\adaplpest (x) |^r  \right] .
\end{equation*}
Let \( \tta(x) \in \RRp \) be such that \( f(x) = \bb e_1^{\T} \tta(x) \). Then, because \( \alpha \in(0,1] \), by Theorem~\ref{Propagation result theorem} and Corollaries \ref{corr:PP for pol bas} and \ref{corr:oracle risk bound for pol bas} we have
\begin{equation*}
\left( \frac{n h^\star \lambda_0}{\s^2_{max}}   \right)^{\frac{r}{2} }
\EE |f(x)  - \adaplpest (x)|^r
        \le
C_r \left[ \z_{k^\star}^{r/2} + 2( \EE|\chi^2_p|^r)^{1/2}
\left( \frac{1+\delta}{(1-\delta)^3}  \right)^{\frac{pk^\star}{4} }
        \exp\left\{ \frac{\varphi(\delta) \Delta_1}{2(1-\delta)}\right\}    \right].
\end{equation*}
By Theorem \ref{upper bound} \( \z_{k^\star} \) is not larger in order than \( K \asymp \log n \).
Then for \( \delta = o \big( \frac{1}{K}  \big)  = o \big( \frac{1}{\log n}  \big)\)
\begin{equation*}
    \EE |f(x)  - \adaplpest (x)|^r
        \le C \left( \frac{\log n}{n h^\star}   \right)^{r/2} (1 + o(1)), \; n \to \infty.
\end{equation*}
The precise constant can be extracted easily, but because we anyway will get only a suboptimal upper
bound, in the following we will not care about the constants. The balance equation
\eqref{kj star} with \( j=1 \) and \( d(n) =\cc O (1) \) and the bounds for the bias and variance from Theorem \ref{nonuniform upper bounds for bias and variance} suggest the choice of bandwidths in the form:
\begin{equation*}
    h^\star \ge  C (L^2 n)^{-\frac{1}{2 \beta +1} }
\end{equation*}
leading to the following bound for the risk:
\begin{equation*}
     \EE |f(x)  - \adaplpest (x)|^r
        \le C L^{\frac{r}{2 \beta +1} }
                  \left( \frac{(\log n)^{\frac{2 \beta +1}{2 \beta}}}{n}   \right)^{\frac{r\beta}{2 \beta +1} }  (1 + o(1)).
\end{equation*}

 \end{proof}

\section{Auxiliary results}
\label{section:Auxiliary results}
\begin{lemma}\label{Pivotality property} {\it Pivotality property}
\par Let \( {(\cc{W})} \) hold. Under \( H_{\kappa} \) for any \( k \le   \kappa \) the risk associated with the adaptive estimator at every step of the procedure does not depend on the parameter \( \tta \):
\begin{eqnarray*}
     \EE_{\tta} | ( \mmle{k} -\aadapest_{k} )^{\T}
    \B{k}
        ( \mmle{k} -\aadapest_{k} ) |^{r}
    &=& \EE_{0} | ( \mmle{k} -\aadapest_{k} )^{\T}
    \B{k}
    ( \mmle{k} -\aadapest_{k} ) |^{r},
\end{eqnarray*}
where \( \EE_0 \) denotes the expectation w.r.t. the centered measure \( \norm{0}{\Sigma} \) or \( \norm{0}{\Sigma_0} \).
\end{lemma}
\begin{proof}
After the first \( k \) steps \( \aadapest_{k} \) coincides with one of \( \mmle{m} \), \( m \le k \), and this event takes place if for some \( l \le m \) the statistic \( T_{l,\,m+1} > \z_l  \). Because the hypothesis \( H_{\kappa} \) implies \( H_{m+1} \) for all \( m< \kappa \)
 and in view of the decomposition \eqref{linearity if quasiMLE} it holds
\begin{eqnarray*}
  & & 
  \left\{T_{l,\,m+1} > \z_l \;
        \text{ for some} \;\; l=1, \ldots, m \,|H_{m+1} \right\}\\
  &=& 
    \left\{  (\mmle{l} - \mmle{m+1} )^{\T}
            \B{l}
            (\mmle{l} - \mmle{m+1} )   > \z_l \;
            \text{ for some} \;\; l=1, \ldots, m \,|H_{m+1} \right\}\\
      &=&\left\{ \left\|  \B{l}^{1/2}
        \left( \B{l}^{-1} \PPsi \W{l} \Sigma_0^{1/2} \eeps
            - \B{m+1}^{-1}  \PPsi \W{m+1} \Sigma_0^{1/2} \eeps \right)
            \right\|^2 > \z_l \; ,\;\; l\le  m \right\}
    \end{eqnarray*}
with \( \eeps \sim \norm{0}{I_n} \). The probability of this event does not depend on the
shift \( \tta \), so without loss of generality \( \tta \) can be taken equal to zero. The risk associated with the estimator \( \aadapest_k \) admits the following decomposition:
\begin{equation*}
\EE_{\tta} |(\mmle{k} - \aadapest_k )^{\T} \B{k} (\mmle{k} - \aadapest_k ) |^r
    =
    \sum_{m=1}^{k-1} \EE_{\tta} |(\mmle{k} - \mmle{m} )^{\T} \B{k} (\mmle{k} - \mmle{m} ) |^r \I{\{ \aadapest_{k} = \mmle{m}\}}.
\end{equation*}
Under \( H_k \) for all \( m < k \) the joint distribution of \( (\mmle{k} - \mmle{m} )^{\T} \B{k} (\mmle{k} - \mmle{m} ) \) does not depend on \( \tta \) by the same argumentation.
\end{proof}

\begin{lemma}\label{simidefinitness of SSigma }
The matrices \(J_k \otimes \Sigma\) and \(J_k \otimes \Sigma_0\) are positive semidefinite for any \( k=2, \ldots, K \).
\par Moreover, under the condition \( \mathfrak{(S)}\) with the same \( \delta \) the following
 relation similar to \( \mathfrak{(S)}\) holds for the covariance matrices \( \SSigma_{k}   \)
  and \( \SSigma_{k,0} \) of the linear estimators:
\begin{equation*}
  (1-\delta) \SSigma_{k} \preceq  \SSigma_{k,0} \preceq (1+\delta)  \SSigma_{k}\;,\;\; k \le K.
\end{equation*}
\end{lemma}
\begin{proof}
Symmetry of \( J_k\) and \(\Sigma \), (respectively, \(\Sigma_0 \) ) implies symmetry of \(J_k \otimes \Sigma\), (respectively, \(J_k \otimes \Sigma_0\)). Notice that any vector \( \gamma_{nk} \in \R^{nk} \) can be represented as a partitioned vector \( \gamma_{nk}^{\T} = ((\gamma_{nk}^{(1)})^{\T},(\gamma_{nk}^{(2)})^{\T}, \ldots, (\gamma_{nk}^{(k)})^{\T}) \), with \( \gamma_{nk}^{(l)} \in \R^{n} \), \( l=1, \ldots, k \). Then
\begin{equation}\label{gamma nk}
    \gamma_{nk}^{\T} (J_k \otimes \Sigma) \gamma_{nk}
        = \big(\sum_{l=1}^k \gamma_{nk}^{(l)}\big)^{\T}  \Sigma \big( \sum_{l=1}^k \gamma_{nk}^{(l)} \big)
        = \tilde{\gamma}_n^{\T} \, \Sigma \, \tilde{\gamma}_n,
\end{equation}
where \( \tilde{\gamma}_n \eqdef \sum_{l=1}^k \gamma_{nk}^{(l)} \in \R^n \). Because
\( \Sigma \succ 0 \), this implies \( \tilde{\gamma}_n^{\T} \Sigma \,\tilde{\gamma}_n > 0 \) for all \( \tilde{\gamma}_n \ne 0 \). But even for \( \gamma_{nk} \ne 0 \), if its subvectors \( \{ \gamma_{nl}^{(l)} \} \) are linearly dependent, \( \tilde{\gamma}_n \) can be zero. Thus there exists a nonzero vector \( \gamma \) such that \( \gamma^{\T} (J_k \otimes \Sigma) \gamma =0 \). This means positive semidefiniteness.

\par The second assertion follows from the observation that
the condition \( \mathfrak{(S)}\) due to the equality \eqref{gamma nk} also holds for the Kronecker product
\begin{equation}\label{ineq for the midle parts of SSigmas}
 (1-\delta) J_k \otimes \Sigma \preceq J_k \otimes \Sigma_0 \preceq (1+\delta) J_k \otimes \Sigma.
\end{equation}
Therefore
\begin{equation*}
     (1-\delta) \DD_k(J_k \otimes \Sigma) \DD_k^{\T}
        \preceq \DD_k(J_k \otimes \Sigma_0 ) \DD_k^{\T}
        \preceq (1+\delta) \DD_k(J_k \otimes \Sigma) \DD_k^{\T}.
\end{equation*}

\end{proof}

\begin{lemma}\label{nonsingularity of SSigma rectang}
Suppose that the weights \( \{ \w{l}{\,i}(x) \} \) for every fixed \( x \in \R^d \) satisfy
\begin{equation}\label{binar product}
    \w{l}{\,i}(x) \w{m}{\,i}(x) = \w{l}{\,i}(x) \;,\;\; l \le m.
\end{equation}
Then under the conditions \( (\mathfrak{D}) \), \( (\mathfrak{Loc}) \), \( (\mathfrak{B}) \) the covariance matrix \( \SSigma_k \) defined by \eqref{def SSigma} is nonsingular with
\begin{equation}\label{det of SSigma rect}
    \det \SSigma_k = \det \B{k}^{-1} \prod_{l=2}^k \det (\B{l-1}^{-1} -\B{l}^{-1}) > 0
                    \;,\;\;k=2, \ldots, K.
\end{equation}
\end{lemma}

\begin{remark}
The condition \eqref{binar product} holds for rectangular kernels with nested supports.
\end{remark}

\begin{proof}
The condition \eqref{binar product} implies \( \W{l} \Sigma \W{m} = \diag(\w{l}{1} \w{m}{1} \sigma_1^{-2} , \ldots,  \w{l}{n}  \w{m}{n}  \sigma_n^{-2} ) = \W{l} \) for any \( l \le m \). Thus, the blocks of \( \SSigma_k \) simplify to \( D_l \Sigma D_m^{\T} =
\B{l}^{-1}  \PPsi  \W{l} \Sigma \W{m}  \PPsi^{\T}  \B{m}^{-1} =
\B{l}^{-1}  \PPsi  \W{l}   \PPsi^{\T}  \B{m}^{-1} \), and \( \SSigma_k  \) has a simple structure:
\[
\SSigma_k =
\begin{pmatrix}
    \B{1}^{-1} & \B{2}^{-1}  & \B{3}^{-1}  & \ldots &       \B{k}^{-1} \\
    \B{2}^{-1} & \B{2}^{-1}  & \B{3}^{-1}  & \ldots &       \B{k}^{-1} \\
    \vdots     &   \vdots    &   \vdots    &  \vdots &       \vdots   \\
    \B{k}^{-1} & \B{k}^{-1}  & \B{k}^{-1}  & \ldots &        \B{k}^{-1}
\end{pmatrix}.
\]
Then the determinant of \( \SSigma_k \) coincides with the determinant of the following irreducible block triangular matrix:
\[
\det \SSigma_k =
\begin{vmatrix}
    \B{1}^{-1}-\B{2}^{-1} &  \B{2}^{-1}-\B{3}^{-1}  & \ldots & \B{k-1}^{-1}-\B{k}^{-1} & \B{k}^{-1} \\
     \bb{0}               &  \B{2}^{-1}-\B{3}^{-1}  & \ldots & \B{k-1}^{-1}-\B{k}^{-1} & \B{k}^{-1} \\
     \vdots               &          \vdots         & \vdots &       \vdots            &   \vdots   \\
     \bb{0}               &   \bb{0}                & \ldots & \B{k-1}^{-1}-\B{k}^{-1} & \B{k}^{-1} \\
     \bb{0}               &   \bb{0}                & \bb{0} &      \bb{0}             & \B{k}^{-1}
\end{vmatrix},
\] implying
\begin{equation*}
    \det \SSigma_k = \det(\B{1}^{-1}-\B{2}^{-1}) \det(\B{2}^{-1}-\B{3}^{-1})\cdot \ldots \cdot
        \det(\B{k-1}^{-1}-\B{k}^{-1}) \det\B{k}^{-1}.
\end{equation*}
Clearly the matrix \( \SSigma_k \) is nonsingular if all the matrices \( \B{l-1}^{-1}-\B{l}^{-1} \) are nonsingular. By \( (\mathfrak{D}) \) and \( (\mathfrak{Loc}) \) \( \B{l} \succ 0 \) for any \( l \). By \( (\mathfrak{B}) \) there exists \( u_0> 1 \) such that \( \B{l}  \succeq  u_0  \B{l-1} \) therefore
\( \B{l-1}^{-1}-\B{l}^{-1} \succeq (1-1/u_0) \B{l-1}^{-1} \succ \B{l-1}^{-1} \succ 0\).
\end{proof}

\begin{lemma}\label{MGF for joint distribution}
Under the alternative the moment generation function (mgf) of the joint distribution of \( \mmle{1}, \ldots , \mmle{K} \) is
\begin{equation}\label{MGF under alternative}
    \EE \exp \big\{ \gamma^{\T} ( \vec \tilde \TTa_K - \vec \TTa^*_K) \big\}
        = \exp \bigg\{ \frac{1}{2} \gamma^{\T} \SSigma_{K,0} \, \gamma \bigg\}.
\end{equation}
Thus, provided that \( \SSigma_{K,0} \succ 0 \), it holds \( \vec \tilde \TTa_K \sim \norm{\vec \TTa^*_K}{\SSigma_{K,0}} \).
\par Similarly, under the null, if \( \SSigma_K \succ 0 \), the joint distribution of \( \vec \tilde \TTa_K \) is \(  \norm{\vec \TTa_K}{\SSigma_K} \) with mgf
\begin{equation}\label{MGF under null}
    \EE \exp \big\{ \gamma^{\T} ( \vec \tilde \TTa_K - \vec \TTa_K) \big\}
        = \exp \bigg\{ \frac{1}{2} \gamma^{\T} \SSigma_K \, \gamma \bigg\}.
\end{equation}
\end{lemma}
\begin{proof}
Let \( \gamma \in \R^{pK} \) be written in a partitioned form
\( \gamma^{\T} = (\gamma_1^{\T}, \ldots, \gamma_K^{\T}) \)
with subvectors \( \gamma_l \in \R^p \), \( l=1, \ldots, K \). Then the mgf for the centered random vector
\( \vec \tilde \TTa_K - \vec \TTa^*_K \in \R^{pK} \) due to the decomposition \eqref{linearity if quasiMLE} \( \mmle{l} = \bbpf{l} + D_l \Sigma_0^{1/2} \eeps\) with \( D_l = \B{l}^{-1}\PPsi \W{l}\)
can be represented as follows:
  \begin{eqnarray*}
    &&\EE \exp \big\{  \gamma^{\T} ( \vec \tilde \TTa_K - \vec \TTa^*_K) \big\}
        = \EE \exp\big\{ \sum_{l=1}^K \gamma_l^{\T} (\mmle{l} - \bbpf{l}) \big\} \\
    &=& \EE \exp\big\{ \sum_{l=1}^K \gamma_l^{\T} D_{l} \Sigma_0^{1/2} \eeps \big\}
        = \EE \exp\big\{ \big(\sum_{l=1}^K D_{l}^{\T} \gamma_l  \big)^{\T} \Sigma_0^{1/2} \eeps \big\}.
  \end{eqnarray*}
A trivial observation that \( \sum_{l=1}^K D_{l}^{\T} \gamma_l   \) is a vector in \( \R^n \) and
\( \Sigma_0^{1/2} \eeps \sim \norm{0}{\Sigma_0} \) by \eqref{true model} implies by definition of \( \SSigma_{K,0} \) the first assertion of the lemma, because
\begin{eqnarray*}
  &&\EE \exp\big\{ \big(\sum_{l=1}^K D_{l}^{\T} \gamma_l  \big)^{\T} \Sigma_0^{1/2} \eeps \big\}
     = \exp \bigg\{ \frac{1}{2} \big(\sum_{l=1}^K D_{l}^{\T} \gamma_l  \big)^{\T}
            \Sigma_0 \big(\sum_{l=1}^K D_{l}^{\T} \gamma_l  \big) \bigg\}  \\
  &=&\exp \bigg\{ \frac{1}{2}\big(\DD_K^{\T} \gamma \big)^{\T} (J_K \otimes \Sigma_0) \DD_K^{\T} \gamma \bigg\} = \exp \bigg\{ \frac{1}{2} \gamma^{\T} \SSigma_{K,0} \, \gamma \bigg\},
\end{eqnarray*}
where \( \DD_K \) is defined by \eqref{def DD_K}.
\end{proof}

\begin{lemma}\label{KL for joint distr}
The Kullback-Leibler divergence between the distributions of \( \vec \tilde \TTa_k \) under the alternative and under the null has the following form:
\begin{eqnarray}
  & &2\KL(\P_{\ff, \Sigma_0}^k,\P_{\tta, \Sigma}^k) \eqdef 2\EE_{\ff, \Sigma_0}
  \log\big( \frac{\dd \P_{\ff, \Sigma_0}^k}{\dd \P_{\tta, \Sigma}^k}  \big) \label{KL joint}\nn
  &=&  \Delta(k)  +\log\bigg(\frac{\det \SSigma_k}{\det \SSigma_{k,0}} \bigg)
            + \tr (\SSigma_k^{-1} \SSigma_{k,0}) -pk,
\end{eqnarray}
where
\begin{eqnarray}
  b(k) &\eqdef& \vec \TTa^*_k - \vec \TTa_k \label{def_b(k)}, \\
  \Delta(k) & \eqdef & b(k)^{\T} \SSigma_k^{-1} b(k)  \label{def_Delta(k)}.
\end{eqnarray}

\end{lemma}
\begin{proof} 

\par Denote the Radon-Nikodym derivative by \( Z_k \eqdef \dd \P_{\ff, \Sigma_0}^k / \dd \P_{\tta, \Sigma}^k \). Then
\begin{eqnarray}\label{logdp}
    \log\big(Z_k(y)\big) =
        \frac{1}{2} \log\bigg(\frac{\det \SSigma_k}{\det \SSigma_{k,0}} \bigg)
        &-&\frac{1}{2} \| \SSigma_{k,0}^{-1/2} (y - \vec \TTa^*_k) \|^2  \nn
        &+& \frac{1}{2}\| \SSigma_k^{-1/2} (y - \vec \TTa_k) \|^2
\end{eqnarray}
 can be considered as a quadratic function of \( \vec \TTa_k \). By the Taylor expansion at the point \( \vec \TTa^*_k \) the last expression reads as follows:
\begin{eqnarray*}
 & & \log\big(Z_k(y)\big) =
        \frac{1}{2} \log\bigg(\frac{\det \SSigma_k}{\det \SSigma_{k,0}} \bigg)
        -\frac{1}{2} \| \SSigma_{k,0}^{-1/2} (y - \vec \TTa^*_k) \|^2  \nn
&+&  \frac{1}{2} \| \SSigma_k^{-1/2} (y - \vec \TTa^*_k) \|^2
+ b(k)^{\T} \SSigma_k^{-1} (y- \vec \TTa^*_k) + \frac{1}{2} \Delta(k).
\end{eqnarray*}
Then the expression for the Kullback-Leibler divergence can be written in the following way:
\begin{eqnarray*}
  & &\KL(\P_{\ff, \Sigma_{0}}^k,\P_{\tta, \Sigma}^k) \eqdef \EE_{\ff, \Sigma_0}
  \log\big(Z_k\big) \\
  &=&  \frac{1}{2} \log\bigg(\frac{\det \SSigma_k}{\det \SSigma_{k,0}} \bigg)
       + \frac{1}{2} \Delta(k)
       + \frac{1}{2} \EE \big\{ \| \SSigma_k^{-1/2} \SSigma_{k,0}^{1/2} \xi \|^2 -\| \xi \|^2
       + 2 b(k)^{\T} \SSigma_k^{-1} \SSigma_{k,0}^{1/2} \xi  \big\},
\end{eqnarray*}
where \( \xi \sim \norm{0}{I_{pk}} \). This implies
\begin{equation}\label{KL general}
    2 \KL(\P_{\ff, \Sigma_0}^k,\P_{\tta, \Sigma}^k)
        = \Delta(k)
            +\log\bigg(\frac{\det \SSigma_k}{\det \SSigma_{k,0}} \bigg)
            + \tr (\SSigma_k^{-1} \SSigma_{k,0}) -pk.
\end{equation}
\par In the case of \emph{homogeneous errors} with \( \sigma_{0,i} = \sigma_{0} \) and
 \( \sigma_{i} = \sigma, i = 1, \ldots ,n \), the calculations simplify a lot. Now
 \begin{equation*}
    \SSigma_k = \sigma^2 \VV_k, \;\;\; \SSigma_{k,0} = \sigma^2_0 \VV_k
 \end{equation*}
with a \( pk \times pk  \) matrix \( \VV_k \) defined as
\begin{equation*}
     \VV_k = \big( \bar{D}_1 \oplus  \cdots \oplus \bar{D}_k  \big)
                    \big( J_k \otimes I_n \big)
                    \big( \bar{D}_1 \oplus  \cdots \oplus \bar{D}_k  \big)^{\T},
\end{equation*}
where \( \bar{D}_l = (\PPsi \cc{W}_l \PPsi^{\T})^{-1} \PPsi \cc{W}_l  \), \( l= 1, \ldots, k \), does not depend on \( \sigma \). Then \( \Delta(k) = \sigma^{-2} \Delta_1(k)  \), with \( \Delta_1(k) \eqdef b(k)^{\T} \VV_k^{-1} b(k) \), \( \det \SSigma_k / \det\SSigma_{k,0}  = (\sigma^2/\sigma_0^2)^{pk} \) and the expression for the Kullback-Leibler divergence reads as follows:
\begin{eqnarray}\label{KL hom}
    \KL(\P_{\ff, \Sigma_{0}}^k,\P_{\tta, \Sigma}^k)
        &=& p k \log\big( \frac{\sigma}{\sigma_0} \big)
            + \frac{1}{2} \Delta(k) + \frac{p k}{2} \big( \frac{\sigma_0^2}{\sigma^2} - 1 \big)\nn
        &=&  p k \log\big( \frac{\sigma}{\sigma_0} \big)
            + \frac{1}{2 \sigma^2} b(k)^{\T} \VV_k^{-1} b(k)
            + \frac{p k}{2} \big( \frac{\sigma_0^2}{\sigma^2} - 1 \big),
\end{eqnarray}
implying the same asymptotic behavior as in \eqref{bounds for KL}.

\end{proof}

\chapter{Dependence on the dimension for complexity of approximation \\of random fields}
\label{chap:Approx}

In this chapter we consider the \( \e \)-approximation by $n$-term partial sums of the Karhunen-Lo\`eve expansion of $d$-parametric random fields of tensor product-type in the average case setting. We investigate the behavior as $d\to \infty$ of the information complexity $n(\e,d)$ of approximation  with error not exceeding a given level $\e$. It was recently shown by Lifshits and Tulyakova \citeasnoun{LT} that for this problem one observes the curse of dimensionality (intractability) phenomenon. We present the exact asymptotic expression for the information complexity \( n(\eps,d) \).

\section{Introduction and set-up}

\par Suppose we have a random function $X(t)$, with $t$ in a compact parameter set $T$, admitting a series representation via random variables $\xi_k$ and the deterministic real functions $\varphi_k$,
namely, \[ X(t) = \sum_{k=1}^{\infty}\xi_k \varphi_k(t),\] where the series converges in the mean and a.s.\ for each $t \in T$.
A more precise description will be given later.
For any finite set of positive integers $\kk\subset \N$ let $X_\kk(t) = \sum_{k\in \kk} \xi_k \varphi_k(t)$. In many problems one needs to approximate $X$, for instance under the $L_2$-norm with a finite-rank process $X_\kk$. Natural questions arise: How large should $\kk$ be in order to yield a
given small approximation error? Given the size of $\kk$, which $\kk$ provides the smallest error?

\par In this chapter we address the first of these questions for a specific class of random functions,
namely {\it tensor product-type random fields} with high-dimensional parameter sets. The tensor product-type field is a separable zero-mean random function $X=\{ X(t)\}_{t\in T}$,
with a rectangular parameter set \( T \subset \RR^d \) and covariance function $\KK^{(d)}$
which can be decomposed into a product of equal ``marginal'' covariances depending on different arguments.
Namely, let $T=[0,1]^d$  and
\be \label{K} \KK^{(d)}(s,t) = \prod_{l=1}^d \KK_l(s_l,t_l) \ee
for all $s_l, t_l \in [0,1]$,
$s=(s_1, ... , s_d)$, $t = (t_1, ... , t_d)$. Obviously, the integral operator
with the kernel~\eqref{K} is the tensor product of the integral operators with the kernels
 $\KK_l(s_l,t_l)$.
\medskip

\par Let $\{\lambda_{i}\}_{i\ge 1}$ be a nonnegative sequence satisfying
\be  \label{l2}
  \sum_{i=1}^\infty \lambda_{i}^2 <\infty
\ee
and let $\{\varphi_i\}_{i>0}$ be an orthonormal basis in $L_2[0,1]$. Consider a family of tensor
product-type random fields \be
\label{defX}
\X=\left\{ X^{(d)}(t), t \in[0,1]^d \right\}\,,\;\; d=1,2,\ldots\, .\ee
According to the multiparametric Karhunen-Lo\`eve expansion (see~\cite{A} for details),
the~family~\eqref{defX} can be given by
\bea \label{xd1} X^{(d)}(t)
    &=& \sum_{\bk\in \N^d} \xi_{\bk} \prod_{l=1}^d \la_{k_l} \prod_{l=1}^d \varphi_{k_l}(t_l)\\
&=&  \sum_{k_1=1}^{\infty} \cdots \sum_{k_d=1}^{\infty}
\xi_{k_1,\ldots,k_d} \la_{k_1} \cdots \la_{k_d} \varphi_{k_1}(t_1)
\cdots \varphi_{k_d}(t_d) ,\;\; \nonumber \eea

\noindent where the series converges a.s.\ for every $t = (t_1, \ldots , t_d)\in [0,1]^d$.
The collection $\{\xi_{\bk}\}$ is an array of noncorrelated random variables with zero mean and
unit variance, and $\la_{k_l}^2$ and $\varphi_{k_l}$ are, respectively, the eigenvalues and eigenfunctions of the family of integral equations
\[ \la_{k_l}^2 \varphi_{k_l}(t_l) = \int_0^1 \KK_l(s_l,t_l)\varphi_{k_l}(s_l) \mathrm{d}s_l \;,\;\;\;\;\; t_l \in [0,1] \;,\;\;\;\;\; l=1,...,d, \] corresponding to the ``marginal'' covariance operators. Clearly, under assumption~(\ref{l2}) the sample paths of
$X^{(d)}$ belong to $L_2([0,1]^d)$ almost surely and the covariance
operator of $X^{(d)}$ has the system of eigenvalues
\be\label{sobd}
\lambda^2_{\bk}= \prod_{l=1}^{d} \la_{k_l}^2\,,\,\,\,\,  \bk \in \N^d.
\ee
\par As was mentioned in~\cite{Sab}, the Karhunen-Lo\`eve expansion or the proper orthogonal decomposition of random functions was introduced independently and almost simultaneously by Kosambi~\cite{Kos}, Lo\`eve~\cite{L}, Karhunen~\cite{Karh1},~\cite{Karh2}, Obukhov~\cite{Ob}, and Pougachev~\cite{Pug}.

\par In what follows we suppress the index $d$ and write $X(t)$
instead of $X^{(d)}(t)$. For any $n>0$, let $X_n$ be the partial sum of (\ref{xd1})
corresponding to $n$ maximal eigenvalues. We study the {\it average case error }
of approximation to $X$ by $X_n$ $$
       e(X, X_n; d) = \left(\EE ||X-X_n ||^2_{2}
       \right)^{1/2},
$$ as $d \to \infty$.

\par It is well known (see, for example,
\cite{BS}, \cite{KL} or \cite{R}) that $X_n$ provides the minimal average quadratic error among all linear approximations to $X$ having rank~$n$. Because we are going to explore a {\it family} of random functions, it is more natural to investigate {\it relative} errors, that is, to compare
the error size with the size of the function itself. Denote the ``marginal'' trace by
\[ \La\eqdef\sum_{i=1}^\infty \la_{i}^2. 
\]
Then
\[
\EE\| X\|^2_2 = \sum_{\bk\in\N^d} \la_{\bk}^2 = \La^d.
\]

The {\it average case information complexity} for the normalized error criterion reads as the minimal number of terms in $X_n$ (or, equivalently, of maximal eigenvalues, if they would be ordered) needed to approximate $X$ with the error not exceeding a given level $\e$:
\[
n(\e,d) \eqdef
 \min \Big\{n : \frac{e(X, X_n; d)}{\left(\EE\| X\|^2_2\right)^{1/2}} \leq \e\Big\}
= \min \{n : \EE \|X-X_n \|^2_2 \leq \e^2\La^d \}.
\]

The study of $n(\e,d)$  we are interested in here belongs to the
class of problems dealing with the dependence of the information complexity
for linear multivariate problems on the dimension, see the papers
of Wo\'{z}niakovski \cite{W92}, \cite{W94a}, \cite{W94b},
\cite{W2006} and the references therein.

Generally, the linear tensor problems with \( \la_2 >0 \) for the normalized error criterion are intractable, since
\begin{equation*}
    n(\e,d) \ge (1 - \eps^2) \big( 1+ \frac{\la_2}{\la_1} \big)^d \;\; \; \text{for all}\;\; \eps \in [0,1)
\end{equation*}
is exponential in \( d \) and the curse of dimensionality takes place, see Theorem~6.6 of~\citeasnoun{Novak and Wozniakowski}.
However, it is interesting to know the exact behavior of the information complexity \( n(\e,d) \) even
in this case, because this kind of negative result can help in lifting the curse of dimensionality.


It was suggested in \cite{LT} to use an auxiliary probabilistic
construction for studying the properties of the deterministic array of
eigenvalues (\ref{sobd}). We follow this approach.

Consider a sequence  of independent identically distributed random
variables $\left\{U_l\right\}, \,\,l=1,2,...$ with the common
distribution given by
\be\label{def_Ul}
\PP(U_l=-\log
\la_{i})=\frac{\la_{i}^2}{\La} \,,\,\,\,\,i=1,2,...
\ee

Under the assumption
\be\label{3d_moment}
\sum_{i=1}^{\infty}|\log \la_{i}|^3 \la_{i}^2 \;<\; \infty,
\ee
the condition $ \EE | U_l |^3 < \infty$ is obviously satisfied.

Let $M$ and $\s^2$ denote, respectively, the mean and the variance of $U_l$. Clearly,
\bea \nonumber M &=& - \sum_{i=1}^{\infty}\log
\la_{i}\,\frac{ \la_{i}^2}{\La},
\\ \nonumber
\s^2 &=&
\sum_{i=1}^{\infty}|\log \la_{i}|^2\, \frac{ \la_{i}^2}{\La}\; -\;
M^2.
\eea
Then the third central moment of $U_l$ is given by
\[
\al^3 \eqdef \EE(U_l -M)^3 =  - \sum_{i=1}^{\infty}\left(\log
\la_{i}\right)^3\,\frac{ \la_{i}^2}{\La} \; - \; 3M\s^2 \;-\;M^3.
\]
If \eqref{3d_moment} is verified, we have $|M|<\infty$, $0 \leq \s^2
< \infty$, and $|\al|< \infty.$

In what follows the explosion coefficient \be\label{Expl} \cc E \eqdef \La
e^{2M} \ee will play a significant role, because its contribution into the ``curse of dimensionality'' is the largest. It was shown in \cite{LT}
that by concavity of the logarithmic function $\cc E >1$, except for the totally degenerate case when the
number of strictly positive eigenvalues is zero or one. In other
words, $\cc E =1$ if and only if $\s=0$. Henceforth, we will exclude this degenerate case.

The following result was obtained in \cite{LT}, Theorem 3.2.

\begin{theorem}\label{T:LT_3.2}
Assume that the sequence \( \{\la_{i}\} \), \( i=1,2, \ldots \), satisfies the condition
\[  \sum_{i=1}^{\infty}|\log \la_{i}|^2\,
\la_{i}^2 < \infty.
\]
Then for every $\e \in (0,1)$ we have
\[
\lim_{d \to \infty} \frac{\log n(\e,d) - d \log \cc E}{\sqrt{d}} =
2q,
\]
where the quantile $q=q(\e)$ is chosen
from the equation
\be\label{q}
1 - \Phi\left( \frac{q}{\s}\right) = \e^2
\ee
with \( \Phi(\cdot)  \) denoting the standard normal distribution function.
\end{theorem}

The authors of \cite{LT} conjectured that under further assumptions on
the sequence $\{\la_{i}\}$ one can prove that
\[
n(\e,d) \approx \frac{C(\e) \cc E^d e^{2q\sqrt{d}}}{\sqrt{d}}\ ,\;
\;\, d \to \infty.
\]
We will show that even a stronger statement holds.

\section{Main result: the exact intractability rate in increasing dimension}

It turns out that two different cases depending on the nature
of the distribution of $U_l$ should be distinguished. The proof and the final result depend on whether  this distribution is a {\it lattice} distribution or not.

Recall that one calls a discrete distribution of a
random variable $U$ a lattice distribution, if there exist numbers $a$
and $h>0$ such that every possible value of $U$ can be
represented in the form $a+\nu h$, where $\nu$ is an integer.
The number $h$ is called the span of the distribution. In the following, when studying the
lattice case, we assume that $h$ is the maximal span of the
distribution; i.e., one cannot represent all possible values of
 $U_l$ in the form $b+\nu h_1$ for some $b$ and $h_1>h$.

Definition \eqref{def_Ul} yields that  the variables $U_l$ have a common lattice
distribution if and only if $\la_{i} = C e^{-n_i h}$ for some positive $C$, $h$
and $n_i \in \N$. We call this situation the {\it lattice case}
and will assume that $h$ is chosen as large as possible.
Otherwise we say that the {\it nonlattice case} holds.

\par By $f(d)=o(g(d))$ we mean that $\lim_{d \to \infty}\frac{f(d)}{g(d)} =0$. In particular, $ f(d)=g(d)\left( 1 + o(1) \right) $ means that $\lim_{d \to \infty}\frac{f(d)}{g(d)} =1 $.

\bigskip
\begin{theorem}\label{main result}
Let the sequence \( \left\{ \la_{i}\right\} \), \( i=1,2, \ldots  \),
satisfy~\eqref{3d_moment}. \par Then for every $\e \in (0,1)$ it holds
\[ n(\e,d) = K\ \phi(\frac{q}{\s})\  \cc E^d e^{2q \sqrt{d}}
\,d^{-1/2} \left( 1 + o(1) \right),\; \;\, d \to \infty,\] where
\beaa\phi(x) &=& \frac{1}{\sqrt{2 \pi}} \, e^{-x^2/2 },\\
\\
K &=& \begin{cases}
\frac{h}{\s(1-e^{-2h})} & \mathrm{in \;the\; lattice \;case,} \\
 \frac{1}{2\s}  &\mathrm{otherwise},
\end{cases}
\eeaa and the quantile $q=q(\e)$ is defined in~\eqref{q}.
\end{theorem}

\begin{remark}
 One can see that the complexity of approximation
increases exponentially as $d \to \infty$. This phenomenon is
referred to as the {\it curse of
dimensionality} or { \it intractability}; see, e.g.,~\cite{R}~and~\cite{W94a}. The notion of the ``curse of
dimensionality'' dates back at least to Bellman~\cite{B}.
\end{remark}
\begin{remark}
By l'H\^opital's rule,
\begin{equation*}
    \lim_{h \to 0} \frac{h}{\s \left(1-e^{-2h}\right)} = \frac{1}{2\s },
\end{equation*}
and thus the relations for \( K \) are in accordance as \( h \to 0 \).

\bigskip
\end{remark}

\section{Proof of the main result}
This section presents a proof of Theorem \ref{main result}.
\begin{proof}
Let $\zeta = \zeta(\e, d)$ be the maximal positive number such that
the sum of eigenvalues satisfies
\[
\sum_{\bk\in\N^d : \la_{\bk}< \zeta}\la_{\bk}^2   \leq \e^2\La^d.
\]
Define a lattice set in $\N^d$ in the following way:
\begin{eqnarray*}
\mathrm{A} =  \mathrm{A}(\e, d) &\eqdef&
        \left\{  \bk \in \N^d : \la_{\bk} \geq \zeta \right\} \\
        &=& \Big\{  \bk \in \N^d : \prod_{l=1}^{d} \la_{k_l} \geq \zeta \Big\}.
\end{eqnarray*}
Since $\la_{\bk} >0$ for any $\bk \in \mathrm{A}$, one can write
 \beaa \nonumber && n(\e,d) \; = \; \# \mathrm{A} \;
 =\;
\sum_{\bk \in \mathrm{A}}\frac{\la_{\bk}^2}{\la_{\bk}^2}
\\ \nonumber
&=& \sum_{ \bk \in
\N^d : -\sum \log \la_{k_l} \leq -\log \zeta}\La^d \exp\Big\{-2
\sum_{l=1}^d \log \la_{k_l}\Big\}\prod_{l=1}^{d}\PP(U_l=-\log
\la_{k_l})
\\ \nonumber
&=& \La^d \,\EE\exp \Big\{2 \sum_{l=1}^d U_l  \Big\}
\ind{\Big\{\sum_{l=1}^d U_l \leq -\log \zeta \Big\}}.
\eeaa
For  centered and normalized sums
\[ Z_d = \frac{\sum_{l=1}^d U_l
- d M }{\s\sqrt{d}}
\]
we have
\[ \Big\{\sum_{l=1}^d U_l \leq -\log
\zeta \Big\} = \left\{ Z_d \leq \theta \right\},
\]
where
\be\label{ta_def}
\theta = \theta(\e,d) = - \frac{\log\zeta + d M}{\s \sqrt{d}}.
\ee

We show now that $\theta$ has a useful probabilistic meaning in terms of $\{U_l\}$ and of their sums.
Applying Lemma~3.1 of \cite{LT} we have for any $d \in \N$ and $z \in \RR^1$
 \beaa
 \sum_{\bk\in \N^d :\la_{\bk} < z} \la_{\bk} ^2  & = & \La^d\;\PP \left( \sum_{l=1}^d U_l > -\log z
 \right)\\ &=&  \La^d\;\PP \left( Z_d > - \frac {\log z + d M}{\s
 \sqrt{d}}\right)\\ &=&  \La^d\;\PP \left( Z_d > \theta_z   \right),
 \eeaa
 where
$$ \theta_z = -\frac{\log z + d M}{\s \sqrt{d}}.
$$
Fix $\e \in (0,1)$. Observe that
$$   \sum_{\bk\in \N^d :\la_{\bk} < z} \la_{\bk} ^2 \; \leq \; \e^2 \La^d
$$
if and only if
$$  \PP \left( Z_d > \theta_z   \right) \; \leq \; \e^2 .
$$
Therefore,
$\theta =\theta(\e, d)$ defined by \eqref{ta_def} is the $(1-\e^2)$-quantile of the
distribution of $Z_d$, namely,
\begin{eqnarray*}
  \theta(\e, d) &=&\min\{\theta:\ \PP\left(Z_d>\theta\right)\le \e^2\}\\
                &=&\min\{\theta:\ \PP\left(Z_d\le \theta\right) >1- \e^2\}.
 \end{eqnarray*}
Let $q=q(\e)$ be the quantile of the normal distribution function chosen from~\eqref{q}.
 Then in view of the central limit theorem
\be\label{ta}
\theta(\e, d) \to \frac{q(\e)}{\s}\; ,
\;\; d \to \infty,
\ee
for any fixed $\e \in(0,1)$.


Now let us return to the information complexity. We
obtain \bea \nonumber n(\e,d) &=& \cc E^d \, \EE \exp \{ 2 \s
\sqrt{d} Z_d\} \ind{\{Z_d \leq \theta \}}
\\ \nonumber
&=& \cc E^d \, \exp \{ 2 \s \sqrt{d}\theta \}
\int_{-\infty}^{\theta} \exp \{ 2 \s \sqrt{d}(z-\theta) \} \,\mathrm{d}F_d(z) ,
\eea
where $F_d(z) = \PP(Z_d < z)$ and $\cc E$ is defined as in~\eqref{Expl}.

Denote
\[
 \Psi_d(z)\eqdef \exp \{ 2 \s \sqrt{d}(z-\theta) \}
\]
and integrate by parts the integral
\[
 \int_{-\infty}^{\theta} \Psi_d(z)  \,\mathrm{d}[F_d(z)-F_d(\theta)] =
 \int_{-\infty}^{\theta} [- F_d(z)+F_d(\theta)] \,\mathrm{d}\Psi_d(z).
\]

From now on we have to distinguish the lattice and nonlattice cases.
\vspace{0.3cm}

\subsection{Nonlattice case}
\par In the following part of the proof we will assume that the
distribution of $\left\{ U_l \right\}$ is not lattice. This is
true in the most interesting cases, such as the Brownian sheet (the Wiener-Chentsov random field),  the completely tucked Brownian sheet (the Brownian pillow), and the d-variate Hoeffding, Blum, Kiefer and Rosenblatt process (see Appendix \ref{sect:Appendix. Examples of random fields} for details).

\par In view of~\eqref{3d_moment} we are able to apply the Cram\'{e}r-Esseen Theorem  (cf.~\cite{GK}, section~42, Theorem~2; \cite{Pe2}, Chap. V, section 5.7, Theorem~5.21; \cite{Pe1}, Chap. VI, section~3, Theorem~4). It leads to
\bea \label{CE_th}
\nonumber &&\int_{-\infty}^{\theta} [-
F_d(z)+F_d(\theta)] \,\mathrm{d}\Psi_d(z)\nn
& =& \int_{-\infty}^{\theta} [-
\Phi(z)+\Phi(\theta)] \,\mathrm{d}\Psi_d(z)\\
&+& \frac{\al^3}{6\s^3 \sqrt{2\pi d}}\int_{-\infty}^{\theta} [(z^2-1)e^{-z^2/2} -
((\theta^2-1)e^{-\theta^2/2} ] \,\mathrm{d}\Psi_d(z) + o\left(
\frac{1}{\sqrt{d}}\right)
\nn
&=& \nonumber I_1 + I_2 -I_3  - I_4 + o\left( \frac{1}{\sqrt{d}}\right),
\eea where \( \Phi(\cdot) \) is the standard normal distribution function and
\beaa \nonumber I_1 &=& \int_{-\infty}^{\theta} [-
\Phi(z)+\Phi(\theta)] \,\mathrm{d}\Psi_d(z),
\\
I_2 &=& \frac{\al^3}{6\s^3 \sqrt{2\pi
d}}\int_{-\infty}^{\theta}  z^2 e^{-z^2/2}  \,\mathrm{d}\Psi_d(z),
\\ \nonumber
I_3 &=& \frac{\al^3}{6\s^3 \sqrt{2\pi d}}\int_{-\infty}^{\theta}
e^{-z^2/2} \,\mathrm{d}\Psi_d(z),
\\ \nonumber
I_4 &=& \frac{\al^3}{6\s^3
\sqrt{2\pi d}}\left(\theta^2 - 1 \right) e^{-\theta^2/2} \\
&=& \frac{\al^3}{6\s^3 \sqrt{2\pi d}}\left(\left(\frac{q}{\s}\right)^2
- 1 \right) \exp\left\{  -\frac{ q^2 }{ 2\s^2 } \right\}\left( 1 + o(1) \right).
\eeaa
(the last equivalence is provided by~\eqref{ta}).

Since $\mathrm{d}\Psi_d(z) = 2 \s \sqrt{d}\Psi_d(z)d z$,  the integral $I_2$ is
given, after a change of variable, by the following expression:
\beaa \nonumber
 I_2 & = & I_2(d,\theta) \\
 &=& \frac{\al^3}{3\s^2 \sqrt{2
 \pi d }}\int_0^{\infty}(\theta - \frac{y}{\sqrt{d}})^2 \exp\Big\{
 -\frac{1}{2}\Big(\theta - \frac{y}{\sqrt{d}}\Big)^2\Big\}\exp\{-2\s y \} \; \mathrm{d} y
 \eeaa
 with $y = -\sqrt{d}(z-\theta)$.

\par For any $ d =1,2,...$,
\[
0\le
\Big(\theta - \frac{y}{ \sqrt{d}}\Big)^2 \exp\Big\{
-\frac{1}{2}\Big(\theta - \frac{y}{ \sqrt{d}}\Big)^2 \Big\}
\leq (|\theta| + y)^2.
\]
This estimate gives us the majorant required in the Lebesgue dominated convergence theorem. Using~\eqref{ta} and passing to the limit in the integral, we obtain, as $d \to \infty$,
\[
I_2(d,\theta) =
 \frac{\al^3}{6\s^3 \sqrt{2
 \pi d }} \;\left(\frac{q}{\s}\right)^2 \exp\left\{  -\frac{ q^2 }{ 2\s^2 } \right\} \left( 1 +
 o(1)\right).
\]
Similarly,
\[
 I_3(d,\theta) =
 \frac{\al^3}{6\s^3 \sqrt{2
 \pi d }}  \exp\left\{  -\frac{ q^2 }{ 2\s^2 } \right\} \left( 1 + o(1)\right).
\]
Thus we obtain that $ \sqrt{d}I_4 =  \sqrt{d}(I_2 -I_3 )\left( 1 + o(1)\right)$, and hence,
 $I_2-I_3-I_4 =  o\left( \frac{1}{\sqrt{d}}\right)$.
\bigskip
\par Consider the main integral $I_1$:
\bea \label{I1}
\nonumber I_1 &=& I_1(d, \theta) = \int_{-\infty}^\theta
[-\Phi(z)+\Phi(\theta)] \,\mathrm{d}\Psi_d(z)\\
\nonumber &=&\frac{1}{\sqrt{2 \pi}} \int_{-\infty}^\theta  \exp\{ 2 \s \sqrt{d}(z-\theta) \} \exp\{ -z^2/2\}  \; \mathrm{d} z \\
\nonumber &=&\frac{1}{\sqrt{2 \pi d }}\int_0^{\infty}\exp\Big\{
 -\frac{1}{2}\Big(\theta - \frac{y}{\sqrt{d}}\Big)^2\Big\}\exp\{-2\s y \} \; \mathrm{d} y
\\ &=&  \frac{1}{2\s \sqrt{2 \pi d}} \exp\left\{  -\frac{ q^2 }{ 2\s^2 } \right\}\left( 1 + o(1) \right)\;,\;\; d \to \infty. \eea

 Then
\[
n(\e,d) = \frac{\cc E^d\, \exp\{ 2q \sqrt{d} \}}{2 \s \sqrt{d}}
\frac{1}{\sqrt{2 \pi}} \exp\left\{  -\frac{ q^2 }{ 2\s^2 } \right\} \left( 1 + o(1)
\right)
\]
as asserted. \vspace{0.3cm}

\subsection{Lattice case}
\par Now we will proceed under the assumption that the random variables $U_l$
have a lattice distribution. Let possible values of the random variable $U_l$ be
\( \tilde{a} + \nu h , \; \nu = 0,\pm 1, \pm 2,..., \) where
$\tilde{a} = M+a$ is a shift and  $h$ is the maximal span of
the distribution. Therefore, all possible values of $Z_d$ have the
form
\[
\frac{d a + \nu h}{\s \sqrt{d}}, \; \nu = 0,\pm 1, \pm 2,... .
\]
Introduce the function \[ S(x) = [x] - x+ \frac{1}{2}, \] where
$[x]$ denotes, as usual, the integer part of $x$, and consider
\[
S_d(x) = \frac{h}{\s} \,S\left( \frac{x \s \sqrt{d} - d
a}{h}\right).
\]
Let $F_d(z)$ be as above. Then under assumption \eqref{3d_moment}
Esseen's result (see Theorem~1~page~43 in~\cite{GK}) yields
\[
F_d(z) - \Phi(z) = \frac{e^{-z^2/2}}{\sqrt{2\pi}} \left(
\frac{S_d(z)}{\sqrt{d}} - \frac{\al^3 (z^2 - 1)}{6 \s^3
\sqrt{d}}\right) + o \left( \frac{1}{\sqrt{d}}\right)
\]
uniformly in $z$.

Comparing with \eqref{CE_th}, we observe that one needs only to
evaluate the additional term \beaa J &=&\frac{1}{\sqrt{2 \pi d}}
\int_{-\infty}^{\theta}[- S_d(z) e^{-z^2/2} + S_d(\theta) e^{-\theta^2/2}]
\mathrm{d} \Psi_d(z)
\\
&=& \frac{1}{\sqrt{2 \pi d}} \int_{-\infty}^{\theta} \Psi_d(z)
\mathrm{d} \left( S_d(z) e^{-z^2/2}\right) =  J_1 - J_2 + J_3 ,
\eeaa where \beaa \nonumber J_1 &=& \frac{1}{\sqrt{2 \pi d}}
\int_{-\infty}^{\theta} \Psi_d(z) S_d'(z)e^{-z^2/2} \mathrm{d} z,
\\ \nonumber J_2 &=& \frac{1}{\sqrt{2 \pi d}}
\int_{-\infty}^{\theta} \Psi_d(z) S_d(z) z e^{-z^2/2} \mathrm{d} z,
\eeaa  and $J_3$ is a ``discrete part'', which  is defined in the
following way. Notice that $S(x)$ is a periodic function with period one;
therefore $S_d(x)$ possesses the period $h/\s \sqrt{d}$ and has
jumps at points $\{ \frac{kh + da}{\s \sqrt{d}}, k \in \mathbb{Z} \}$. If
the point $\theta$ belongs to this lattice, then there exists an integer $k'$
such that $\theta = \frac{k'h + da}{\s \sqrt{d}}$. Hence, one can
integrate the discontinuous  part of the integral $J$ with respect
to the measure $\frac{h}{\s}\delta_{\frac{kh + da}{\s \sqrt{d}}}$
and obtain
$$ J_3 = \frac{1}{\sqrt{2 \pi d}}\, \frac{h}{\s} \sum_{k=-\infty}^{k'}
\Psi_d\Big(\frac{kh + da}{\s \sqrt{d}}\Big)
\exp\Big\{-\frac{1}{2}\Big(\frac{kh + da}{\s \sqrt{d}}\Big)^2\Big\}.$$
\bigskip
\par We start by estimating $J_1$. At the points where the
derivative $S_d'(z)$ exists, one can easily calculate that
\begin{equation*}
    S_d'(z) = \frac{h}{\s}S\Big(\frac{z\s\sqrt{d} - da}{h} \Big) =
-\sqrt{d}.
\end{equation*}
Therefore, as in the nonlattice case, by
the Lebesgue dominated convergence theorem we have \bea \label{J1}\nonumber J_1 &=&
\frac{-\sqrt{d}}{\sqrt{2 \pi d}} \int_{-\infty}^{\theta} \exp\{ 2\s
\sqrt{d}(z-\theta)\} \exp\{ -z^2/2\} \mathrm{d} z \\\nonumber &=&
\frac{-1}{\sqrt{2 \pi d}} \int_{0}^{\infty} \exp
\Big\{-\frac{1}{2}\Big(\theta
- \frac{y}{\sqrt{d}}\Big)^2 \Big\}\exp\{-2\s y \} \mathrm{d} y\\
&=& \frac{-1}{ 2\s \sqrt{2 \pi d} }\exp\Big\{  -\frac{ q^2 }{ 2\s^2 } \Big\}\left( 1 + o(1) \right)\;,\; d \to
\infty , \eea
which yields $\sqrt{d}J_1 = - \sqrt{d} I_1\left( 1+ o(1)\right)$.
\bigskip
\par As for the integral $J_2$, this one, for sufficiently large $d$, becomes negligible.
Indeed, \beaa J_2 &=& \frac{1}{\sqrt{2 \pi d}}
\int_{-\infty}^{\theta}
\exp\{ 2\s \sqrt{d}(z-\theta)\}S_d(z) z  \exp\{ -z^2/2\} \mathrm{d} z \\
 &=&       \frac{1}{\sqrt{2 \pi d}}  \frac{1}{\sqrt{d}} \int_{0}^{\infty}
    \exp \Big\{-\frac{1}{2}\Big(\theta
- \frac{y}{\sqrt{d}}\Big)^2 \Big\} (\theta - \frac{y}{\sqrt{d}})S_d\Big(\theta
- \frac{y}{\sqrt{d}}\Big)\exp\{-2\s y \} \mathrm{d}y \\
&\leq& \frac{3 h}{2 \s d\sqrt{2 \pi } }\int_{0}^{\infty} \exp
\Big\{-\frac{1}{2}\Big(\theta - \frac{y}{\sqrt{d}}\Big)^2 \Big\} \Big(\theta -
\frac{y}{\sqrt{d}}\Big)\exp\{-2\s y \} \mathrm{d}y
\\ &=& \frac{3h}{ 4\s^2 d \sqrt{2 \pi }  }\left( \frac{q}{\s}\right)^2 \exp\left\{  -\frac{ q^2 }{ 2\s^2 } \right\}\left( 1 + o(1) \right)\;,\; d \to \infty
. \eeaa And, of course,  $J_2 = o\left(
\frac{1}{\sqrt{d}}\right)$.

\par Now we consider the essential summand \bea \label{J3} \nonumber J_3 &=&
\frac{1}{\sqrt{2 \pi d}}\, \frac{h}{\s} \sum_{k=-\infty}^{k'}
\exp\Big\{2\s \sqrt{d} \Big( \frac{k h + d a }{\s \sqrt{d}} - \theta
\Big) \Big\} \exp\Big\{- \frac{1}{2} \Big(
 \frac{k h + d a}{\s \sqrt{d}} \Big)^2\Big\}
\\\nonumber&=& \frac{1}{\sqrt{2 \pi d}}\,
\frac{h}{\s} \sum_{k=-\infty}^{k'}\exp\{2 h (k - k') \}
\exp\Big\{-\frac{1}{2} \Big( \frac{k h + d a}{\s \sqrt{d}} \Big)^2\Big\}
 \\ \nonumber&=& \frac{1}{\sqrt{2 \pi d}}\,
\frac{h}{\s} \sum_{l=0}^{\infty}  \exp\{-2 h l\}
\exp\Big\{-\frac{1}{2} \Big( \frac{(k' - l) h + d a}{\s \sqrt{d}} \Big)^2\Big\}
 \\\nonumber &=& \frac{1}{\sqrt{2 \pi d}}\,
\frac{h}{\s} \sum_{l=0}^{\infty}  \exp\{-2 h l\}
\exp\Big\{-\frac{1}{2} \Big( \theta - \frac{l h}{\s \sqrt{d}}\Big)^2\Big\}
\\ &=& \frac{1}{\s
\sqrt{ d}}\, \frac{h }{(1-e^{-2h})} \frac{1}{\sqrt{2
\pi}}\exp\left\{  -\frac{ q^2 }{ 2\s^2 } \right\}\left( 1 + o(1) \right)\;,\; d \to \infty. \eea
\bigskip
We obtained
$$\sqrt{d}J_3 = \sqrt{d} \frac{2h }{(1-e^{-2h})} I_1 \left( 1 + o(1)\right).$$
Putting together \eqref{I1}, \eqref{J1}, and \eqref{J3}, we
get  \[ n(\e,d) = \frac{\cc E^d \, e^{ 2q \sqrt{d} }}{ \s
\sqrt{d}} \frac{h }{(1-e^{-2h})} \frac{1}{\sqrt{2 \pi}} \,
\exp \left\{  -\frac{ q^2 }{ 2\s^2 } \right\} \left( 1 + o(1) \right),\; \;\, d \to \infty. \]
\end{proof}
\section{Appendix. Examples of tensor product-type random fields}
\label{sect:Appendix. Examples of random fields}
This section contains some examples of random fields to which the above general result can be applied.
\subsection{Wiener-Chentsov random field}
The Wiener-Chentsov field or the Brownian sheet (see \cite{Lft}) is a zero-mean Gaussian random function $W^{(d)}$ with covariance function equal to a product of the covariance functions corresponding to the Wiener process  $W$:
\[ \KK_{W^{(d)}} (s,t) = \prod_{l=1}^d \min\{s_l, t_l\} ,\;s=(s_1,...,s_d),\;t=(t_1,...,t_d) \in T.   \]
Therefore the marginal eigenvalues have the following form:
 \[\la_{W;i}^2 = (\pi (i-1/2))^{-2},\;i=1,2,\ldots \,.\]

\subsection{Completely tucked Brownian sheet}
The completely tucked Brownian sheet (the Brownian pillow) is a zero-mean Gaussian random function $B^{(2)}$ with covariance function equal to a product of the covariance functions corresponding to the standard Brownian bridge $B(t) = W(t) - t W(1)$, namely
\[ \KK_{B^{(2)}} (s,t) = \prod_{l=1}^2\left( \min\{s_l, t_l\} - s_l t_l  \right) ,\;s,t \in [0,1]^2.   \]
Respectively, the marginal eigenvalues (see~\cite{AnDar}) are equal to \[\la_{B;\,i}^2 = (\pi
i)^{-2},\;i=1,2,\ldots \,.\]
\par  In the literature different terms are in use for this random field. In~\cite{vanderVaartWellner} the term ``completely
tucked Brownian sheet'' is used; in~\cite{CsHor} ``tied-down Kiefer process'' is used; in~\cite{KonProt} this field is called ``the Brownian pillow''.

The notion of ``completely tucked Brownian sheet'' and its generalization for the case $d>2$ was introduced by Blum, Kiefer, and Rosenblatt~\cite{BKRos} as the limit distribution for a functional of
an empirical process occurring in nonparametric testing of independency, the so-called ``independence
empirical process''~(see~\cite{vanderVaartWellner}). Therefore, the $d$-parametric generalization of the completely tucked Brownian sheet is often referred to as the ``d-variate Hoeffding, Blum, Kiefer, and Rosenblatt process'' (see, for example,~\cite{KonProt}).
The mention of Hoeffding's name in the term is motivated by the fact that the test studied in~\cite{BKRos} is equivalent to the one suggested earlier by Hoeffding in~\cite{Hoef}. However, the limiting distribution, the covariance function, the eigenvalues and the eigenfunctions of the respective integral equation were obtained in~\cite{BKRos}. Higher-dimensional generalizations were later treated in~\cite{Dugue} and~\cite{Deh}.

\subsection{Centered Gaussian processes}
In some statistical problems it is convenient to use centered empirical processes and corresponding limiting Gaussian processes.

\par For any Gaussian process $X = \{ X(t)\}$, $t
\in[0,1]$ we define the centered process
\[ \mathring{X}(t) \eqdef X(t)-\int_0^1 X(u) \mathrm{d}u.  \]
The centered Brownian bridge $\mathring{B}$, also referred to in the literature as the Watson process, was introduced in~\cite{Wats} for nonparametric goodness-of-fit testing on a circle. Watson showed that the covariance function is given by
\[ \KK_{\mathring{B}}(s,t) =  \min\{s,t\} - st + \frac{1}{2}(s^2 + t^2-s-t)+\frac{1}{12}     \;,\;\;s,\,t \in[0,1],\]
and the covariance operator with this kernel has a double spectrum, i.e.,
\[ \la_{\mathring{B};2i}^2 = \la_{\mathring{B};(2i-1)}^2 =(2\pi i)^{-2},\;i=1,2,\ldots \, .\]

\par The covariance function of the centered Wiener process $\mathring{W}$ has the form
\[ \KK_{\mathring{W}}(s,t) =  \min\{s,t\}  +\frac{1}{2}(s^2 +t^2)-s-t +\frac{1}{3}   \;,\;\;s,\,t \in[0,1],\]
and the corresponding eigenvalues coincide with those of the standard Brownian bridge, i.e.,
\[ \la_{\mathring{W};\,i}^2  = \la_{B;\,i}^2 = (\pi i)^{-2},\;i=1,2,\ldots. \]
 This is in accordance with the well-known equality in distribution for the $L_2$-norms of the Brownian bridge and the centered Wiener process; see~\cite{BegNOrs}.

\par The centered integrated Brownian bridge
\[ \breve{B}(t) = \Bar{B}(t) - \int_0^1\Bar{B}(u)\mathrm{d}u,\] where
\[ \Bar{B}(t) = \int_0^t B(u)\mathrm{d}u ,\,t \in[0,1]\] was considered in a framework of goodness-of-fit testing and small deviation probabilities under the $L_2$-norm in~\cite{HenzeN} and \cite{BegNOrs}, where
its covariance function
\[ \KK_{\breve{B}}(s,t) =  \frac{st\min\{s,t\}}{2} - \frac{\min\{s,t\}^3}{6} -
       \frac{(st)^2}{4} - \frac{s^2+t^2}{6} - \frac{s^4+t^4}{24} + \frac{s^3 +t^3}{6}
       + \frac{1}{45}, \]
 \noindent $s,\,t \in[0,1] $, and eigenvalues \[ \la_{\breve{B};\,i}^2 = (\pi i)^{-4}
 ,\;i=1,2,\ldots\,,
 \]
\noindent were obtained.

\subsection{Multivariate Anderson-Darling processes}

\par The tensor product of Anderson-Darling processes $A^{(d)}(t)$, $t~\in~[0,1]^d$,
is a zero-mean Gaussian random function $A^{(d)}(t)$, $t~\in~[0,1]^d$ with covariance function
\[ \KK_{A^{(d)}}(s,t) = \prod_{l=1}^d
    \frac{\min\{s_l, t_l\} - s_l t_l}{ \sqrt{s_l (1 - s_l)} \sqrt{t_l (1 - t_l)} }
        ,\;s_l, t_l \in [0,1].   \]
\noindent The eigenvalues of the corresponding covariance operator are given by
\[ \la_{\bk}^2 = \prod_{l=1}^d \frac{1}{k_l (k_l +1) } , \; \bk=(k_1, \ldots , k_d)  \in \N^d. \]

 \par In the one-dimensional case the Anderson-Darling process coincides in distribution with $\frac{B(t)}{\sqrt{t(1-t)}}$ , $t \in[0,1]$, and was introduced in~\cite{AnDar} in the context of goodness-of-fit testing. Anderson and Darling obtained its covariance function and the exact spectrum.

\par In~\cite{Pycke} another multivariate extension of the Anderson-Darling process, defined as a zero-mean Gaussian process with the covariance function
\[ \KK_A^{\mu}(s,t) = \left( \frac{\min\{s, t\} - s t}{ \sqrt{s (1 - s)} \sqrt{t (1 - t)}
}\right)^{\mu}
        ,\;s,\, t \in [0,1] ,\, \mu >0,     \]is given.

\noindent The eigenvalues of its covariance operator are of the form
\[ \la_{\mu,j}^{2} =\frac{\mu}{(\mu +j-1)(\mu + j)}\,, \; j=1,2, \ldots \,.  \]
When the parameter $\mu$ is positive integer, the random field, defined in such a way (more precisely, the square of its $L_2$-norm), is the limit distribution
for Cram\'er--von Mises-type statistics.

\addcontentsline{toc}{chapter}{Bibliography}

\Huge \textbf{Index of notation}
\vspace{10pt}
 \par \normalsize
\begin{align}
&\lfloor x \rfloor    &     &\text{greatest integer strictly less than the real number \( x \)}   \nn
& [x]                     &     & \text{integer part of \( x \)}    \nn
&  \log                   &     &\text{natural logarithm} \nn
&  \eqdef                 &     & \text{equals by definition} \nn
& \text{w.r.t.}   &     & \text{with respect to}    \nn
& \Sigma(\beta,L)       &     & \text{H\"older class of functions}     \nn
&     &    & \nn
&\textbf{Sets}&     &  \nn
& \emptyset    &     & \text{the empty set}     \nn
&  \#\{ \cdot \}            &     & \text{cardinality of the set \( \{ \cdot \} \) } \nn
&  A \cap B  &     & \text{intersection, \( \{ x: x\in A \; \text{and} \; x \in B  \} \)}     \nn
&    &     &     \nn
&\textbf{Special functions}    &     &     \nn
& \Gamma(\cdot)   &     & \text{the \( \Gamma \)-function}     \nn
& \Phi(\cdot)   &     & \text{the standard normal distribution function} \nn
&    &     &     \nn
&\textbf{Landau notation}&   & \nn
& f(x)=o(g(x)), \, x \to x_0 &  & \text{means that \( \lim_{x \to x_0}f(x)/g(x) =0 \)} \nn
& f(x)=\cc O (g(x)), \, x \to x_0   &     & \text{means that \( |f(x)|\le C|g(x)| \),
                                        as \( x \to x_0 \) } \notag
\end{align}
\noindent
\textbf{Linear algebra}
\begin{align}
& \gamma^{\T} \,,\; A^{\T}&     & \text{transpose of the vector \( \gamma \)
                                            or of the matrix \( A \)}\nn 
& \la_j(A)                &     & \text{\( j \)th eigenvalue of \( A \)}    \nn
& \la_1(A) \,,\; \la_{max}(A)&  & \text{largest eigenvalue of the symmetric matrix \( A \)}     \nn
& \tr(A)                  &     & \text{trace of \( A \), the sum of the diagonal elements
                                    of square matrix \( A \)}    \nn
& \rank(A)                &     & \text{rank of \( A \)}     \nn
&  \dim \cc U             &     & \text{dimension of the vector space \( \cc U \)}     \nn
&  \cc{C}(A)              &    & \text{column space of \( A \), the space spanned
                                            by the columns of \( A \)}    \nn
& \| \gamma \|            &    & \text{\( L_2  \) vector norm, Euclidean norm}    \nn
& \| A \|                 &     & \text{induced matrix norm based on \( L_2  \)
                                            vector norm (p. 31)}    \nn
& A \preceq B             &     & \text{\( B-A \succeq 0   \), L\"{o}wner partial ordering (p. 32)}    \nn
& A \succ 0               &     & \text{\( A \) is positive definite,
                                \( \gamma^{\T} A \gamma > 0 \) for \( x \ne 0 \) } \nn
& A \succeq 0             &     & \text{\( A \) is nonnegative definite,
                                        \( \gamma^{\T} A \gamma \ge 0 \) }     \nn
& A^{-1}                  &     & \text{inverse of \( A \) when \( A \) is nonsingular}    \nn
& \det A                  &     & \text{determinant of a square matrix \( A \) }    \nn
& A \otimes B             &     & \text{Kronecker product of \( A \) and \( B \) (p. 40) }    \nn
& \diag(x_1, \ldots, x_n) &     & \text{\(n \times n\) matrix with diagonal elements
                                        \(x_1, \ldots, x_n\) }    \nn
&                         &     & \text{and zeros elsewhere} \nn
& \vec A,                 &     & \text{if \( A \) is an \( m \times n \) matrix, then \( \vec A \) is
                                            an \( mn \times 1 \) vector  }    \nn
&                     &     & \text{formed by writing the columns of \( A \) one below the other} \nn
& \kappa(A)           &     & \text{ \( \kappa_2(A) \) conditional number of the positive
                                    definite matrix \( A \), }    \nn
&                     &     & \text {\( \kappa(A) \eqdef \la_{max}(A)/ \la_{min} (A) \)}\notag
\end{align}
\textbf{Probability and statistics}
\begin{align}
& \delta_x   &     & \text{Dirac measure on \( x \)}    \nn
& \eqdistr              &     & \text{equality in distribution}     \nn
& \text{a.s.}   &     & \text{almost surely}     \nn
&  \I\{ \cdot \}          &     &\text{indicator of the set \( \{ \cdot \} \) } \nn
& \norm{0}{1}   &     & \text{the standard normal distribution}    \nn
& \phi(\cdot)   &     & \text{density of the distribution \( \norm{0}{1} \)}    \nn
& \norm{0}{I_n}         &     & \text{standard normal distribution in \( \RRn \)}  \nn
& \norm{\tta}{\Sigma}   &     & \text{normal distribution with mean \( \tta \)
                                        and covariance matrix \( \Sigma \)}    \nn
& \tilde\tta=\argmax_{\tta \in \Theta}\LL(\tta)&   &\text{means that \( \LL(\tilde\tta)= \max_{\tta \in \Theta}\LL(\tta)\) } \nn
& \MSE                  &     & \text{mean squared risk at a point}    \nn
& \KL(P,P_{\tta})       &     & \text{Kullback-Leibler divergence between the measures \( P \)
                                        and \( P_{\tta} \) (p. 23) }\notag
\end{align}
\textbf{Assumptions}
\begin{align}
&\mathfrak{(Lp1)} - \mathfrak{(Lp4)} &   &  \text{p. 10}   \nn
& \mathfrak{(D)}   &     & \text{p. 22}   \nn
& \mathfrak{(Loc)} &     & \text{p. 23}   \nn
&  (\cc W)         &     & \text{p. 24}   \nn
&\text{Propagation conditions} (PC)& &\text{p. 28}   \nn
&  \mathfrak{(S)}  &     & \text{p. 29}    \nn
&  \mathfrak{(S1)} &     &  \text{p. 49}   \nn
&(SMB)             &     &  \text{p. 42}   \nn
&(SMBj)            &     & \text{p. 54}    \nn
& \mathfrak{(Lp1')} -\mathfrak{(Lp4')}  &     &  \text{p. 59--60}   \nn
& \mathfrak{(Lp1^{\mathrm{d} })}        &     & \text{p. 49}    \notag
\end{align}

\Huge \textbf{Erkl\"arung}
\vspace{20pt}

\par \normalsize

Ich erkl\"are, dass ich die dem angestrebten Verfahren zugrunde liegende Promotionsordnung (Amtliches Mitteilungsblatt Nr. 34/2006) kenne.

Ich erkl\"are, dass ich vorliegende Arbeit selbst\"andig und nur unter Verwendung der angegebenen Literatur und Hilfsmittel angefertigt habe.

Ich habe mich anderw\"arts noch nicht um einen Doktorgrad beworben, und ich besitze keinen Doktorgrad in dem Promotionsfach.


\vspace{40pt}

\par Nora Serdyukova

\par Berlin, 24. May 2010.

\end{document}